\documentclass[12pt,reqno]{amsart}
\usepackage{amssymb}
\usepackage{cases}
\usepackage{bbm}
\usepackage{mathrsfs}
\usepackage{graphicx}
\usepackage{amsfonts}
\usepackage{enumerate}
\usepackage{stmaryrd}

\usepackage{amssymb, amsmath, amsfonts, latexsym}
\usepackage{enumerate}
\usepackage[notref,notcite]{showkeys}

\allowdisplaybreaks[4]

\newtheorem{thm}{Theorem}[section]
\newtheorem{corollary}[thm]{Corollary}
\newtheorem{lemma}[thm]{Lemma}
\newtheorem{proposition}[thm]{Proposition}
\newtheorem{example}[thm]{Example}
\newtheorem{remarks}[thm]{Remark}

\newtheorem{hyp}[thm]{Hypothesis}

\theoremstyle{definition}
\newtheorem{defn}{Definition}[section]
 \theoremstyle{remark}

\numberwithin{equation}{section}

\newcommand{\ee}{\mathbb{E}}

\newcommand{\nn}{\mathbb{N}}
\newcommand{\rr}{\mathbb{R}}
\newcommand{\pp}{\mathbb{P}}

\newcommand{\e}{\varepsilon}

\def\AA{\mathcal A}
\def\BB{\mathcal B}

\def\DD{\mathcal D}
\def\FF{\mathcal F}
\def\EE{\mathcal E}

\def\LL{\mathcal L}

\def\PP{\mathcal P}
\def\SS{\mathcal S}

\def\d{{\rm d}}

\def\AA{\mathcal A}
\def\BB{\mathcal B}

\def\DD{\mathcal D}
\def\FF{\mathcal F}
\def\EE{\mathcal E}
\def\GG{\mathcal G}
\def\HH{\mathcal H}

\def\LL{\mathcal L}

\def\PP{\mathcal P}

\def\SS{\mathcal S}

\def\<{\langle}
\def\>{\rangle}

\def\beq{\begin{equation}}
\def\nneq{\end{equation}}

\def\bdef{\begin{defn}}
\def\ndef{\end{defn}}

\def\bthm{\begin{thm}}
\def\nthm{\end{thm}}

\def\bprop{\begin{proposition}}
\def\nprop{\end{proposition}}

\def\brmk{\begin{remarks}}
\def\nrmk{\end{remarks}}

\def\bexa{\begin{exa}}
\def\nexa{\end{exa}}

\def\blem{\begin{lemma}}
\def\nlem{\end{lemma}}

\def\bcor{\begin{corollary}}
\def\ncor{\end{corollary}}

\def\benu{\begin{enumerate}}
\def\nenu{\end{enumerate}}

\def\bhyp{\begin{hyp}}

\def\nhyp{\end{hyp}}

\def\<{\langle}
\def\>{\rangle}

\date{}

\def\bexe{\begin{exe}}
\def\nexe{\end{exe}}

\def\bprf{\begin{proof}}
\def\nprf{\end{proof}}

\def\bdes{\begin{description}}
\def\ndes{\end{description}}

\title[LDP and MDP for the Multivalued McKean-Vlasov SDEs with jumps]{Large and Moderate deviation principles for the Multivalued McKean-Vlasov SDEs with jumps}

\thanks{Corresponding authors: Wei Liu and Fengwu Zhu}
\author{Lingyan Cheng}
\address{Lingyan Cheng. School of Mathematics and Statistics, Nanjing University of Science and Technology,
Nanjing 210094, Jiangsu, PR China.}
\email{cly@njust.edu.cn}

\author{Caihong Gu}
\address{Caihong Gu. School of Mathematics and Statistics, Wuhan University, Wuhan 430072, Hubei, PR China.}
\email{gucaihong@whu.edu.cn}

\author{Wei Liu}
\address{Wei Liu\\School of Mathematics and Statistics, Wuhan University, Wuhan 430072, Hubei, PR China.}
\email{wliu.math@whu.edu.cn}

\author{Fengwu Zhu}
\address{Fengwu Zhu. School of Mathematics and Statistics, Wuhan University, Wuhan 430072, Hubei, PR China.}
\email{fwzhu\_math@whu.edu.cn}

\date{}

\begin{document}
\maketitle

\noindent {\bf Abstract:}
By using the weak convergence method, we establish the large and moderate deviation principles for the multivalued McKean-Vlasov SDEs with non-Lipschitz coefficients driven by L\'{e}vy noise in this paper. The Bihari's inequality is used to overcome the challenges arising from the non-Lipschitz conditions on the coefficients.
 \vskip0.3cm

 \vskip0.3cm

%  \subjclass is required.
\vskip0.3cm

\section{Introduction}

\noindent
Consider the following multivalued McKean-Vlasov stochastic differential equations (MMVSDEs for short) driven by L\'{e}vy noise:
\beq\label{eq1}
\begin{cases}
\d X_t\in -A(X_t) \d t +b(X_t,\LL_{X_t})\d t+\sigma (X_t,\LL_{X_t})\d W_t\\
\quad \quad \quad+ \int_{Z} G(X_{t-},\LL_{X_t},z)\tilde{N} (\d z,\d t),\ t\in[0,T], \\
X_0=x_0 \in \overline{D(A)},
\end{cases}
\nneq
where $A$ is a multivalued maximal monotone operator defined on (a domain within) $\mathbb{R}^d$ (see Definition \ref{defA}), $\mathcal{L}_{X_t}$ denotes the law of $X_t$, $W$ is a Brownian motion (BM for short), $N$ is a Poisson random measure (PRM for short) defined on $(\Omega, \mathcal{F}, \mathbb{P})$ with intensity measure $\nu$, $\tilde{N}(\mathrm{d}z, \mathrm{d}t) := N(\mathrm{d}z, \mathrm{d}t) - \nu(\mathrm{d}z) \mathrm{d}t$ denotes the compensated Poisson measure, and $Z$ is a locally compact Polish space. We assume that $N$ and $W$ are independent.

%We notice that $b$, $\sigma$, $G$ are the functionals of path $X$ and the law of $X$.

For any given probability measure $\mu$, let
\begin{eqnarray}\label{coef}
    b(x,\mu):=\int_{\mathbb{R}^d} \tilde{b}(x,y)\mu(dy),~~\sigma (x,\mu):=\int_{\mathbb{R}^d} \tilde{\sigma}(x,y)\mu(dy),\nonumber\\
    \int_{Z}G(x,\mu,z)\nu(dz):=\int_{Z}\int_{\mathbb{R}^d} \tilde{G}(x,y,z)\mu(dy)\nu(dz),
\end{eqnarray}
where $\tilde{b}: \mathbb{R}^d\times\mathbb{R}^d\rightarrow\mathbb{R}^d$, $\tilde{\sigma}: \mathbb{R}^d\times\mathbb{R}^d\rightarrow\mathbb{R}^d\otimes\mathbb{R}^d$ and $\tilde{G}: \mathbb{R}^d\times\mathbb{R}^d\times Z \rightarrow\mathbb{R}^d$ are all continuous functions.

When $A=0$ and $G=0$, equation \eqref{eq1} is the classical McKean-Vlasov stochastic differential equations (MVSDEs for short) diven by Brownian motion, which was first suggested by Kac \cite{Kac1,Kac2} as a stochastic toy model for the Vlasov kinetic equation of plasma, and then introduced by McKean \cite{Mckean}. The theory and applications of MVSDEs and associated interacting particle systems have been extensively studied by a large number of researchers under various settings. One can refer to \cite{CGM08,G92,GH03,GLWZ22,GLWZ21,HSS21,LW20,LWZ21,Mckean,MV20,RZ21} and the references therein.  The large deviation principle (LDP for short) was established by Herrmann et al. \cite{Herrmann} and Dos Reis et al. \cite{Reis}. The MVSDEs with jumps have been extensively studied in recent years, see \cite{JMW08,MSSZ20,LSZZ,NBKDR} etc. The third author et al.  \cite{LSZZ} established the LDP for the MVSDEs driven by L\'{e}vy noise.

When $A \neq 0$ and $G=0$, equation \eqref{eq1} become the MMVSDEs driven by Brownian motion. C\'{e}pa \cite{Cepa1,Cepa2} first studied the classical Multivalued stochastic differential equations (MSDEs for short), i.e. the case that $b$ and $\sigma$ are independent of $\LL_{X_t}$. In the papers, C\'{e}pa introduced a pair of continuous $\FF_t$-adapted processes $(X_t,K_t)$ to solve the MSDEs. After that, many researchers have begun to study the MSDEs, see \cite{Hu10,RW12,RWZ15,RWZ10,RX10,RXZ10,Xu,ZH20,ZXC07}.  Ren et al. \cite{RXZ10} proved the Freidlin-Wentzell LDP for MSDES by using the weak convergence method developed by Dupuis and Ellis \cite{DE97},  Ren et al.  \cite{RWZ15} showed a general LDP, and Zhang \cite{ZH20} established the moderate deviation principle (MDP for short). Though there are a lot of results about MSDEs, there are only  few results on MMVSDEs (i.e. $b$ and $\sigma$ depend on $\LL_{X_t}$).  Recently, Chi \cite{Chi14} proved the existence and uniqueness of strong solutions for MMVSDEs, and obtained the existence of the weak solutions for them. Qiao and Gong \cite{QG} established the well-posedness and stability under non-Lipschitz conditions on the coefficients. The third author et al.  \cite{FLQZ23} established the LDP, MDP and central limit theorem.

When $A \neq 0$ and $G\neq 0$, if the coefficients $b$ and $\sigma$ are independent of $\LL_{X_t}$, equation \eqref{eq1} become the MSDEs with jumps. We emphasize that for a multi-valued operator $A$ whose domain does not necessarily cover the entire space, the continuity of sample paths in the stochastic processes under study is indispensable for establishing the existence of solutions. It is worth noting that in the case of jump processes, intuition suggests and a simple example can demonstrate that the equation may admit no solution if the domain of $A$ is not the entire space.  Ren and Wu \cite{RW11} proved the existence and uniqueness of solutions of MSDEs driven by Poisson point processes under an additional assumption that the domain of the multivalued maximal monotone operator is the whole space $\rr^d$. Later in \cite{Wu12}, Wu relaxed the additional assumption.  Wu \cite{Wu11} established the LDP for MSDEs with Poisson jumps. Ren and Wu \cite{RW13} studied the optimal control problem about the MSDEs with L\'{e}vy jumps. When the coefficients $b$ and $\sigma$ depend on $\LL_{X_t}$, we prove the existence and uniqueness of the strong solution of MMVSDE \eqref{eq1} with jumps in another paper \cite{CGLZ24}, as well as the weak solution. However there are still few works about MMVSDEs with jumps.

In this paper, we aim to study the LDP and MDP about MMVSDEs with non-lipschitz coefficients driven by L\'{e}vy noise. For any $\e \in (0,1]$, consider the following MMVSDEs with jumps:
\beq\label{eq2}
\begin{cases}
\d X_t^\e \in -A(X_t^\e) \d t +b_\e(X_t^\e,\LL_{X_t^\e})\d t+\sqrt{\e}\sigma_\e (X_t^\e,\LL_{X_t^\e})\d W_t\\
\quad \quad \quad+ \e \int_{Z} G_\e (X^\e_{t-},\LL_{X_t^\e},z)\tilde{N}^{\e^{-1}} (\d z,\d t),\ t\in[0,T], \ \e \in (0,1], \\
X_0^\e=x_0 \in \overline{D(A)}.
\end{cases}
\nneq
Assume that  $(X^\e,K^\e)$ is a {\bf strong} solution of \eqref{eq2} (see Definition \ref{solution}). Our aim is to investigate the deviations of $X^\e$ from the deterministic solution $X^0$ by studying the asymptotic behavior of the trajectory
$$
\frac{X_t^\e- X_t^0}{\lambda(\e)},
$$
where $(X^0,K^0)$ satisfies the following mutivalued differential equation
\beq\label{eq00}
\begin{cases}
\d X_t^0 \in -A(X_t^0) \d t +b(X_t^0,\LL_{X_t^0})\d t,\ t\in[0,T], \\
X_0^0=x_0 \in \overline{D(A)}.
\end{cases}
\nneq

Our contribution is as follows:
\benu
\item
when $\lambda(\e) \equiv 1$, we establish the LDP for \eqref{eq1};
\item
when $\lambda(\e)$ satisfies
$$
\lambda(\e) \to 0,\quad \frac{\e}{ \lambda^2(\e)} \to 0 \mbox{ as } \e \to 0,
$$
we establish the MDP for \eqref{eq1}.
\nenu

Large deviation principles can provide an exponential estimate for tail probability in terms of some explicit rate function. Recent years, there are a lot of works on LDP for classical stochastic evolution equations and SPDEs driven by BM and PRM. Among the approaches to deal with these problems, the weak convergence method based on a variational representation for positive measurable functionals of a BM and PRM, see \cite{BCD13,BD00,BDG16,BDM08,BDM11}. The reader can refer to \cite{BD19} for an excellent review of the advances on the weak convergence method during the past decade.

For the MVSDE without jumps, Herrmann et al. \cite{Herrmann} obtained the LDP in path space equipped with the uniform
norm, assuming the superlinear growth of the drift but imposing coercivity condition, and a
constant diffusion coefficient. Dos Reis et al. \cite{Reis} obtained LDPs in path space topologies
under the assumption that coefficients $b$ and $\sigma$ have some extra regularity with respect to
time. The approach in \cite{Herrmann} and  \cite{Reis} is to first replace the distribution $\LL_{X_t^\e}$ of $X_t^\e$ in the
coefficients with a Dirac measure $\delta_{X_t^0}$ and then to use discretization, approximation and exponential equivalence arguments. However, the discretization and approximation techniques
can not be applied to the case of L\'evy  noise and also require stronger conditions
on the coefficients even in the Gaussian case.
Therefore, in this paper, we apply the weak convergence method to establish the LDP and MDP for $X^\e$ under non-Lipschitz condition. The Bihari's inequality is used to overcome the challenges arising from the non-Lipschitz conditions on the coefficients.

The rest of the paper is organized as follows. In section 2, we introduce some notions and notations about MMVSDEs and the Laplace principle. In section 3, we present the main results on LDP and MDP for \eqref{eq1}. The proofs will be given in section 4.

\section{Preliminaries}

\vskip0.3cm
In this section, we recall some basic notions and notations.

\subsection{Notations}

\subsubsection{Notation and Preliminaries}
Set $\nn: =\{1,2,3,...\}$, $\rr:=(-\infty, +\infty)$ and $\rr_+:=[0,+\infty)$. For a metric space $S$, define the following notations:
\benu
\item
$\BB (S)$: the Borel $\sigma$-field on $S$;
\item
$C_c(S)$: the space of real-valued continuous functions with compact supports;
\item
$C([0,T],S)$: $C([0,T],S)$: the space of continuous functions $f:[0,T] \to S$ equipped with the uniform convergence topology;
\item
$\DD([0,T],S)$: the space of all c\`{a}dl\`{a}d functions $f:[0,T] \to S$ equipped with the Skorokhod topology.
\nenu
%we denote by $\BB (S)$ the Borel $\sigma$-field on $S$ and by $C_c(S)$ the space of real-valued continuous functions with compact supports. Let $C([0,T],S)$ represent the space of continuous functions $f:[0,T] \to S$ equipped with the uniform convergence topology. Let $\DD([0,T],S)$ represent the space of all c\`{a}dl\`{a}d functions $f:[0,T] \to S$ equipped with the Skorokhod topology.

For an $S$-valued measurable map $X$ defined on some probability space $(\Omega,\FF,\pp)$, we denote by $\LL_X$ the measure induced by $X$ on the measurable $(S,\BB(S))$. For a measurable space $(U,\mathcal{U})$, let $Pr(U)$ denote the space of all probability measures defined on $(U,\mathcal{U})$.

Moreover, if $S$ is a locally compact Polish space, we denote by $ M(S) $  the space of all Borel measures on $S$ and $M_{FC}(S)$ the set of all $\mu \in M(S)$ with $\mu (O)< +\infty$ for each compact subset $O \subseteq S$.  $M_{FC}(S)$ is equipped with the weakest topology, thus all mappings
$$
\mu \to \int_S f(s) \mu(\d s), \forall f \in C_c(S)
$$
are continuous. This topology is metrizable, so $M_{FC}(S)$ is a Polish space (see \cite{BDM11} for more details).

\subsubsection{Framework}
Throughout this paper, we fix $T>0$ as a constant. Let $\rr^d$ be equipped with the standard inner product $\langle \cdot, \cdot  \rangle$ and induced Euclidean norm $| \cdot|$. For matrices in the space $\rr^d \otimes \rr^d$, we denote by $\|\cdot\|_{\rr^d \otimes \rr^d}$  the Hilbert-Schmidt norm.
%Let $|\cdot|$ be the Euclidean norm of vector in $\rr^d$.

Let $Z$ be a locally compact Polish space equipped with a $\sigma$-finite measure $\nu \in M_{FC}(Z)$. Consider the filtered probability space $ (\Omega, \mathcal{F}, \{ \mathcal{F}_t\}_{t\in [0,T]},\pp)$ with
$$
\Omega:=C([0,T], \rr^d) \times M_{FC}( [0,T] \times Z \times \rr_+), \quad \mathcal{F}:=\BB (\Omega).
$$

We introduce the coordinate mappings
$$
\begin{aligned}
&W:\Omega \to C([0,T], \rr^d), \quad W(\alpha,\beta)(t) =\alpha(t), \ t \in [0,T],\\
&N:\Omega \to M_{FC}([0,T]\times Z \times \rr_+), \quad N(\alpha ,\beta)=\beta.
\end{aligned}
$$

For each $t \in [0,T]$, defined the $\sigma$-algebra
$$
\GG_t:=\sigma \left( \{ W_s,N((0,s)\times A): 0 \le s \le t, A \in \BB(Z \times \rr_+) \}\right).
$$
Given the measure $\nu$, by the result in \cite{IW81}, there exists a unique probability measure $\pp$ on $(\Omega,\FF)$ such that:
\benu
\item
$W$ is a $\rr^d$-cylindrical BM;
\item
$N$ is a PRM on $[0,T]\times Z \times \rr_+$ with intensity measure $Leb_T \otimes \nu \otimes Leb_{\infty}$, where  $Leb_T$ and $Leb_{\infty}$ stand for the Lebesgue measures on $[0,T]$ and $\rr_+$ respectively;
\item
$W$ and $N$ are independent.
\nenu

Denote by $\mathbb{F}:=\{\FF_t\}_{t\in [0,T]}$ the $\pp$-completion of $\{\GG_t\}_{t\in [0,T]}$  and $\mathcal{P}$ the $\mathbb{F}$-predictable $\sigma$-field on $[0,T] \times \Omega$. The cylindrical BM $W$ and the PRM $N$ will be defined on the (filtered) probability space $(\Omega,\mathcal{F},\mathbb{F},\pp)$. The corresponding compensated PRM will be denoted by $\tilde{N}$.

Denote
$$
\mathcal{R}_+= \{\varphi:[0,T]\times \Omega \times Z \to \rr_+: \varphi \mbox{ is } (\mathcal{P} \otimes \mathcal{B}(Z))/\mathcal{B}(\rr_+)\mbox{-measurable} \}.
$$
For any $\varphi \in \mathcal{R}_+,\ N^{\varphi}:\Omega \to M_{FC}([0,T]\times Z )$ is a counting process on $[0,T] \times Z$ defined by
$$
N^{\varphi((0,t]\times A)}=\int_{(0,t]\times A \times \rr_+} 1_{[0,\varphi(s,z)]}(r)N(\d s , \d z , \d r), 0 \le t \le T, A \in \mathcal{B}(Z).
$$
$N^{\varphi}$ can be viewed as a controlled random measure, with $\varphi$ selecting the intensity. Analogously, the compensated version $\tilde{N}^\varphi$ is defined by replacing $N$ with $\tilde{N}$. If $\varphi \equiv c>0$, we write $N^\varphi=N^c$ and $\tilde{N}^\varphi=\tilde{N}^c$.

\subsubsection{Energy-Constrained Spaces}
For each $f \in L^2([0,T], \rr^d)$, define
$$
Q_1(f):=\frac{1}{2} \int_0^T | f(s) |^2 \d s,
$$
and for each $m>0$, denote
$$
S_1^m=\left\{f \in L^2([0,T],\rr^d): Q_1(f) \le m  \right\}.
$$
Equipped with the weak topology, $S_1^m$ is a compact subset of $L^2([0,T],\rr^d)$.

For each measurable function $g:[0,T] \times Z \to [0, +\infty)$, define
$$
Q_2 (g) := \int_{[0,T] \times Z} l(g(s,z)) \nu(\d z)\d s,
$$
where $l(x)=x \log x -x +1$, $l(0):=1$. For each $m >0$, denote
$$
S_2^m := \left\{g:[0,T] \times Z \to [0,+\infty)| Q_2 (g) \le m \right\}.
$$
Any measurable function $g \in S_2^m$ can be identified with a measure $\hat{g} \in M_{FC}([0,T] \times Z)$, defined by
$$
\hat{g} (A) = \int_A g(s,z) \nu(\d z) \d s, \ \forall A \in \BB([0,T] \times Z).
$$
This identification induces a topology under which $S_2^m$ is a compact space.

Denote
$$
S := \bigcup_{m\in \nn} \left\{ S_1^m \times S_2^m \right\},
$$
and equip it with the usual product topology.

Let $\{ Z_n\}_{n \in \nn}$ be a sequence of compact sets satisfying that $Z_n \subseteq Z $ and  $Z_n \nearrow Z $. For each $n \in \nn$, let
$$
\mathcal{R}_{b,n} = \left\{ \psi \in \mathcal{R}_{+}: \psi(t,z,\omega) \in
\begin{cases}
[\frac{1}{n},n], \mbox{ if } z \in Z_n,\\
\{ 1\}, \mbox{ if } z \in Z_n^c,
\end{cases}
\mbox{ for all } (t,\omega) \in [0,T] \times \Omega \right\},
$$
and
$\mathcal{R}_b= \bigcup_{n=1}^{+\infty} \mathcal{R}_{b,n}$. For any $m \in (0, +\infty)$, let $ \SS_1^m$ and $\SS_2^m$ be two spaces of stochastic processes on $\Omega$ defined by
\begin{align*}
&\SS_1^m := \{\varphi: [0,T] \times \Omega \to \rr^d: \mathbb{F} \mbox{-predictable and } \varphi (\cdot, \omega) \in S_1^m \mbox{ for } \pp\mbox{-a.s.} \ \omega \in \Omega \},\\
&\SS_2^m := \{ \psi \in \mathcal{R}_b: \psi(\cdot, \cdot, \omega) \in S_2^m \mbox{ for } \pp \mbox{-a.s. } \omega \in \Omega \}.
\end{align*}

\subsubsection{Wasserstein distance}
Denote by $\PP(\rr^d)$ the collection of probability measures on $(\rr^d, \BB (\rr^d))$, and
$$
\PP_2:= \left\{  \mu \in \PP(\rr^d): \| \mu \|_2^2 :=\int_{\rr^d} |y|^2 \mu (\d y) < +\infty \right\}
$$
the space of probability measures with finite second moments.
Note that $ \PP_2 $ is a Polish space equipped with the Wasserstein distance
$$
W_2(\mu_1,\mu_2):= \inf_{\phi \in \mathscr{C} ( \mu_1, \mu_2)}\left( \int_{\rr^d \times \rr^d} |x-y|^2 \phi(\d x, \d y)
 \right)^{\frac{1}{2}},
$$
where $\mathscr{C} ( \mu_1, \mu_2) $ is the set of all couplings for any $ \mu_1, \mu_2\in \PP_2$.

\brmk\label{WXY}
By the definition, it is easy to see that for any $\rr^d$-valued random variables $X$ and $Y$,
\beq
W_2(\LL_X,\LL_Y)\le [\ee|X-Y|^2]^{1/2}.
\nneq
\nrmk

\subsection{Maximal Monotone Operator}

Let $2^{\rr^d}$ be the set of all the subsets of $\rr^d$,  $A$ is said to be a multivalued operator on $\rr^d$ if $A$ is an operator from $\rr^d$ to $2^{\rr^d}$. Let
$$
\begin{aligned}
&D(A):= \{ x \in \rr^d: A(x) \neq \emptyset \},\\
&Gr(A):= \{ (x,y) \in \rr^{2d}: x \in D(A), y \in A(x) \}.
\end{aligned}
$$

\bdef\label{defA}

\benu
\item
A multivalued operator $A$ is called monotone if
$$
\langle x_1-x_2, y_1 - y_2 \rangle \ge 0, \quad \forall (x_1,y_1),(x_2,y_2) \in Gr(A).
$$
\item
A monotone operator $A$ is called maximal if
$$
(x_1,y_1) \in Gr(A) \Leftrightarrow \langle x_1-x_2, y_1 - y_2 \rangle \ge 0, \quad \forall (x_2,y_2) \in Gr(A).
$$
\nenu
\ndef

A particular example of a multi-valued maximal monotone operator is the sub-differential of a proper, convex and  lower semi-continuous function $\varphi:\mathbb{R}^d\rightarrow(-\infty,+\infty]$, defined by
\begin{eqnarray*}
    \partial\varphi(x):=\Big\{x^*\in\mathbb{R}^d\mid\langle y-x,x^*\rangle+\varphi(x)\leq\varphi(y),\forall y\in\mathbb{R}^d\Big\}.
\end{eqnarray*}
In the one-dimensional case, every maximal monotone operator on $\mathbb{R}$ can be represented in this manner.

The following is an explicit example.
\begin{example}
Consider the indicator function of a closed convex set $\mathbb{K}\subseteq\mathbb{R}^d$,
\begin{eqnarray*}
    \text{I}_{\mathbb{K}}(x):=\Big\{\begin{array}{ll}
                            0,~~~~~~~~x\in\mathbb{K}; \\
                            +\infty,~~~~x\in\mathbb{R}^d\backslash\mathbb{K}.
                          \end{array}
\end{eqnarray*}
The sub-differential operator of $I_{\mathbb{K}}$ is given by
\begin{eqnarray*}
    \partial \text{I}_{\mathbb{K}}(x)=\Bigg\{\begin{array}{lll}
                                 {0},~~~~~~~~x\in \text{Int}(\mathbb{K}); \\
                                 \Pi_x,~~~~~~x\in \text{Fr}(\mathbb{K}); \\
                                 \emptyset,~~~~~~~~x\in\mathbb{R}^d\backslash\mathbb{K},
                               \end{array}
\end{eqnarray*}
where $\text{Fr}(\mathbb{K})$ denotes the frontier of the set $\mathbb{K}$ and $\Pi_x$ is the exterior normal cone which is defined with respect to $\mathbb{K}$ at $x$.
\end{example}

Given $T>0$, let
$$
V_0=\{ K \in C([0,T],\rr^d): K \mbox{ is of finite variation and } K_0=0\}.
$$

Set
$$
\AA:= \left\{  (X,K):X \in \DD([0,T],\overline{D(A)}), K \in V_0 \mbox{ and } \langle X_t -x, \d K_t -y \d t\rangle \ge 0, \forall (x,y) \in Gr(A)\right\}.
$$
We have the following characterization for the element in $\AA$ (cf. \cite{Cepa2,ZXC07}).

\bprop\label{equivalent}
Let $(X,K)$ be a pair of functions with $X \in \DD([0,T],\overline{D(A)})$ and $K \in V_0$. Then the following statement are equivalent:
\benu
\item
$(X,K) \in \AA;$

\item
For any $(x,y) \in \DD([0,1], \rr^d)$ with $(x_t,y_t) \in Gr(A)$, it holds that
$$
\langle X_t -x_t, \d K_t -y_t \d t \rangle \ge 0;
$$

\item
For any $(X',K') \in \AA$, it holds that
\beq\label{monotone}
\langle X_t -X'_t, \d K_t - \d K'_t \rangle \ge 0.
\nneq
\nenu
\nprop

\subsection{Solutions to multivalued McKean-Vlasov SDEs with jumps }
Given $T>0$.
%Let $V_0$ be the set of all continuous functions $K:[0,T] \to \rr^d$ with finite variations and $K_0=0$.
For any $K \in V_0$ and $s \in [0,T]$, denote $|K|_0^s$ by the variation of $K$ on $[0,s]$.

\bdef\label{solution}(Strong solution)
A  pair of $(\FF_t)$-adapted processes $(X,K)$ is called a strong solution of \eqref{eq1} with the initial value $x$ if $(X,K)$ on a filtered probability space $(\Omega, \FF, \{\FF_t\}_{t\in [0,T]}, \pp)$ such that

\benu
\item
$$
\pp(X_0 = x_0)=1;
$$
\item
$$
(X_\cdot(\omega),K_\cdot(\omega)) \in \AA,\ \pp \mbox{-a.s.};
$$
\item
it holds that
$$
\pp \left\{ \int_0^T |b(X_s,\LL_{X_s})|+\|\sigma(X_s,\LL_{X_s})\|_{\rr^d \otimes \rr^d}^2+\int_Z |G(X_s,\LL_{X_s},z)|^2 \nu(\d z) \d s< +\infty \right\}=1
$$
and

\begin{align*}
X_t=&x_0-K_t+\int_0^t b(X_s,\LL_{X_s})\d s+\int_0^t \sigma (X_s,\LL_{X_s})\d W_s\\
&+ \int_0^t\int_{Z} G(X_{s-},\LL_{X_s},z)\tilde{N} (\d z,\d s),\ t\in[0,T],\ \pp \mbox{-a.s.}.
\end{align*}

%\item
%for any c\`{a}dl\`{a}g and $(\FF_t)$-adapted functions $(\alpha, \beta)$ with
%$$
%(\alpha_t, \beta_t) \in Gr(A), \quad \forall t \in [0,T],
%$$
%the measure $\langle X_t - \alpha_t, \d K_t- \beta_t \d t \rangle \ge 0$ almost surely.
\nenu
\ndef

%
%\bdef\label{solution}
%A  pair of processes $(X,K)$ is called a strong solution of \eqref{eq1} with the initial value $x$ if $X,K$ are $(\FF_t)$-adapted processes satisfying
%
%\benu
%\item
%$X$ is c\`{a}dl\`{a}g, $X_0=x$ and $X_t \in \overline{D(A)}$ for every $t \ge 0$;
%
%\item
%$K$ is continuous, $K_0=0$, and the total variation $|K|_T^0 < +\infty$ almost surely for any $0 < T <+\infty$;
%
%\item
%\begin{align*}
%X_t=&x+\int_0^t b(X_s,\LL_{X_s})\d s+\int_0^t \sigma (X_s,\LL_{X_s})\d W_s\\
%&+ \int_0^t\int_{Z} G(X_{s-}-K_t,\LL_{X_s},z)\tilde{N} (\d z,\d s),\ t\in[0,T],\ a.s.\ ;
%\end{align*}
%
%\item
%for any c\`{a}dl\`{a}g and $(\FF_t)$-adapted functions $(\alpha, \beta)$ with
%$$
%(\alpha_t, \beta_t) \in Gr(A), \quad \forall t \in [0,T],
%$$
%the measure $\langle X_t - \alpha_t, \d K_t- \beta_t \d t \rangle \ge 0$ almost surely.
%\nenu
%\ndef
%
%
%\blem\label{Amonotone}
%Let $(X,K)$ and $(X',K')$ be two pairs of processes satisfying (1), (2), (4) of the above definition. Then
%\beq\label{monotone}
%\langle X_t -X'_t, \d K_t - \d K'_t \rangle \ge 0
%\nneq
%\nlem

\blem\label{lem24}
Suppose that $Int (D(A)) \neq \emptyset $. Then for any $ a \in Int(Dom(A))$, there exist two positive constants $r$ and $\mu $ such that for any pair $(X,K)$ satisfying Definition \ref{solution},
$$
\int_s^t \langle X_v -a , \d K_v \rangle \ge r |K|_t^s - \mu \int_s^t |X(v)-a| \d v -r \mu(t-s),
$$
where $|K|_t^s$ denotes the total variation of $K$ on $[s,t]$.
\nlem

\bdef(Weak solution)
We say that equation (\ref{eq1}) admits a weak solution with initial law $\mathcal{L}_{X_0}\in\mathcal{P}(\mathbb{R}^d)$, if there exists a stochastic basis $\mathcal{S}:=(\Omega,\mathcal{F},\{\mathcal{F}_t\}_{t\geq0},\mathbb{P})$, a $d$-dimensional standard $\mathcal{F}_t$-Brownian motion $(W_t)_{t\geq0}$, a compensated Poisson measure $\widetilde{N}$ as well as a pair of $\mathcal{F}_t$-adapted processes $(X,K)$ defined on $\mathcal{S}$ such that\\
(i) $X_0$ has the law $\mathcal{L}_{X_0}$ and $(X_.(\omega),K_.(\omega))\in\mathcal{A}$ for $\mathbb{P}$-almost all $\omega\in\Omega$;\\
(ii) it holds that \begin{eqnarray*}
\int_0^T\left(|b(X_t,\mathcal{L}_{X_t})|+\|\sigma(X_t,\mathcal{L}_{X_t})\|_{\rr^d \otimes \rr^d}^2+\int_{Z}|G(X_{t-},\mathcal{L}_{X_t},z)|^2\nu(\d z)\right)\d t<+\infty
\end{eqnarray*}
and
\begin{eqnarray*}
    X_t=x_0-K_t+\int_0^t b(X_s,\mathcal{L}_{X_s})\d s+\int_0^t \sigma(X_s,\mathcal{L}_{X_s})\d W_s \\
    ~~~~+\int_0^t \int_{Z} G(X_{s-},\mathcal{L}_{X_s},z)\widetilde{N}(\d s,\d z),~ t\in[0,T].
\end{eqnarray*}

Such solution will be denote by $(\SS;W,\tilde{N},(X,K))$.
\ndef

\bdef\label{lawuni}
(Uniqueness in law) Let $ (\SS;W,\tilde{N},(X,K)) $ and $ (\SS';W',\tilde{N}',(X',K')) $ be two weak solutions with the same initial distribution $ \mathcal{L}_{X_{0}}=\mathcal{L}_{X'_{0}} $. The uniqueness in law is said to hold for \eqref{eq1} if $ (X,K) $ and $ (X',K') $ have the same law.
\ndef
	
\bdef\label{pathuni}
(Pathwise Uniqueness) Let $ (\SS;W,\tilde{N},(X,K)) $ and $ (\SS;W,\tilde{N},(X',K')) $ be two weak solutions with the same initial distribution. The pathwise uniqueness is said to hold for \eqref{eq1} if for all $ t\in[0,T] $, $ (X_{t},K_{t})=(X'_{t},K'_{t}) $.
\ndef

\subsection{Large deviation principle}
We first recall the definitions of a rate function and LDP. Let $\EE$ be a Polish space with the Borel $\sigma$-field $\BB(\EE)$.

\bdef (Rate function)
A function $I: \EE \to [0,+\infty]$ is called a rate function on $\EE$, if for each $M< +\infty$, the level set $\{x \in \EE: I(x) \le M\}$ is a compact subset of $\EE$.
\ndef

\bdef (LDP)
Let $I$  be a rate function on $\EE$. Given a collection $\{ h_\e \}_{\e >0}$ of positive reals, a family $\{ X^\e \}_{\e>0}$ of $\EE$-valued random elements is said to satisfy a LDP on $\EE$ with speed $h_{\e}$ and rate function $I$ if the following two claims hold:
\benu
\item[(a)](Upper bound)
For each closed subset $C$ of $\EE$,
$$
\limsup_{\e \to 0} h_\e \log P(X^\e \in C) \le - \inf_{x\in C} I(x);
$$
\item[(b)](Lower bound) For each open subset $O$ of $\EE$,
$$
\liminf_{\e \to 0} h_\e \log P(X^\e \in O) \ge - \inf_{x\in O} I(x).
$$
\nenu
\ndef

\subsection{Bihari's inequality}	

The following lemma  will be used in the proofs.

\begin{lemma}\label{Bihari}
		(Bihari's inequality \cite{M08})  Let $\varrho: \rr_+ \to \rr_+$ be a continuous nondecreasing function such that $\varrho(t)>0$ for all $t>0$. Let $g(\cdot)$ be a Borel measurable bounded nonnegative function on $[0,T]$. Let $q(\cdot)$ be a nonnegative integrable function on $[0,T]$. If
$$
g(t)\leq C+\int_{0}^{t}q(s)\varrho(g(s))\mathrm{d}s, \quad t\in [0,T]
$$
where $C>0$ is a constant, then
$$
g(t)\leq f^{-1} \left( f(C)+\int_0^t q(s) \d s\right)
$$
holds for all $t \in [0,T]$ such that
$$
f(C)+\int_0^t q(s) \d s \in Dom(f^{-1}),
$$
where $f(r)=\int_1^r \frac{1}{ \varrho(s)} \d s$ and $f^{-1}$ is the inverse function of $f$.
\end{lemma}

\section{Large and Moderate deviation principles}

In this section, we consider the following perturbed equation of \eqref{eq1},

\beq\label{eq2'}
\begin{cases}
\d X_t^\e \in -A(X_t^\e) \d t +b_\e(X_t^\e,\LL_{X_t^\e})\d t+\sqrt{\e}\sigma_\e (X_t^\e,\LL_{X_t^\e})\d W_t\\
\quad \quad \quad+ \e \int_{Z} G_\e (X^\e_{t-},\LL_{X_t^\e},z)\tilde{N}^{\e^{-1}} (\d z,\d t),\ t\in[0,T], \ \e \in (0,1], \\
X_0^\e=x_0 \in \overline{D(A)},
\end{cases}
\nneq
where
$$
b_\e:  \overline{D(A)} \times \PP_2 \to \rr^d, \quad \sigma_\e: \overline{D(A)} \times  \PP_2 \to \rr^d \otimes \rr^d
$$
and
$$
G_\e : \overline{D(A)} \times \PP_2 \times Z \to \rr^d
$$
are measurable maps.

For any given probability measure $\mu$,
\begin{eqnarray*}
    b_\e(x,\mu):=\int_{\mathbb{R}^d} \tilde{b}_\e(x,y)\mu(dy),~~\sigma_\e (x,\mu):=\int_{\mathbb{R}^d} \tilde{\sigma}_\e(x,y)\mu(dy),\\
    \int_{Z}G_\e(x,\mu,z)\nu(dz):=\int_{Z}\int_{\mathbb{R}^d} \tilde{G}_\e(x,y,z)\mu(dy)\nu(dz),
\end{eqnarray*}
where $\tilde{b}_\e: \mathbb{R}^d\times\mathbb{R}^d\rightarrow\mathbb{R}^d$, $\tilde{\sigma}_\e: \mathbb{R}^d\times\mathbb{R}^d\rightarrow\mathbb{R}^d\otimes \rr^d$ and $\tilde{G}_\e: \mathbb{R}^d\times\mathbb{R}^d\times Z \rightarrow\mathbb{R}^d$ are all continuous functions.

\subsection{Large deviation principle}

%We assume that\\
%
%{\bf (H0)} For any $\e \in (0,1]$, there exists a unique solution $(X^\e, K^\e)$ to \eqref{eq2}.
%
%Moreover, by the classical Yamada-Watanabe theorem, there exists a measurable function $\GG^\e$ such that
%$$
%X^\e=\GG^\e (\sqrt{\e }W,\e N^{\e^{-1}}  ).
%$$

The aim of this section is to establish the large deviation principle for the
the solutions $\{ X^\e, \e \in (0,1]\}$ to \eqref{eq2'} as $ \e $ decreases to $0$. We first present the assumptions.

\bhyp\label{H11}
There exists $L>0$, for all $x,x',y \in \rr^d$ and $\mu,\mu' \in \PP_2$, such that
\benu
\item[{\bf (H1)}]
A is a maximal monotone operator and ${\rm Int} D(A) \neq \emptyset$.
\item[{\bf (H2)}]
The functions $b$, $\sigma$ and $G$ satisfy the following conditions:
\begin{align*}
&\langle  x-x',b(x,\mu)-b(x',\mu') \rangle  \le \kappa\left( |x-x'|^2\right) +\kappa \left( W_2^2(\mu, \mu') \right),\\
&\| \sigma(x,\mu) - \sigma (x',\mu')\|_{\rr^d \otimes \rr^d}^2 \vee \int_Z | G(x,\mu,z) -G(x',\mu',z)|^2 \nu (\d z)\\
\le& \kappa \left(|x-x'|^2\right)  +  \kappa \left(W_2^2(\mu, \mu')\right),\\
\end{align*}
where $\kappa: \rr^+ \to \rr^+$ is a continuous and non-decreasing concave function with $ \kappa(0)=0,\ \kappa(u)>0$ for every $u>0$ such that $\int_{0^+} \frac{1}{\kappa(u)} \d u =+\infty$.
\item[{\bf (H3)}]
The functions $\tilde{b}$, $\tilde{\sigma}$ and  $\tilde{G}$ are continuous in $(x,y)$ and satisfy the linear growth condition:
\begin{align}\label{H31}
|\tilde{b}(x,y)|^2 \vee \|\tilde{\sigma}(x,y)\|^2_{\rr^d \otimes \rr^d} \vee \int_{Z} |\tilde{G}(x,y,z)|^2 \nu(\d z) \le L(1+|x|^2+|y|^2).
\end{align}
\item[{\bf (H4)}]
For every $x \in \overline{D(A)}$, $x+\tilde{G}(x,y,z) \in \overline{D(A)}$.
\nenu
\nhyp
Under {\bf Hypothesis \ref{H11}}, we have proved in \cite{CGLZ24} that equation \eqref{eq1} has a unique strong solution $ \left( X, K\right) $.

\brmk
It is obvious that  {\bf (H3)} implies the following statement: for all $(x,\mu)\in\mathbb{R}^d\times\mathcal{P}_2$,
\benu
\item[{\bf (H3)'}]
$b$, $\sigma$ and $G$ are continuous in $(x,\mu)$ and satisfy
\begin{align}\label{blinear}
    |b(x,\mu)| \vee \|\sigma(x,\mu)\|_{\rr^d \otimes \rr^d}^2 \vee \int_{Z}|G(x,\mu,z)|^2\nu(\d z)\leq L\left(1+|x|^2+\|\mu\|_2^2\right).
\end{align}
\nenu

Indeed,  {\bf (H3)'} is sufficient for using weak convergence method to prove the LDP and MDP. However, in our another paper \cite{CGLZ24}, the stronger assumption {\bf (H3)} is required to guarantee the existence and uniqueness of a strong solution to the stochastic differential equation \eqref{eq1}.
\nrmk

To establish the large deviation principle (LDP) and ensure the existence and uniqueness of strong solutions for equation \eqref{eq2'}, we need the following notations and assumptions.

Set
$$
L^2(\nu) = \{ f :Z \to \rr|f \mbox{ is } \BB(Z)/\BB(\rr ) \mbox{-measurable and } \int_Z |f(z)|^2 \nu (\d z) < +\infty  \}
$$
and
\begin{align}
\mathcal{H} = \Big\{g: Z \to \rr_+ |&g \mbox{ is Borel measurable and there exists } c>0 \mbox{ such that } \nonumber\\
 &\int_O e^{cg^2(z)} \nu (\d z) < +\infty \mbox{ for all } O \in \BB(Z) \mbox{ with } \nu(O) <+\infty \Big\}.
\end{align}

\bhyp\label{H22}
\benu
\item[{\bf (H5)}]
As $\e \to 0$, the maps $b_\e$ and $ \sigma_\e $ converge uniformly to $b$ and $\sigma$ respectively, i.e., there exist some  nonnegative constants $\rho_{b,\e}$ and $\rho_{\sigma,\e }$ converging to $0$ as $\e \to 0$ such that
\begin{align*}
&\sup_{(x,\mu) \in \rr^d \times \PP_2}  |b_\e(x,\mu) -b(x,\mu)| \le \rho_{b,\e},\\
&\sup_{(x,\mu) \in \rr^d \times \PP_2}  \|\sigma_\e(x,\mu) -\sigma(x,\mu)\|_{\rr^d \otimes \rr^d} \le \rho_{\sigma,\e}.
\end{align*}
\item[{\bf (H6)}]
There exist $L_1,L_2,L_3 \in \HH \cap L^2(\nu)$ such that for all $ t \in [0,T]$, $x,x' \in \rr^d$, $\mu, \mu' \in \PP_2 $ and $z \in Z$,
$$
\begin{aligned}
&| G(x,\mu,z) -G(x',\mu',z) |^2 \le L_1^2(z)\left( \kappa \left(|x-x'|^2\right) +\kappa \left( W_2^2(\mu,\mu') \right) \right),\\
&|G(0,\delta_0,z)| \le L_2(z)
\end{aligned}
$$
and there exists nonnegative constant $\rho_{G,\e}$ converging to $0$ as $\e \to 0$ such that
$$
\sup_{(x,\mu)\in \rr^d \times \PP_2} | G_\e(x, \mu, z)-G(x,\mu,z)| \le \rho_{G,\e} L_3(z).
$$
\item[{\bf (H7)}]
The function $\tilde{G}_\e$ is continuous in $(x,y)$ and satisfy the linear growth condition, i.e., for some constant $L>0$ and for any $\e \in (0,1]$,
\begin{align}
\int_{Z} |\tilde{G}_\e(x,y,z)|^2 \nu(\d z) \le L(1+|x|^2+|y|^2)
\end{align}
and $x+\tilde{G}_\e(x,y,z) \in \overline{D(A)}$, $\forall z \in Z,\ y \in \rr^d$.
\nenu
\nhyp

\brmk
By {\bf (H3)} and {\bf (H5)}, for some constant $L>0$, we have the following condition:
\benu
\item[{\bf (H8)}]
for any $\e \in (0,1]$, the functions $\tilde{b}_\e$ and $\tilde{\sigma}_\e$ are continuous in $(x,y)$ and satisfy the linear growth condition
\begin{align}
|\tilde{b}_\e(x,y)|^2 \vee \|\tilde{\sigma}_\e(x,y)\|^2_{\rr^d \otimes \rr^d}  \le L(1+|x|^2+|y|^2).
\end{align}
\nenu
\nrmk

Although the value of $L$ may be different in each hypothesis, we use the same notation $L$ throughout this paper for convenience.

\brmk
Under hypotheses {\bf (H2), (H5), (H6)}, for any fixed $\e \in (0,1]$, it can be directly verified that $b_\e,\ \sigma_\e$ and $G_\e$ inherit the required conditions prescribed in {\bf (H2)}.
\nrmk

Hence, by {\bf Hypothesis \ref{H11}} and {\bf Hypothesis \ref{H22}}, applying the theorem in \cite{CGLZ24}, we can obtain that for any fixed $\e \in (0,1]$, equation \eqref{eq2'} admits a unique strong solution. Denote the solution by $ \left( X^{\e}, K^{\e}\right) $.
Moreover, by the classical Yamada-Watanabe theorem, there exists a measurable function $\GG^\e$ such that
$$
X^\e=\GG^\e (\sqrt{\e }W,\e N^{\e^{-1}}  ).
$$

By \cite[Theorem 2.10]{CGLZ24}, we can easily obtain the following result by taking the diffusion term and jump term as zero.

\bprop
Assume that {\bf (H1) and (H2)} hold, then there exists a unique pair of $(X^0, K^0)$ satisfying that
\benu
\item
$X^0 \in C([0,T], \overline{D(A)})$,
\item
$\int_0^T |b(X_s^0,\LL_{X_s^0}) | \d s < +\infty,\ (X_t^0,K_t^0) \in \AA,\ \forall t \in [0,T]$,
\item
\beq\label{eq0}
X_t^0=x_0 + \int_0^t b(X_s^0, \LL_{X_s^0}) \d s -K_t^0, \quad \forall t \in [0,T].
\nneq
\nenu
\nprop

\brmk Note that $X^0$ is a deterministic path, and $\LL_{X_s^0}=\delta_{X_s^0}$ for any $s\in [0,T]$.\nrmk

%In the sequel, we will always use $(X^0,K^0)$ to denote the unique solution to \eqref{eq0}.

Since when perturbing the BM and PRM of the mapping $\GG_\e (\cdot,\cdot) $, $\mu^\e$ the distribution in the coefficients is already deterministic and hence it is not affected by the perturbation. We use the method in \cite{LSZZ} to deal with this technical difficulty. So we have the following two lemmas.

The first one is stated in \cite[Theorem 3.8]{LSZZ}
%Denote
%$$
%\dd := \left\{x \in \DD([0,T],\overline{D(A)}): \int_0^T \| x(t)\| \d t < +\infty \right\}.
%$$
\blem\label{Lemma1} Assume that the following assumptions hold.

{\bf (A0)}: For any fixed $\e >0$ and $\LL_{X^\e}$, the maps $b_\e(\cdot,\LL_{X^\e}):\rr^d \to \rr^d $, $\sigma_\e(\cdot,\LL_{X^\e}): \rr^d \to \rr^d \otimes \rr^d$ and $ G_\e(\cdot, \LL_{X^\e},\cdot): \rr^d \times Z \to \rr^d $ are measurable maps.

{\bf (A1)}: {\bf Hypothesis \ref{H11}} and {\bf Hypothesis \ref{H22}} hold.

Then
\eqref{eq2} has a unique solution $(X^\e, K^\e) $ as stated in Definition \ref{solution} with initial value $X_0^\e=x_0 \in \overline{D(A)} $ and  $K_0^\e=0 $.

Moreover, we have

\benu
\item[(1)]
there exists a map $\Gamma_{\LL_{X^\e}}$ such that
$$
X^\e = \Gamma^\e_{\LL_{X^\e}} \left(\sqrt{\e} W_\cdot, \e N^{\e^{-1}}\right);
$$

\item[(2)]
for any $m \in (0,+\infty)$, $ u_\e = (\phi_\e, \psi_\e) \in \SS_1^m \times \SS_2^m$, let
$$
Z^{\e,u_\e} := \Gamma^\e_{\LL_{X^\e}} \left(\sqrt{\e} W_\cdot + \int_0^\cdot \phi_\e(s) \d s , \e N^{\e^{-1} \psi_\e}\right),
$$
then $\{Z^{\e,u_\e},K^{\e,u_{\e}}\}$ is the unique solution of the equation

\begin{align}
 Z_t^{\e,u_\e}=&x_0 + \int_0^t b_\e(Z_s^{\e,u_\e},\LL_{X^\e})\d s+\sqrt{\e} \int_0^t \sigma_\e (Z_s^{\e,u_\e},\LL_{X^\e})\d W_s \nonumber \\
&+  \int_0^t \sigma_\e (Z_s^{\e,u_\e},\LL_{X^\e}) \phi_\e(s)\d s+ \e \int_0^t \int_{Z} G_\e (Z_{s-}^{\e,u_\e},\LL_{X^\e},z)\tilde{N}^{\e^{-1} \psi_\e } (\d z,\d s)\nonumber \\
&+ \int_0^t \int_Z G_\e(Z_s^{\e,u_\e},\LL_{X^\e},z) \left( \psi_\e(s,z) -1 \right) \nu(dz) ds-K_t^{\e,u_\e},\ t\in[0,T],  \ \pp \mbox{-a.s.}
\end{align}
and

\begin{align}
&\int_0^T | b(Z_s^{\e,u_\e}, \LL_{X^\e})| \d s + \int_0^T \| \sigma(Z_s^{\e,u_\e}, \LL_{X^\e} )\|_{\rr^d \otimes \rr^d}^2 \d s \nonumber\\
&+\int_0^T | \sigma ( Z_s^{\e,u_\e}, \LL_{X^\e} ) \phi_\e(s)| \d s + \int_0^T \int_Z |G(Z_s^{\e,u_\e}, \LL_{X^\e},z)|^2 \psi_\e(s,z) \nu(\d z) \d s \nonumber\\
&+  \int_0^T \int_Z |G(Z_s^{\e,u_\e}, \LL_{X^\e},z)( \psi_\e(s,z) -1)| \nu(\d z) \d s+|K_t^{u_\e}|_0^T < +\infty, \pp \mbox{-a.s.}
\end{align}
and
$$
Z_t^{\e,u_\e} \mbox{ is } \FF_t \mbox{-adapted. }
$$

\nenu

\nlem

%In order to obtain the LDP, we need the following notations and assumptions.
%
%Set
%$$
%L^2(\nu) = \{ f :Z \to \rr|f \mbox{ is } \BB(Z)/\BB(\rr ) \mbox{-measurable and } \int_Z |f(z)|^2 \nu (\d z) < +\infty  \}
%$$
%and
%\begin{align}
%\mathcal{H} = \Big\{g: Z \to \rr_+ |&g \mbox{ is Borel measurable and there exists } c>0 \mbox{ such that } \nonumber\\
% &\int_O e^{cg^2(z)} \nu (\d z) < +\infty \mbox{ for all } O \in \BB(Z) \mbox{ with } \nu(O) <+\infty \Big\}.
%\end{align}
%
%{\bf (B1)}
%There exist $L_1,L_2,L_3 \in \HH \cap L^2(\nu)$ such that for all $ t \in [0,T]$, $x,x' \in \rr^d$, $\mu, \mu' \in \PP_2 $ and $z \in Z$,
%
%$$
%\begin{aligned}
%&| G(x,\mu,z) -G(x',\mu',z) |^2 \le L_1^2(z) \kappa \left(|x-x'|^2 +W_2^2(\mu,\mu') \right),\\
%&|G(0,\delta_0,z)| \le L_2(z)
%\end{aligned}
%$$
%and there exists nonnegative constant $\rho_{G,\e}$ converging to $0$ such that
%$$
%\sup_{(x,\mu)\in \rr^d \times \PP_2} | G_\e(x, \mu, z)-G(x,\mu,z)| \le \rho_{G,\e} L_3(z).
%$$

\blem\label{lemmaYu}
Assume that {\bf Hypothesis \ref{H11}} and {\bf (H6)} hold. Then for any $ u =(\phi,\psi) \in S $, there exists a unique solution $(Y^u, K^u)$,  $ Y^u=\{ Y^u_t\}_{t \in [0,T]} \in \DD([0,T], \overline{D(A)})$ to the following equation
\begin{align}\label{Yu}
Y^u_t =&x_0+ \int_0^t b(Y_s^u, \LL_{X_s^0}) \d s + \int_0^t \sigma (Y_s^u, \LL_{X_s^0}) \phi(s) \d s \nonumber\\
&+ \int_0^t \int_Z G(Y_s^u, \LL_{X_s^0},z)(\psi (s,z) -1) \nu (\d z) \d s -K_t^u, \ t \in [0,T].
\end{align}
Moreover, for any $ m>0$,
$$
\sup_{u=(\phi, \psi) \in S_1^m \times S_2^m } \sup_{t \in [0,T]} |Y_t^u | < + \infty.
$$
\nlem

\bprf
By the similar proof in Proposition 5.5 in \cite{LSZZ}, we can get the result. So we omit the tedious proofs here.
\nprf

We now state the main result in this subsection.

\bthm\label{LDP}
Assume that {\bf Hypothesis \ref{H11}} and {\bf Hypothesis \ref{H22}} hold.  Then $ \{ X_t^\e, \e \in (0,1], t\in [0,T] \}$ satisfy a LDP on $\DD([0,T], \overline{D(A)})$ with speed $\e $ and the rate function $I$ given by
$$
I(g) :=  \inf_{(\phi,\psi) \in S, g=Y^u}\{ Q_1(\phi) + Q_2(\psi)\},
$$
where
$$
Q_1(f) := \frac{1}{2} \int_0^T | f(s)|^2 \d s,\quad Q_2(g) =\frac{1}{2} \int_{[0,T] \times Z} \ell(g(s,z)) \nu(\d z) \d s,
$$
for $u=(\phi,\psi) \in S $, $Y^u$ is the unique solution of \eqref{Yu}. Here we use the convention that $ \inf \emptyset = +\infty$.
\nthm

\bprf
By Lemma \ref{lemmaYu}, we can define a map
$$
\Gamma^0:S \ni u = (\phi,\psi) \mapsto Y^u \in \DD ([0,T], \overline{D(A)}).
$$

For any $\e \in (0,1], m \in (0,+\infty)$ and $ u_\e=(\phi_\e, \psi_\e) \in \SS_1^m \times \SS_2^m$, consider the following controlled equation
\beq\label{controlledeq}
\begin{cases}
\d Z_t^{\e,u_\e}\in -A(Z_t^{\e,u_\e})\d t+ b_\e(Z_t^{\e,u_\e},\LL_{X_t^\e})\d t+\sqrt{\e}  \sigma_\e (Z_t^{\e,u_\e},\LL_{X_t^\e})\d W_t\\
\quad \quad \quad+  \sigma_\e (Z_t^{\e,u_\e},\LL_{X_t^\e}) \phi_\e(t)\d t+ \e  \int_{Z} G_\e (Z_{t-}^{\e,u_\e},\LL_{X_t^\e},z)\tilde{N}^{\e^{-1} \psi_\e } (\d z,\d t)\\
\quad \quad \quad+  \int_Z G_\e(Z_t^{\e,u_\e},\LL_{X_t^\e},z) \left( \psi_\e(t,z) -1 \right) \nu(dz) dt,\ t\in[0,T]; \\
Z_0^{\e,u_\e} =x_0 \in \overline{D(A)}.
\end{cases}
\nneq

By Lemma \ref{Lemma1} and the Girsanov's theorem, \eqref{controlledeq} admits a unique solution $(Z_t^{\e,u_\e}, K_t^{\e,u_\e})$ and $X^\e$ is the solution of \eqref{eq2}.

By the weak convergence method, it is sufficient to verify the following two claims:

\benu
\item[{\bf (LDP1)}]
For any given $ m \in (0,+\infty)$, let $u_n=(\phi_n, \psi_n), n \in \nn, u =(\phi,\psi) \in S_1^m \times S_2^m$ such that $u_n \to u $ in $ S_1^m \times S_2^m $ as $n \to +\infty$. Then
\beq\label{LDP1}
\lim_{n \to +\infty} \sup_{t \in [0,T]} | \Gamma^0(u_n)(t)- \Gamma^0(u)(t)|=0.
\nneq
\item[{\bf (LDP2)}]
For any given $m \in (0,+\infty)$, let $\{u_\e=(\phi_\e, \psi_\e), \e \in (0,1] \} \subset  \SS_1^m \times \SS_2^m$. Then
\beq\label{LDP2}
\lim_{\e \to 0} \ee \left( \sup_{t \in [0,T]} |Z_t^{\e,u_\e} -\Gamma^0(u_\e)(t) |^2 \right)=0.
\nneq

\nenu

\nprf

The verifications of {\bf (LDP1)} and {\bf (LDP2)} will be given in Section 4.1.

\subsection{Moderate deviation principle}

Lemma \ref{Lemma1} can also be used to establish MDP of $X^\e$ as $\e \to 0$.

Assume that $\lambda (\e)>0$, $\e>0$ satisfy
\beq\label{Lambda}
\lambda(\e) \to 0 \mbox{ and } \frac{\e}{\lambda^2 (\e)} \to 0 \mbox{ as } \e \to 0.
\nneq

Define
\begin{align*}
M^{\e}_t:=\frac{1}{\lambda(\e)}(X_t^\e-X_t^0), \quad t\in[0,T],
\end{align*}
where $X^0_t$ solves equation (\ref{eq00}), i.e.,
\beq
\begin{cases}
\d X_t^0 \in -A(X_t^0) \d t +b(X_t^0,\LL_{X_t^0})\d t,\ t\in[0,T], \\
X_0^0=x_0 \in \overline{D(A)}.
\end{cases}
\nneq

Then we consider the following multivalued SDE with jumps
\begin{align}\label{eqmdp}
\begin{cases}			\,\mathrm{d}M^{\e}_t\in\,-A(M^{\e}_t)\mathrm{d}t+\frac{1}{\lambda(\e)}\left(b_{\e}(X^\e_t,\mathcal{L}_{X_t^{\e}})
-b(X_t^0,\mathcal{L}_{X_t^0})\right)\mathrm{d}t+\frac{\sqrt{\e}}{\lambda(\e)}
\sigma_{\e}(X_t^\e,\mathcal{L}_{X_t^{\e}})\mathrm{d}W_t \\
\,
\quad\quad\quad +\frac{\e}{ \lambda(\e)} \int_Z G_\e(X_{t-}^\e, \LL_{X_t^\e},z)\tilde{N}^{\e^{-1}}(\d t, \d z),
\\
\,
M^{\e}_0=0.
\end{cases}
\end{align}

Under {\bf Hypothesis \ref{H11}} and {\bf Hypothesis \ref{H22}}, \eqref{eqmdp} has a unique strong solution (see \cite{CGLZ24}). Denote the solution by$(M^{\e}_t,\hat{K}^\e)$.

By Definition \ref{solution}, $(M^{\e}_t,\hat{K}^\e)$ is the unique solution to the following equation

\begin{align}\label{mep}
\begin{cases}
\mathrm{d}M^{\e}_t=&\frac{1}{\lambda(\e)}\left(b_{\e}(\lambda(\e)M^{\e}_t+X^0_t,\mathcal{L}_{X_t^{\e}})-b(X_t^0,\mathcal{L}_{X_t^0})\right)\mathrm{d}t   \\		&+\frac{\sqrt{\e}}{\lambda(\e)}\sigma_{\e}(\lambda(\e)M^{\e}_t+X_t^0,\mathcal{L}_{X_t^{\e}})\mathrm{d}W_t-\mathrm{d}\hat{K}^\e_t\\
&+\frac{\e}{ \lambda(\e)} \int_Z G_\e(X_{t-}^\e, \LL_{X_t^\e},z)\tilde{N}^{\e^{-1}}(\d t, \d z),\\
M^{\e}_0=0.
\end{cases}
\end{align}

Denote
$$
\mathcal{R}:=\{ \varphi:[0,T]\times \Omega \times Z \to \rr: \varphi \mbox{ is } (\mathcal{P} \otimes \mathcal{B}(Z))/\mathcal{B}(\rr)\mbox{-measurable} \}.
$$
For any given $ \e >0$ and $m \in (0,+\infty)$, denote
\begin{align*}
&S_{+,\e}^m := \{g:[0,T]\times Z \to [0,+\infty) | Q_2(g) \le m\lambda^2(\e) \},\\
&S_{\e}^m := \{\varphi:[0,T]\times Z \to \rr | \varphi=(g-1)/\lambda(\e), g \in S_{+,\e}^m  \},\\
&\mathcal{S}_{+,\e}^m := \{g \in \mathcal{R}_b|g(\cdot,\cdot,\omega) \in S_{+,\e}^m, \mbox{ for } P \mbox{-a.e. } \omega \in \Omega  \},\\
&\mathcal{S}_{\e}^m := \{\varphi \in \mathcal{R}|\varphi(\cdot,\cdot,\omega) \in S_{\e}^m, \mbox{ for } P \mbox{-a.e. } \omega \in \Omega \}.
\end{align*}
Denote $L_2(\nu_T)$ the space of all $ \mathcal{B}([0,T]) \otimes \mathcal{B}(Z)/\mathcal{B}(\rr)$ measurable functions $f$ satisfying that
$$
\interleave f \interleave_2^2:= \int_0^T \int_Z |f(s,z)|^2 \nu(\d z) \d s < +\infty.
$$
Then $(L_2(\nu_T),\interleave \cdot \interleave_2)$ is a Hilbert space. Denote by $B_2(r)$ the ball of radius $r$ centered at $0$ in $ L_2(\nu_T)$. Throughout this paper, $B_2(r)$ is equipped with the weak topology of $L_2(\nu_T)$ and therefore compact.
Suppose $g \in S_{+,\e}^m$. By Lemma 3.2 in \cite{BDG16}, there exists a constant $\kappa_2(1) >0$ (independent of $\e$) such that $ \varphi 1_{\{|\varphi| \le 1/\lambda(\e)\}} \in B_2( \sqrt{m \kappa_2(1)})$, where $ \varphi=(g-1)/\lambda(\e)$.

Let
\begin{align*}	\Upsilon^\e_{\mathcal{L}_{X^\e}}(\cdot):=\frac{1}{\lambda(\e)}\left(\Gamma^\e_{\mathcal{L}_{X^\e}}(\cdot)-X^0\right),
\end{align*}
then
\begin{enumerate}
	\item[(a)] $\Upsilon^\e_{\mathcal{L}_{X^\e}}$ is a measurable map from $C([0,T],\mathbb{R}^d)\times L_2(\nu_T)\mapsto \DD([0,T],\overline{D(A)}) $ such that
$$
M^\e=\Upsilon^\e_{\mathcal{L}_{X^\e}}\left(\sqrt{\e}W_\cdot,\e N^{\e^{-1}} \right);
$$
\item[(b)] for any $m \in (0,+\infty)$, $ u_\e = (\phi_\e, \psi_\e) \in \SS_1^m \times \SS_2^m$, let
$$
M^{\e,u_\e} := \Upsilon^\e_{\LL_{X^\e}}(\sqrt{\e} W_\cdot + \int_0^\cdot \phi_\e(s) \d s , \e N^{\e^{-1} \psi_\e}).
$$
\end{enumerate}
Since we aim to establish the MDP for $ X^\e$, it is equivalent to prove that $M^\e $ satisfies a LDP. Denote $\nabla b(x,\mu)$ as the derivative of $b(x,\mu)$ with respect to the first variable.
We need the following assumptions:
	
{\bf (C0)}
There exist $L',q'\geq 0$ such that for all $x,x'\in\mathbb{R}^{d}$,
\begin{align}\label{b0}
\|\nabla b(x,\mathcal{L}_{X^0_s})-\nabla b(x',\mathcal{L}_{X^0_s})\|_{\rr^d \otimes \rr^d}^2\leq L'(1+|x|^{q'}+|x'|^{q'})|x-x'|.
\end{align}
	
{\bf (C1)}
	\begin{align}\label{b1}
		\int_{0}^{T}\|\nabla b(X^0_t,\mathcal{L}_{X^0_t})\|_{\rr^d \otimes \rr^d}^2\mathrm{d}t<+\infty.
	\end{align}
	
{\bf (C2)}
	
	\begin{equation}\label{b2}
		\begin{aligned}
			\lim\limits_{\e\to 0}\frac{\rho_{b,\e}}{\lambda(\e)}=0,
		\end{aligned}
	\end{equation}
where $\rho_{b,\e}$  is given in {\bf (H5)}.

To overcome technical difficulties, we need to strengthen {\bf (H2)} and {\bf (H6)} to the following assumptions.
\benu
\item[{\bf (H2)'}]
The functions $b$, $\sigma$ and $G$ are continuous in $(x,\mu)$, and
\begin{align*}
&\langle  x-x',b(x,\mu)-b(x',\mu) \rangle \le L  |x-x'|^2,\\
& |b(x,\mu)-b(x,\mu')| \le LW_2(\mu, \mu'),\\
&\| \sigma(x,\mu) - \sigma (x',\mu')\|_{\rr^d \otimes \rr^d}^2 \vee \int_Z | G(x,\mu,z) -G(x',\mu',z)|^2 \nu (\d z) \le L \left(|x-x'|^2+W_2^2(\mu, \mu')\right).
\end{align*}
\item[{\bf (H6)'}]
There exist $L_1,L_2,L_3 \in \HH \cap L^2(\nu)$ such that for all $ t \in [0,T]$, $x,x' \in \rr^d$, $\mu, \mu' \in \PP_2 $ and $z \in Z$,

$$
\begin{aligned}
&| G(x,\mu,z) -G(x',\mu',z) | \le L_1(z) \left(|x-x'| +W_2(\mu,\mu') \right),\\
&|G(0,\delta_0,z)| \le L_2(z)
\end{aligned}
$$
and there exists nonnegative constant $\rho_{G,\e}$ converging to $0$ such that
$$
\sup_{(x,\mu)\in \rr^d \times \PP_2} | G_\e(x, \mu, z)-G(x,\mu,z)| \le \rho_{G,\e} L_3(z).
$$
\nenu

We know that {\bf (H2)'} $\Rightarrow$ {\bf (H2)} and {\bf (B1)'} $\Rightarrow$ {\bf (B1)}.
	
\begin{proposition}\label{nu}
Assume that {\bf Hypothesis \ref{H11}}, {\bf(C0)} and {\bf(C1)} hold. Then for any fixed $m\in(0,+\infty)$ and $u=(\phi,\psi)\in S_1^m \times B_2(m)$, there exists a unique solution $V^u=\big\{\big(V^{u}_t,\hat{K}^u_t\big), t\in[0,T]\big\}\in C([0,T],\mathbb{R}^d)$ to the following equation:
\begin{equation}\label{mdp1}
\left\{ \begin{aligned}
				&\mathrm{d}V^{u}_t=\nabla b(X^0_t,\mathcal{L}_{X^0_t})V^{u}_t\mathrm{d}t+\sigma(X_t^0,\mathcal{L}_{X_t^0})\phi(t)\mathrm{d}t\\
&\quad \quad\quad+\int_Z G(X_t^0, \LL_{X^0_t},z) \psi(t,z)\nu(\d z) \d t-\d \hat{K}^u_t,     \\
				&V^{u}_0=0.
\end{aligned}  \right.
\end{equation}
Moreover,
\begin{align}\label{nulim}
			\sup_{u \in S_1^m \times B_2(m)}\sup_{t\in[0,T]}|V^{u}_t|<+\infty.
\end{align}
\end{proposition}
		
		\begin{proof}
Since $\big\{\hat{K}^u_t, t\in[0,T]\big\}$ is of finite variation with $u \in S_1^m \times B_2(m)$, we have
$|\hat{K}^u|_0^T<+\infty.$

By {\bf (H2)}, {\bf (H3)}, {\bf (B1)} and using the fact that
\beq\label{123bound} \int_Z (L_1^2(z) +L_2^2(z)+L_3^2(z)) \nu (\d z) < + \infty,
\nneq
we can prove that
\beq\label{Gbound}
\int_0^T \|\sigma( X_s^0, \LL_{X_s^0})\|_{\rr^d \otimes \rr^d}^2 \d s+\int_0^T \int_Z | G (X_{s}^0,\LL_{X_s^0},z)|^2\nu(\d z)\d s < +\infty.
\nneq
			
By {\bf(H2)}, {\bf (H3)}, Remark \ref{WXY} and the fact that $X^0\in C([0,T],\mathbb{R}^d)$ and $u=(\phi,\psi) \in S_1^m \times B_2(m)$, we have
\begin{align}\label{22}
\int_{0}^{T}\lvert\sigma(X_t^0,\mathcal{L}_{X_t^0})\phi(t)\rvert\mathrm{d}t  \nonumber
\leq &\left(\int_{0}^{T}\lVert\sigma(X_t^0,\mathcal{L}_{X_t^0})\rVert_{\rr^d \otimes \rr^d}^2\mathrm{d}t\right)^{\frac{1}{2}}\left(\int_{0}^{T}|\phi(t)|^2\mathrm{d}t\right)^\frac{1}{2}   \nonumber\\
\leq &\left(\int_{0}^{T}\lVert\sigma(X_t^0,\mathcal{L}_{X_t^0})\rVert_{\rr^d \otimes \rr^d}^2\mathrm{d}t\right)^{\frac{1}{2}}(2m)^{\frac{1}{2}}
<+\infty
\end{align}
and
\begin{align}\label{23}
&\int_0^T \int_Z |G(X_t^0, \LL_{X_t^0},z)\psi(t,z)| \nu(\d z) \d t \nonumber\\
\le & \left( \int_0^T \int_Z |G(X_t^0, \LL_{X_t^0},z)|^2 \nu(\d z) \d t  \right)^{\frac{1}{2}} \left( \int_0^T \int_Z |\psi(t,z)|^2 \nu(\d z) \d t\right)^{\frac{1}{2}} \nonumber\\
\le & \left( \int_0^T \int_Z |G(X_t^0, \LL_{X_t^0},z)|^2 \nu(\d z) \d t  \right)^{\frac{1}{2}} m^{\frac{1}{2}} < + \infty.
\end{align}
			
Due to {\bf (C1)}, the Gronwall's inequality and the estimates above we can easily prove that the equation (\ref{mdp1}) has a unique solution $\big\{\big(V^{u}_t,\hat{K}^u_t\big), t\in[0,T]\big\}$.
The estimate (\ref{nulim}) follows by using Gronwall's inequality.			
\end{proof}

Now we state our main result about the moderate deviation principle.
\bthm\label{MDP}
Assume that Hypotheses \ref{H11} and \ref{H22} hold with condition {\bf (H2)} replaced by {\bf (H2)'} and {\bf (H6)} replaced by {\bf (H6)'}, and that conditions {\bf (C0)}, {\bf (C1)} and {\bf (C2)} are satisfied. Then $\left\{M^{\e},\e>0\right\}$ satisfies a LDP on $\DD([0,T],\overline{D(A)})$ with speed $\frac{\e}{\lambda^2(\e)}$ and the rate function $I$ given by for any $g \in  \DD([0,T],\overline{D(A)})$
\begin{align}
I(g):=\frac{1}{2}\inf_{\left\{u=(\phi,\psi) \in L^2([0,T] ,\mathbb{R}^d)\times L_2(\nu_T), V^{u}=g\right\}}\left \{ \int_{0}^{T}|\phi(s)|^2\mathrm{d}s +\int_0^T \int_Z | \psi(s,z)|^2 \nu(\d z) \d s \right\},
\end{align}
where for $u=(\phi,\psi) \in L^2([0,T] ,\mathbb{R}^d)\times L_2(\nu_T)$, $\big(V^{u},\hat{K}^u\big)$ is the unique solution of (\ref{mdp1}). Here we use the convention that $ \inf \emptyset = +\infty$.
\nthm
	
\begin{proof}
By {\bf Proposition \ref{nu}}, we can define a map
\begin{align}
\Upsilon^0:L^2([0,T],\mathbb{R}^d) \times L_2(\nu_T) \ni u=(\phi,\psi) \mapsto V^u \in \DD([0,T],\overline{D(A)}),
\end{align}
where $V^u$ is the unique solution of (\ref{nu}).
		
For any $\e>0, m\in(0,+\infty)$ and $u_{\e}=(\phi_\e, \psi_\e) \in\mathcal{S}_1^{m} \times \mathcal{S}_{+,\e}^{m} $, recall that $\left\{\left(M^{\e,u_\e}_t,\hat{K}^{\e,u_\e}_t\right)\right\}_{t\in[0,T]}$ is the solution to the following SDE:
		\begin{equation}\label{mdp2m}
			\left\{\begin{aligned}
				\mathrm{d}M^{\e,u_\e}_t=&\frac{1}{\lambda(\e)}\left(b_\e(\lambda(\e)M^{\e,u_{\e}}_t+X_t^0,\mathcal{L}_{X_t^\e})-b(X_t^0,\mathcal{L}_{X_t^0})\right)\mathrm{d}t  \\
				&+\frac{\sqrt{\e}}{\lambda(\e)}\sigma_{\e}(\lambda(\e)M^{\e,u_{\e}}_t+X^0_t,\mathcal{L}_{X_t^\e})\mathrm{d}W_t  \\
				&+\sigma_{\e}(\lambda(\e)M^{\e,u_{\e}}_t+X^0_t,\mathcal{L}_{X_t^\e})\phi_{\e}(t)\mathrm{d}t-\mathrm{d} \hat{K}^{\e,u_\e}_t \\
&+\frac{\e }{ \lambda(\e)} \int_Z G_{\e}(\lambda(\e)M^{\e,u_{\e}}_{t-}+X^0_{t-},\mathcal{L}_{X_t^\e},z ) \tilde{N}^{\e^{-1}\psi_\e} (\d z, \d t) \\
&+ \frac{1}{\lambda(\e)} \int_Z G_{\e}(\lambda(\e)M^{\e,u_{\e}}_{t}+X^0_{t},\mathcal{L}_{X_t^\e},z )(\psi_\e(t,z)-1) \nu( \d z) \d t,  \\				
\,\,\,\, M^{\e,u_{\e}}_0=&\,\,0.
			\end{aligned}\right.
		\end{equation}

Similar to the proof of LDP, it is sufficient to verify the following two claims:
		
\noindent{\bf (MDP1)} For any given $m\in(0,+\infty)$, let $\left\{u_n=(\phi_n,\psi_n), n\in\mathbb{N}\right\}, u=(\phi,\psi)\in S_1^m \times B_2(m)$ be such that $u_n\to u$ in $S_1^m \times B_2(m)$ as $n\to +\infty$.  Then
\beq\label{md1}
\lim\limits_{n\to +\infty}\sup_{t\in[0,T]}|\Upsilon^0(u_n)(t)-\Upsilon^0(u)(t)|=0.
\nneq
		
\noindent{\bf (MDP2)} For any given $m\in(0,+\infty)$, let $\left\{u_{\e}=(\phi_\e,\psi_\e),\e>0\right\}\in \mathcal{S}_1^m \times \mathcal{S}_{+,\e}^m$, and for some $\beta \in (0,1]$, $ \varphi_\e 1_{\{ |\varphi_\e| \le \beta/\lambda(\e) \}} \in B_2( \sqrt{m \kappa_2(1)})$ where $\varphi_\e= (\psi_\e -1)/ \lambda(\e)$. Set
\beq\label{md2}
\tilde{u}_\e:=\left( \phi_\e, \varphi_\e 1_{\{ |\varphi_\e| \le \beta/\lambda(\e) \}} \right).
\nneq
Then for any $\xi>0$,
$$
\lim\limits_{\e\to0}P\big(\sup_{t\in[0,T]}|M^{\e,u_{\e}}_t-\Upsilon^0(\tilde{u}_{\e})(t)|>\xi\big)=0.
$$
		
		The verifications of {\bf (MDP1)} and {\bf (MDP2)} will be given in Section 4.2.
		
	\end{proof}

\section{Proof of LDP and MDP}
In this section, we present the proofs of Theorems \ref{LDP} and \ref{MDP}. We observe that while hypothesis {\bf (H3)} naturally implies {\bf (H3)'}, the weaker condition {\bf (H3)'} suffices to establish the subsequent proofs. Therefore, in what follows, we will work under hypothesis {\bf (H3)'} rather than {\bf (H3)}, as this relaxation maintains the validity of our arguments while broadening the potential applications of the theorems.

\subsection{proof of LDP1 and LDP2}

\bprop\label{LDP1p}
For any given $ m \in (0,+\infty)$, let $u_n=(\phi_n, \psi_n), n \in \nn, u =(\phi,\psi) \in S_1^m \times S_2^m$ such that $u_n \to u $ in $ S_1^m \times S_2^m $ as $n \to +\infty$. Then
\beq\label{LDP11}
\lim_{n \to +\infty} \sup_{t \in [0,T]} | \Gamma^0(u_n)(t)- \Gamma^0(u)(t)|=0.
\nneq
\nprop

\bprf
Let $Y^u$ be the solution of \eqref{Yu} and $Y^{u_n}$ be the solution of of \eqref{Yu} with $u$ replaced by $u_n$.
By the definition of $\Gamma^0$, we have $Y^u=\Gamma^0(u)$ and $ Y^{u_n}= \Gamma^0(u_n)$, we only need to prove
$$
\lim_{n \to +\infty} \sup_{t \in [0,T]} | Y^{u_n}- Y^{u}|=0.
$$
Note that $ Y^{u},Y^{u_n}\in \DD([0,T],\overline{D(A)}), \forall n \in \nn$.
By Lemma \ref{lemmaYu}, we know that $\{ Y^{u_n}\}_{n \ge 1}$  is uniformly bounded, i.e.,
$$
\sup_{n \ge 1 } \sup_{t \in [0,T]} |Y_t^{u_n} | < + \infty.
$$

For simplicity, We denote for some constant $C>0$
\beq\label{upbound}
\max \left\{ \sup_{n \ge 1 } \sup_{t \in [0,T]} |Y_t^{u_n} |,  \sup_{t \in [0,T]} |Y_t^{u} | \right\} \le C.
\nneq

In this proof, $C$ is some positive constant independent of $n$. The value of $C$ may be different from line to line.

Set $\omega_n(t):=Y_t^{u_n}-Y_t^u $, we have
\begin{align*}
&\omega_n(t)=Y_t^{u_n}-Y_t^u=-(K_t^{u_n}-K_t^u)\\
&+ \int_0^t \left[ b(Y_s^{u_n},\LL_{X_s^0})- b(Y_s^{u},\LL_{X_s^0})\right] \d s\\
&+ \int_0^t \left[ \sigma(Y_s^{u_n},\LL_{X_s^0})\phi_n(s) - \sigma(Y_s^{u},\LL_{X_s^0})\phi(s)\right]\d s\\
&+ \int_0^t \int_Z\left[ G(Y_s^{u_n},\LL_{X_s^0},z)(\psi_n(s,z)-1)-G(Y_s^{u},\LL_{X_s^0},z)(\psi(s,z)-1) \right]\nu(\d z) \d s.
\end{align*}

By {\bf (H2)} and Proposition \ref{equivalent}, we have
\begin{align}\label{eqI0}
&|\omega_n(t)|^2=-2\int_0^t \langle \omega_n(s), \d K_s^{u_n}-\d K_s^u \rangle \nonumber \\
&+2\int_0^t\langle \omega_n(s),   b(Y_s^{u_n},\LL_{X_s^0})- b(Y_s^{u},\LL_{X_s^0}) \rangle \d s\nonumber\\
&+2\int_0^t \langle \omega_n(s), \sigma(Y_s^{u_n},\LL_{X_s^0})\phi_n(s) - \sigma(Y_s^{u},\LL_{X_s^0})\phi(s)  \rangle \d s\nonumber\\
&+ 2 \int_0^t \int_Z \langle \omega_n(s), G(Y_s^{u_n},\LL_{X_s^0},z)(\psi_n(s,z)-1)-G(Y_s^{u},\LL_{X_s^0},z)(\psi(s,z)-1)\rangle\nu(\d z) \d s \nonumber\\
\le & 2\int_0^t\langle \omega_n(s),   b(Y_s^{u_n},\LL_{X_s^0})- b(Y_s^{u},\LL_{X_s^0}) \rangle \d s\nonumber\\
&+ 2 \int_0^t \langle \omega_n(s),  \sigma(Y_s^{u},\LL_{X_s^0}) \left( \phi_n(s) - \phi(s) \right)  \rangle
\d s\nonumber\\
&+ 2 \int_0^t \langle \omega_n(s),  \left( \sigma(Y_s^{u_n},\LL_{X_s^0})-\sigma(Y_s^{u},\LL_{X_s^0}) \right)  \phi_n(s)   \rangle \d s\nonumber\\
&+2 \int_0^t \int_Z \langle \omega_n(s), \left[G(Y_s^{u_n},\LL_{X_s^0},z)-G(Y_s^{u},\LL_{X_s^0},z)\right] (\psi_n(s,z)-1)\rangle\nu(\d z) \d s \nonumber\\
&+ 2 \int_0^t \int_Z \langle \omega_n(s), G(Y_s^{u},\LL_{X_s^0},z) (\psi_n(s,z)-\psi(s,z))\rangle\nu(\d z) \d s\nonumber\\
=&:I_1(t)+I_2(t)+I_3(t)+I_4(t)+I_5(t).
\end{align}
By {\bf (H2)},
\beq\label{eqI1}
|I_1(t)|\le 2L \int_0^t \kappa\left(\sup_{r \in [0,s]}|\omega_n(r)|^2\right) \d s.
\nneq

For $I_2(t)$, since $\phi,\ \phi_n \in S_1^m$, set
$$
g_n(t):= \int_0^t \sigma(Y_s^u, \LL_{X_s^0})(\phi_n(s) - \phi(s)) \d s.
$$
We prove that $g_n(\cdot) \to 0$ as $n \to +\infty$ in $ \DD([0,T],\overline{D(A)}) $.

First of all we prove that
\begin{itemize}
\item[(1)]
$
\sup_{n\ge 1} \sup_{t \in [0,T]} |g_n(t)| <+\infty;
$\\
\item[(2)]
$ \{[0,T]\ni t \mapsto g_n(t), n \ge 1 \} $ is equi-continous.
\end{itemize}
For $0 \le s <t \le T$, by {\bf (H3)'} and \eqref{upbound}, for some positive constant $C$ independent of $n$, we have
\begin{align}
|g_n(t) - g_n(s)|&=\left| \int_s^t  \sigma(Y_r^u, \LL_{X_r^0})(\phi_n(r) - \phi(r))  \d r \right|\nonumber \\
&\le \left( \int_s^t \|  \sigma(Y_r^u, \LL_{X_r^0})\|^2_{\rr^d \otimes \rr^d}  \d r \right)^{\frac{1}{2}} \left( \int_s^t |\phi_n(r) - \phi(r)|^2 \d r  \right)^{\frac{1}{2}} \nonumber \\
&\le 2 m^{\frac{1}{2}} \left( \int_s^t \|  \sigma(Y_r^u, \LL_{X_r^0})\|^2_{\rr^d \otimes \rr^d}  \d r \right)^{\frac{1}{2}} \nonumber \\
&\le  2 m^{\frac{1}{2}} \left( \int_s^t L (1+| Y_r^u |^2+\| \LL_{X_r^0}\|_2^2 )  \d r \right)^{\frac{1}{2}} \nonumber \\& \le 2Cm^{\frac{1}{2}} \sqrt{t-s},
\end{align}
which means (2) holds.

Letting $s= 0$, we have
$$
|g_n(t)|\le 2Cm^{\frac{1}{2}} \sqrt{T} <+\infty.
$$
Then (1) holds.

Combining (1) and (2), by the Ascoli-Arzel\'{a} lemma, we get that $\{g_n,\ n \ge 1\}$ is pre-compact in
$\DD([0,T],\overline{D(A)})$.
Since $\phi_n \to \phi$ in $S_1^m$ and
$$
\int_0^t \|  \sigma(Y_r^u, \LL_{X_r^0})\|^2_{\rr^d \otimes \rr^d}  \d r <+\infty,
$$
we have for any $t \in [0,T]$, $\lim_{n \to +\infty} |g_n(t)|=0$, which implies that
\beq\label{gn}
\lim_{n \to +\infty} \sup_{t \in[0,T]}|g_n(t)|=0.
\nneq

For $I_2(t)$, by Taylor formula to $\langle \omega_n(t), g_n(t)\rangle$, we have
\begin{align*}
\frac{1}{2}I_2(t)=&\langle \omega_n(t), g_n(t)\rangle+\int_0^t \langle g_n(s), \d (K_s^{u_n}-K_s^u) \rangle \\
&- \int_0^t \left\langle g_n(s), b(Y_s^{u_n},\LL_{X_s^0})- b(Y_s^{u},\LL_{X_s^0})\right\rangle \d s\\
&- \int_0^t \left\langle g_n(s), \sigma(Y_s^{u_n},\LL_{X_s^0})\phi_n(s) - \sigma(Y_s^{u},\LL_{X_s^0})\phi(s)\right\rangle \d s\\
&- \int_0^t \int_Z\left\langle g_n(s), \left[ G(Y_s^{u_n},\LL_{X_s^0},z)(\psi_n(s,z)-1)-G(Y_s^{u},\LL_{X_s^0},z)(\psi(s,z)-1) \right] \right\rangle \nu(\d z) \d s\\
=&: I_{21}(t)+I_{22}(t)+I_{23}(t)+I_{24}(t)+I_{25}(t).
\end{align*}

Since $\sup_{t \in [0,T]} |I_{21}(t)| \le \sup_{t \in [0,T]}|g_n(t)|\sup_{t \in [0,T]}|\omega_n(t)| $,
by \eqref{upbound} and \eqref{gn} we have
\beq\label{I21}
\lim_{n \to +\infty}\sup_{t \in [0,T]} |I_{21}(t)| =0.
\nneq

Since $\{ K_s^{u_n}, t \in [0,T]\}$ and $\{ K_s^{u_n}, t \in [0,T]\} $ are of finite variation, by \eqref{gn} we have
\beq\label{I22}
\lim_{n \to +\infty}\sup_{t \in [0,T]} |I_{22}(t)| =0.
\nneq

For $I_{23}(t)$ and $I_{24}(t)$, by {\bf (H3)'}, \eqref{upbound} and  \eqref{gn}, using the same deduction to the above, we obtain that
\beq\label{I234}
\lim_{n \to +\infty}\sup_{t \in [0,T]} |I_{23}(t)| =0,\quad \lim_{n \to +\infty}\sup_{t \in [0,T]} |I_{24}(t)| =0.
\nneq

For $I_{25}(t)$,
\begin{align*}
&\lim_{n \to +\infty}\sup_{t \in [0,T]} |I_{25}(t)| \\
\le& \sup_{t \in [0,T]} |g_n(t)| \int_0^T \int_Z \left| G(Y_s^{u_n},\LL_{X_s^0},z)(\psi_n(s,z)-1)-G(Y_s^{u},\LL_{X_s^0},z)(\psi(s,z)-1) \right| \nu(\d z) \d s.
\end{align*}

Since $\kappa(\cdot)$ is concave and increasing, there must exist a positive constant $a$ such that
\beq\label{ku}
\kappa(u) \le a(1+u).
\nneq
Since $\psi,\ \psi_n \in S_2^m$, by \eqref{upbound}, \eqref{ku}, {\bf (H3)'}, {\bf (H6)} and H\"{o}lder's inequality, we have
\begin{align}\label{I25b}
&\int_0^T \int_Z \left| G(Y_s^{u_n},\LL_{X_s^0},z)(\psi_n(s,z)-1)-G(Y_s^{u},\LL_{X_s^0},z)(\psi(s,z)-1) \right| \nu(\d z) \d s\nonumber\\
\le & \int_0^T \int_Z \left| \left( G(Y_s^{u_n},\LL_{X_s^0},z)-G(Y_s^{u},\LL_{X_s^0},z)\right)(\psi_n(s,z)-1)\right| \nu(\d z) \d s \nonumber\\
&+ \int_0^T \int_Z \left|G(Y_s^{u},\LL_{X_s^0},z)\left(\psi_n(s,z)-\psi(s,z)\right) \right| \nu(\d z) \d s\nonumber\\
\le & \int_0^T \int_Z \sqrt{ \kappa \left(\left| Y_s^{u_n}-Y_s^{u}\right|^2\right)} L_1(z)\left| (\psi_n(s,z)-1)\right| \nu(\d z) \d s \nonumber\\
&+ \left(\int_0^T \int_Z \left|G(Y_s^{u},\LL_{X_s^0},z)\right|^2 \nu(\d z) \d s\right)^{\frac{1}{2}} \nonumber\\
&\times\left[ \left( \int_0^T \int_Z |\psi_n(s,z)|^2 \nu(\d z) \d s \right)^\frac{1}{2}+\left( \int_0^T \int_Z |\psi(s,z)|^2 \nu(\d z) \d s \right)^\frac{1}{2}\right]\nonumber\\
\le &   \sup_{t \in [0,T]} \sqrt{ a \left( 1+  \left| Y_t^{u_n}-Y_t^{u}\right|^2 \right)}\int_0^T \int_Z  L_1(z)\left| (\psi_n(s,z)-1)\right| \nu(\d z) \d s \nonumber\\
& +2 m^{\frac{1}{2}} \left(\int_0^T \int_Z \left|G(Y_s^{u},\LL_{X_s^0},z)\right|^2 \nu(\d z) \d s\right)^{\frac{1}{2}}.
\end{align}

By Lemma 3.4 in \cite{BCD13}, we have the following result.

For every $\theta>0$, there exists some $\beta>0$ such that for any $A \in \BB([0,T])$ with $Leb_T(A) \le \beta$,
\beq\label{L}
\sup_{i=1,2,3} \sup_{\psi \in S_2^m} \int_A \int_Z L_i(z) | \psi(s,z) -1| \nu(\d z) \d s \le \theta.
\nneq
Hence we have
\beq\label{L1}
 \sup_{\psi \in S_2^m} \int_0^T  \int_Z  L_1(z) | \psi(s,z)-1|\nu(\d z) \d s < +\infty.
\nneq
By \eqref{upbound}, \eqref{I25b}, \eqref{L1} and {\bf (H3)'},
$$
\int_0^T \int_Z \left| G(Y_s^{u_n},\LL_{X_s^0},z)(\psi_n(s,z)-1)-G(Y_s^{u},\LL_{X_s^0},z)(\psi(s,z)-1) \right| \nu(\d z) \d s <+\infty.
$$
Hence, by \eqref{gn}
\beq\label{I25}
\lim_{n \to +\infty}\sup_{t \in [0,T]} |I_{25}(t)| =0.
\nneq
Combing \eqref{I21}, \eqref{I22}, \eqref{I234} and \eqref{I25}, we obtain that
\beq \label{I2}
\lim_{n \to +\infty}\sup_{t \in [0,T]} |I_{2}(t)| =0.
\nneq

For $I_3(t)$, by {\bf (H2)}, Young's inequality and the definition of $S_1^m$, we have
\begin{align}\label{eqI3}
|I_3(t)|=& 2 \left|\int_0^t \langle \omega_n(s),  \left( \sigma(Y_s^{u_n},\LL_{X_s^0})-\sigma(Y_s^{u},\LL_{X_s^0}) \right)  \phi(s)   \rangle \d s\right|\nonumber\\
\le & 2 \int_0^t |\omega_n(s)| \sqrt{  \kappa \left( |\omega_n(s)|^2\right)} |\phi_n(s)| \d s \nonumber\\
\le & 2 \left( \int_0^t |\omega_n(s)|^2  \kappa \left( |\omega_n(s)|^2\right) \d s \right)^{\frac{1}{2}} \left(\int_0^t |\phi_n(s)|^2\d s\right)^{\frac{1}{2}}\nonumber\\
\le & 2 \sqrt{ m} \left( \sup_{s\in [0,t]} |\omega_n(s)|^2 \right)^{\frac{1}{2}} \left( \int_0^t \kappa \left( |\omega_n(s)|^2\right)\d s \right)^{\frac{1}{2}}\nonumber\\
\le & \frac{1}{4} \sup_{s \in [0,t]} |\omega_n(s)|^2+C \int_0^t  \kappa \left( \sup_{r \in [0,s]}|\omega_n(r)|^2\right) \d s.
\end{align}

For $I_4(t)$, by {\bf (H6)}, \eqref{L1}, Young's inequality, we have
\begin{align}\label{eqI4}
|I_4(t)|= & 2 \left|\int_0^t \int_Z \langle \omega_n(s), \left[G(Y_s^{u_n},\LL_{X_s^0},z)-G(Y_s^{u},\LL_{X_s^0},z)\right] (\psi_n(s,z)-1)\rangle\nu(\d z) \d s\right|\nonumber\\
\le & 2\int_0^t \int_Z  L_1(z)|\omega_n(s)|  \sqrt{  \kappa \left( |\omega_n(s)|^2\right)}
|\psi_n(s,z)-1|\nu(\d z) \d s\nonumber\\
\le & 2\int_0^t \left( \eta_1 |\omega_n(s)|^2 +C \kappa \left( |\omega_n(s)|^2\right) \right)\int_Z  L_1(z)|\psi_n(s,z)-1|\nu(\d z) \d s \nonumber\\
\le & \frac{1}{4} \sup_{s \in [0,t]} |\omega_n(s)|^2+  C \int_0^t  \kappa \left(\sup_{r \in [0,s]} |\omega_n(r)|^2\right) \int_Z  L_1(z)|\psi_n(s,z)-1|\nu(\d z) \d s,
\end{align}
the last inequality holds since we can choose $\eta_1$ small enough such that $ \eta_1 \int_0^t \int_Z L_1(z)|\psi_n(s,z)-1|\nu(\d z) \d s < \frac{1}{4}$.

For $I_5(t)$, since $\psi,\ \psi_n \in S_2^m$, set
$$
h_n(t):= \int_0^t \int_Z G(Y_s^{u},\LL_{X_s^0},z) (\psi_n(s,z)-\psi(s,z)) \nu(\d z)\d s.
$$
By the similar proof to \eqref{gn}, we can obtain that $h_n(\cdot) \to 0$ as $n \to +\infty$ in $ \DD([0,T],\overline{D(A)}) $.

For $I_5(t)$, by Taylor formula to $\langle \omega_n(t), h_n(t)\rangle$, we have
\begin{align*}
\frac{1}{2}I_5(t)=&\langle \omega_n(t), h_n(t)\rangle+\int_0^t \langle h_n(s), \d (K_s^{u_n}-K_s^u) \rangle \\
&- \int_0^t \left\langle h_n(s), b(Y_s^{u_n},\LL_{X_s^0})- b(Y_s^{u},\LL_{X_s^0})\right\rangle \d s\\
&- \int_0^t \left\langle h_n(s), \sigma(Y_s^{u_n},\LL_{X_s^0})\phi_n(s) - \sigma(Y_s^{u},\LL_{X_s^0})\phi(s)\right\rangle \d s\\
&- \int_0^t \int_Z\left\langle h_n(s), \left[ G(Y_s^{u_n},\LL_{X_s^0},z)(\psi_n(s,z)-1)-G(Y_s^{u},\LL_{X_s^0},z)(\psi(s,z)-1) \right] \right\rangle \nu(\d z) \d s\\
=&: I_{51}(t)+I_{52}(t)+I_{53}(t)+I_{54}(t)+I_{55}(t).
\end{align*}
By the similar deduction to \eqref{I2}, we can obtain that
\beq \label{I5}
\lim_{n \to +\infty}\sup_{t \in [0,T]} |I_{5}(t)| =0.
\nneq

Combining \eqref{eqI0}, \eqref{eqI1}, \eqref{eqI3} and \eqref{eqI4}, we have
\begin{align}
|\omega_n(t)|^2 \le & \frac{1}{2} \sup_{s \in [0,t]} |\omega_n(s)|^2+(2L+C) \int_0^t  \kappa \left( \sup_{r \in [0,s]}|\omega_n(r)|^2\right) \d s\nonumber\\
&+C\int_0^t \kappa \left( \sup_{r \in [0,s]}|\omega_n(r)|^2\right) \int_Z  L_1(z) | \psi_n(s,z)-1|\nu(\d z) \d s\nonumber\\
&+I_2(t)+I_5(t),
\end{align}
which  implies that
\begin{align}
&\sup_{t \in [0,T]}|\omega_n(t)|^2 \nonumber\\
\le&  \int_0^T  \kappa \left( \sup_{r \in [0,s]}|\omega_n(r)|^2\right) \left(2(2L+C)+2C\int_Z  L_1(z) | \psi_n(s,z)-1|\nu(\d z) \right) \d s\nonumber\\
&+2\sup_{t \in [0,T]}I_2(t)+2\sup_{t \in [0,T]}I_5(t)\nonumber\\
=&:  \int_0^T  \kappa \left( \sup_{r \in [0,s]}|\omega_n(r)|^2\right) \left(2(2L+C)+2C\int_Z  L_1(z) | \psi_n(s,z)-1|\nu(\d z) \right) \d s+O_1(n),
\end{align}
where
\beq\label{O1}
O_1(n) \to 0 \mbox{ as } n \to +\infty.
\nneq

Setting $f(t)=\int_1^t \frac{1}{\kappa(s)} \d s $, it follows from Lemma \ref{Bihari} that
\begin{align}
&\sup_{t \in [0,T]} | Y^{u_n}- Y^{u}|^2 \\
\le & f^{-1}\left( f(O_1(n)) + \int_0^T \left(2(2L+C)+2C\int_Z  L_1(z) | \psi_n(s,z)-1|\nu(\d z) \right) \d s   \right).
\end{align}
By \eqref{L1}, we have
$$
\int_0^T \left(2(2L+C)+2C\int_Z  L_1(z) | \psi_n(s,z)-1|\nu(\d z) \right) \d s<+\infty.
$$

Recalling the condition $ \int_{0+} \frac{1}{\kappa(s)} \d s=+\infty $, by \eqref{O1} we can conclude that
$$
f(O_1(n))+ \int_0^T \left(2(2L+C)+2C\int_Z  L_1(z) | \psi_n(s,z)-1|\nu(\d z) \right) \d s  \to -\infty \mbox{ as } n \to +\infty.
$$
On the other hand, because $f$ is a strictly increasing function, then we obtain that $f$ has an inverse function which is strictly increasing and $f^{-1}(-\infty)=0$. Thus,
$$
f^{-1}\left(  f(O_1(n))+ \int_0^T \left(2(2L+C)+2C\int_Z  L_1(z) | \psi_n(s,z)-1|\nu(\d z) \right) \d s \right)\to 0 \mbox{ as  } n \to +\infty.
$$
Hence, we have
$$
\lim_{n \to +\infty} \sup_{t \in [0,T]} | Y^{u_n}- Y^{u}|^2=0,
$$
then
$$
\lim_{n \to +\infty} \sup_{t \in [0,T]} | Y^{u_n}- Y^{u}|=0,
$$
which is the desired result.
\nprf

To verify {\bf {(LDP2)}}, we need the following result.

\blem\label{LDPlem}
Under {\bf (H1-H3)}, {\bf (H5)} and {\bf (H6)},
\beq
\lim_{\e \to 0}\ee  \left(\sup_{t \in [0,T]} | X_t^\e-X_t^0|^2 \right)= 0.
\nneq
\nlem

In the following two proofs, $C$ is the positive constant independent of $\e$. The value of $C$ may be different from line to line.

\bprf
Note that
\begin{align}
X_t^\e - X_t^0=&-(K_t^\e - K_t^0)+\int_0^t \left( b_\e( X_s^\e, \LL_{X_s^\e})-b( X_s^0, \LL_{X_s^0})\right) \d s \nonumber\\
&+\sqrt{\e} \int_0^t \sigma_\e (X_s^\e,\LL_{X_s^\e})\d W_s+ \e \int_0^t \int_{Z} G_\e (X_{s-}^\e,\LL_{X_s^\e},z)\tilde{N}^{\e^{-1}} (\d z,\d s).
\end{align}
By It\^{o}'s formula,
\begin{align}\label{J0}
&|X_t^\e - X_t^0|^2\\
=& -2\int_0^t \langle X_s^\e -X_s^0, dK_s^\e -dK_s^0  \rangle \nonumber \\
&+2\int_0^t \langle X_s^\e -X_s^0,  b_\e( X_s^\e, \LL_{X_s^\e})-b( X_s^0, \LL_{X_s^0}) \rangle \d s\nonumber \\
&+2\sqrt{\e}\int_0^t  \langle X_s^\e -X_s^0,  \sigma_\e( X_s^\e, \LL_{X_s^\e}) \d W_s\rangle \nonumber \\
&+2\e \int_0^t \int_Z  \langle X_{s-}^\e -X_{s-}^0,  G_\e (X_{s-}^\e,\LL_{X_s^\e},z) \rangle\tilde{N}^{\e^{-1}} (\d z,\d s) \nonumber \\
&+\e \int_0^t \| \sigma_\e( X_s^\e, \LL_{X_s^\e}) \|_{\rr^d \otimes \rr^d}^2 \d s \nonumber \\
&+ \e^2 \int_0^t \int_Z | G_\e (X_{s-}^\e,\LL_{X_s^\e},z)  |^2 N^{\e^{-1}} (\d z,\d s) \nonumber \\
=&: J_1(t)+J_2(t)+J_3(t)+J_4(t)+J_5(t)+J_6(t).
\end{align}

For $J_1(t)$, by Definition \ref{solution}, we have
\beq\label{J1}
J_1(t) \le 0.
\nneq

For $J_2(t)$, by {\bf (H2)}, {\bf (H5)}, H\"{o}lder's inequality, Young's inequlity and Remark \ref{WXY}, we have for any $\eta_2 >0$,
\begin{align}\label{J2}
J_2(t)=&2\int_0^t \langle X_s^\e -X_s^0,  b_\e( X_s^\e, \LL_{X_s^\e})-b( X_s^0, \LL_{X_s^0}) \rangle \d s \nonumber \\
\le& 2 \int_0^t \langle X_s^\e -X_s^0,  b_\e( X_s^\e, \LL_{X_s^\e})-b( X_s^\e, \LL_{X_s^\e}) \rangle \d s\nonumber \\
&+ 2 \int_0^t \langle X_s^\e -X_s^0,  b( X_s^\e, \LL_{X_s^\e})-b( X_s^0, \LL_{X_s^0}) \rangle \d s\nonumber \\
\le& 2 \rho_{b,\e} \int_0^T | X_s^\e-X_s^0| \d s +2 \int_0^T \left( \kappa\left(| X_s^\e-X_s^0|^2 \right)+ \kappa \left(W_2^2( \LL_{X_s^\e}, \LL_{X_s^0})\right)  \right)\d s \nonumber \\
\le&  \eta_2 \int_0^T  | X_s^\e-X_s^0|^2 \d s +2  \int_0^T  \kappa\left(| X_s^\e-X_s^0|^2\right) \d s +2  \int_0^T  \kappa\left(\ee \left(| X_s^\e-X_s^0|^2\right)\right) \d s+ C \rho_{b,\e}^2 T \nonumber \\
\le &  \eta_2 \int_0^T  | X_s^\e-X_s^0|^2 \d s+2  \int_0^T  \kappa\left(| X_s^\e-X_s^0|^2\right) \d s \nonumber \\
&+2 \int_0^T  \ee \left(   \kappa\left(| X_s^\e-X_s^0|^2\right)  \right)  \d s+ C \rho_{b,\e}^2 T.
\end{align}

Hence,
\begin{align}\label{J11}
&\ee \left( \sup_{t \in [0,T]} J_2(t)\right) \nonumber\\
\le& \eta_2 T  \ee \left(\sup_{t \in [0,T]}  | X_t^\e-X_t^0|^2 \right)+ 4 \ee \int_0^T  \kappa\left(\sup_{r \in [0,s]}| X_r^\e-X_r^0|^2\right) \d s + C \rho_{b,\e}^2 T.
\end{align}

For $J_5(t)$, by {\bf (H2)}, {\bf (H5)}, \eqref{ku} and Remark \ref{WXY}, we have
\begin{align}\label{J5}
\ee \left( \sup_{t \in [0,T]} J_5(t)\right)\le &\e \ee \int_0^T \| \sigma_\e( X_s^\e, \LL_{X_s^\e}) \|_{\rr^d \otimes \rr^d}^2 \d s \nonumber\\
\le& C \e \ee \int_0^T \|\sigma_\e(X_s^\e, \LL_{X_s^\e})-\sigma(X_s^\e, \LL_{X_s^\e})\|_{\rr^d \otimes \rr^d}^2 \d s \nonumber\\
&+C \e \ee \int_0^T \|\sigma( X_s^\e, \LL_{X_s^\e})-\sigma( X_s^0, \LL_{X_s^0})\|_{\rr^d \otimes \rr^d}^2 \d s \nonumber\\
&+ C \e \int_0^T \|\sigma( X_s^0, \LL_{X_s^0})\|_{\rr^d \otimes \rr^d}^2 \d s \nonumber\\
\le& C \e \rho_{\sigma,\e}^2 T\nonumber \\
&+C  \e \ee \int_0^T  \left(\kappa \left(  | X_s^\e-X_s^0|^2 \right) +\kappa \left( W_2^2( \LL_{X_s^\e}, \LL_{X_s^0}) \right) \right) \d s \nonumber\\
&+ C\e \int_0^T \|\sigma( X_s^0, \LL_{X_s^0})\|_{\rr^d \otimes \rr^d}^2 \d s \nonumber \\
\le& C \e \rho_{\sigma,\e}^2 T+ 2 C T\e \ee \left(\sup_{t \in [0,T]} | X_t^\e-X_t^0|^2 \right) +C a T\e.
\end{align}

For $J_3(t)$, by Burkholeder-Davis-Gundy's inequality, Young's inequality and \eqref{J5}, we have for any $\eta_3>0$
\begin{align}\label{J3}
\ee \left( \sup_{t \in [0,T]} J_3(t)\right)\le & C \sqrt{\e} \ee \left( \int_0^T  |X_s^\e -X_s^0|^2 \|\sigma_\e( X_s^\e, \LL_{X_s^\e}) \|_{\rr^d \otimes \rr^d}^2 \d s \right)^\frac{1}{2}\nonumber\\
\le & \eta_3 \ee \left( \sup_{t \in[0,T]} |X_t^\e -X_t^0|^2 \right)  +C \e \ee \int_0^T \|\sigma_\e( X_s^\e, \LL_{X_s^\e}) \|_{\rr^d \otimes \rr^d}^2 \d s \nonumber\\
\le&  C\e \rho_{\sigma,\e}^2 T+ (\eta_3+2 C T\e) \ee \left(\sup_{t \in [0,T]} | X_t^\e-X_t^0|^2 \right) +C a T\e.
\end{align}

For $J_6(t)$, by {\bf (H6)}, \eqref{ku}, \eqref{123bound}, \eqref{Gbound} and Remark \ref{WXY}, we have
\begin{align}\label{J6}
&\ee \left( \sup_{t \in [0,T]} J_6(t)\right)\nonumber\\
=& \e \ee \left(  \int_0^T \int_Z | G_\e (X_{s}^\e,\LL_{X_s^\e},z)  |^2  \nu(\d z) \d s \right) \nonumber\\
\le & C \e \ee \left( \int_0^T \int_Z | G_\e (X_{s}^\e,\LL_{X_s^\e},z)-G (X_{s}^\e,\LL_{X_s^\e},z)  |^2  \nu(\d z) \d s  \right) \nonumber\\
&+ C \e \ee \left( \int_0^T \int_Z | G (X_{s}^\e,\LL_{X_s^\e},z)-G (X_{s}^0,\LL_{X_s^0},z)  |^2  \nu(\d z) \d s \right)\nonumber \\
&+   C \e   \int_0^T \int_Z | G (X_{s}^0,\LL_{X_s^0},z)  |^2  \nu(\d z) \d s \nonumber\\
\le& C \e \rho_{G ,\e}^2 T \int_Z L_3^2(z) \nu(\d z) \nonumber \\
&+ C \e \ee \int_0^T \int_Z L_1^2(z) \left( \kappa \left(   | X_s^\e-X_s^0|^2 \right)   + \kappa \left( W_2^2( \LL_{X_s^\e}, \LL_{X_s^0})\right) \right) \nu(\d z) \d s\nonumber \\
&+ C \e  \int_0^T \int_Z | G (X_{s}^0,\LL_{X_s^0},z)  |^2  \nu(\d z) \d s \nonumber\\
\le & C \e \rho_{G ,\e}^2 T \int_Z L_3^2(z) \nu(\d z)\nonumber\\
&+ C \e T \ee  \left(\sup_{t \in [0,T]} | X_t^\e-X_t^0|^2 \right) \int_Z L_1^2(z) \nu (\d z) +C \e  \nonumber \\
\le & C \e \rho_{G ,\e}^2 T + C \e T \ee  \left(\sup_{t \in [0,T]} | X_t^\e-X_t^0|^2 \right) +C \e
\end{align}

For $J_4(t)$, by Burkholeder-Davis-Gundy's inequality, Young's inequality, \eqref{123bound} and \eqref{J6}, we have
for any $\eta_4>0$,
\begin{align}\label{J4}
\ee \left( \sup_{t \in [0,T]} J_4(t)\right)\le & C\e \ee  \left(\int_0^T \int_Z | X_{s-}^\e -X_{s-}^0|^2 | G_\e (X_{s-}^\e,\LL_{X_s^\e},z)|^{2} N^{\e^{-1}} (\d z,\d s) \right)^{\frac{1}{2}} \nonumber\\
\le&\eta_4 \ee \left(\sup_{t \in [0,T]} | X_t^\e-X_t^0|^2 \right)+C \e \ee \left(\int_0^T \int_Z  | G_\e (X_{s}^\e,\LL_{X_s^\e},z)|^{2} \nu(\d z)(\d s) \right)\nonumber\\
\le& \left( \eta_4+C \e T \int_Z L_1^2(z) \nu (\d z) \right) \ee  \left(\sup_{t \in [0,T]} | X_t^\e-X_t^0|^2 \right) \nonumber\\
&+C \e \rho_{G ,\e}^2 T \int_Z L_3^2(z) \nu(\d z)+C \e  \nonumber \\
\le & \left( \eta_4+C \e T  \right) \ee  \left(\sup_{t \in [0,T]} | X_t^\e-X_t^0|^2 \right) +C \e \rho_{G ,\e}^2 T +C \e.
\end{align}

Combining \eqref{J0}-\eqref{J4}, we have
\begin{align}\label{Jzong}
&\left(1-\eta_2 T -\eta_3-\eta_4-4CT\e -2C\e T \right) \ee  \left(\sup_{t \in [0,T]} | X_t^\e-X_t^0|^2 \right)\nonumber\\
\le &  4\int_0^T \ee \left( \kappa \left( \sup_{r \in [0,s]}| X_r^\e-X_r^0|^2 \right) \right) \d s +C \rho_{b,\e}^2 T +C\e \rho_{\sigma,\e}^2 T+C \e \rho_{G,\e}^2+C \e.
\end{align}

We can choose $\eta_2, \eta_3, \eta_4$ and $\e_0>0$ small enough such that, for any $\e \in (0,\e_0]$,
\beq\label{1/5}
1-\eta_2 T -\eta_3-\eta_4-4CT\e -2C\e T \ge C_0 \ge \frac{1}{5}.
\nneq
Hence, we obtain
\begin{align}\label{Jzong'}
&\frac{1}{5} \ee  \left(\sup_{t \in [0,T]} | X_t^\e-X_t^0|^2 \right)\nonumber\\
\le &  4\int_0^T \ee \left( \kappa \left( \sup_{r \in [0,s]}| X_r^\e-X_r^0|^2 \right) \right) \d s+C \rho_{b,\e}^2 T +C\e \rho_{\sigma,\e}^2 T+C \e \rho_{G,\e}^2+C \e\nonumber\\
=&:4\int_0^T \ee \left( \kappa \left( \sup_{r \in [0,s]}| X_r^\e-X_r^0|^2 \right) \right) \d s+O_2(\e),
\end{align}
where
\beq\label{O2}
O_2(\e) \to 0 \mbox{ as } \e \to 0.
\nneq

Setting $f(t)=\int_1^t \frac{1}{\kappa(s)} \d s $, it follows from Lemma \ref{Bihari} that
$$
\ee  \left(\sup_{t \in [0,T]} | X_t^\e-X_t^0|^2 \right) \le f^{-1}\left(  f(O_2(\e))+ 4T\right).
$$

Recalling the condition $ \int_{0+} \frac{1}{\kappa(s)} \d s=+\infty $, by \eqref{O2} we can conclude that
$$
f(O_2(\e))+ 4T \to -\infty \mbox{ as } \e \to 0.
$$
On the other hand, because $f$ is a strictly increasing function, then we obtain that $f$ has an inverse function which is strictly increasing, and $f^{-1}(-\infty)=0$. Thus,
$$
f^{-1}\left(  f(O_2(\e))+ 4T\right)\to 0 \mbox{ as  } \e \to 0.
$$
Hence, we have the desired result
\beq
\lim_{\e \to 0}\ee  \left(\sup_{t \in [0,T]} | X_t^\e-X_t^0|^2 \right)=0.
\nneq
\nprf

Next we will verify {\bf (LDP2)}.

\bprop
For any given $m \in (0,+\infty)$, let $\{u_\e=(\phi_\e, \psi_\e), \e \in (0,1] \} \subset  \SS_1^m \times \SS_2^m$. Then
\beq\label{LDP21}
\lim_{\e \to 0} \ee \left( \sup_{t \in [0,T]} |Z_t^{\e,u_\e} -\Gamma^0(u_\e)(t) |^2 \right)=0.
\nneq
\nprop

\bprf
Let $Y^{u_\e}$ be the solution of of \eqref{Yu} with $u$ replaced by $u_\e$, then $ \Gamma^0(u_\e)=Y^{u_\e}$. Note that
\begin{align}
&Z_t^{\e,u_\e}-Y_t^{u_\e}  \nonumber\\
=& -(K_t^{\e, u_\e}-K_t^{u_\e})+ \int_0^t \left( b_\e(Z_s^{\e,u_\e},\LL_{X_s^\e}) - b(Y_s^{u_\e}, \LL_{X_s^0})\right)\d s \nonumber\\
&+ \int_0^t \sqrt{\e}  \sigma_\e (Z_s^{\e,u_\e},\LL_{X_s^\e})\d W_s \nonumber\\
&+ \int_0^t \left( \sigma_\e (Z_s^{\e,u_\e},\LL_{X_s^\e})- \sigma (Y_s^{u_\e},\LL_{X_s^0})\right) \phi_\e(s) \d s \nonumber\\
&+\e \int_0^t \int_{Z} G_\e (Z_{s-}^{\e,u_\e},\LL_{X_s^\e},z)\tilde{N}^{\e^{-1} \psi_\e} (\d z,\d s) \nonumber\\
&+\int_0^t \int_Z \left( G_\e (Z_{s}^{\e,u_\e},\LL_{X_s^\e},z)- G (Y_{s}^{u_\e},\LL_{X_s^0},z)\right)(\psi_\e (s,z)-1) \nu(\d z) \d s.
\end{align}

By It\^{o}'s formula, we have
\begin{align}\label{M0}
&|Z_t^{\e,u_\e}-Y_t^{u_\e}|^2 \nonumber\\
=& -2 \int_0^t \langle Z_s^{\e,u_\e}-Y_s^{u_\e},\d K_s^{\e, u_\e}-\d K_s^{u_\e} \rangle \nonumber\\
&+ 2\int_0^t  \langle Z_s^{\e,u_\e}-Y_s^{u_\e} ,b_\e(Z_s^{\e,u_\e},\LL_{X_s^\e}) - b(Y_s^{u_\e}, \LL_{X_s^0}) \rangle\d s \nonumber\\
&+ 2 \sqrt{\e} \int_0^t \langle Z_s^{\e,u_\e}-Y_s^{u_\e}, \sigma_\e (Z_s^{\e,u_\e},\LL_{X_s^\e})\d W_s \rangle \nonumber\\
&+ 2 \int_0^t \langle Z_s^{\e,u_\e}-Y_s^{u_\e},\left( \sigma_\e (Z_s^{\e,u_\e},\LL_{X_s^\e})- \sigma (Y_s^{u_\e},\LL_{X_s^0})\right) \phi_\e(s) \rangle \d s \nonumber\\
&+2\e \int_0^t \int_{Z} \langle Z_{s-}^{\e,u_\e}-Y_{s-}^{u_\e},  G_\e (Z_{s-}^{\e,u_\e},\LL_{X_s^\e},z)\rangle \tilde{N}^{\e^{-1} \psi_\e } (\d z,\d s) \nonumber\\
&+2 \int_0^t \int_Z \langle Z_s^{\e,u_\e}-Y_s^{u_\e}, \left( G_\e (Z_{s}^{\e,u_\e},\LL_{X_s^\e},z)- G (Y_{s}^{u_\e},\LL_{X_s^0},z)\right)(\psi_\e (s,z)-1) \rangle \nu(\d z) \d s \nonumber\\
&+ \e \int_0^t \| \sigma_\e (Z_s^{\e,u_\e},\LL_{X_s^\e}) \|_{\rr^d \otimes \rr^d}^2 \d s \nonumber\\
&+ \e^2 \int_0^t \int_Z |G_\e (Z_{s}^{\e,u_\e},\LL_{X_s^\e},z) |^2 N^{\e^{-1} \psi_\e } (\d z,\d s) \nonumber\\
=&: M_1(t)+M_2(t)+M_3(t)+M_4(t)+M_5(t)+M_6(t)+M_7(t)+M_8(t).
\end{align}

For $M_1(t)$, by Definition \ref{solution}, we have
\beq
M_1(t)\le 0.
\nneq

For $M_2(t)$, similar to the proof of \eqref{J2}, by {\bf (H2)}, {\bf (H5)}, H\"{o}lder's inequality, Young's inequlity and Remark \ref{WXY}, we have for any $\eta_5>0$,
\begin{align}
&\sup_{t \in [0,T]} |M_2(t)| \nonumber\\
\le & 2\int_0^T  \langle Z_s^{\e,u_\e}-Y_s^{u_\e} ,b_\e(Z_s^{\e,u_\e},\LL_{X_s^\e}) - b(Z_s^{\e,u_\e},\LL_{X_s^\e}) \rangle\d s \nonumber\\
&+2\int_0^T  \langle Z_s^{\e,u_\e}-Y_s^{u_\e} ,b(Z_s^{\e,u_\e},\LL_{X_s^\e}) - b(Y_s^{u_\e}, \LL_{X_s^0}) \rangle\d s \nonumber\\
\le & 2 \rho_{b,\e} \int_0^T | Z_s^{\e,u_\e}-Y_s^{u_\e} | \d s +2 \int_0^T\left( \kappa \left(|Z_s^{\e,u_\e}-Y_s^{u_\e}|^2\right)+ \kappa \left(W_2 ( \LL_{X_s^\e},\LL_{X_s^0}) \right) \right) \d s \nonumber \\
\le & \eta_5  \int_0^T |Z_s^{\e,u_\e}-Y_s^{u_\e} |^2 \d s +2 \int_0^T  \kappa \left(|Z_s^{\e,u_\e}-Y_s^{u_\e}|^2 \right) \d s \nonumber\\
 &+2 \int_0^T  \kappa \left( \ee \left( |X_s^\e -X_s^0|^2\right)\right) \d s+C  \rho_{b,\e}^2 T\nonumber\\
\le & \eta_5  \int_0^T |Z_s^{\e,u_\e}-Y_s^{u_\e} |^2 \d s +2 \int_0^T  \kappa \left(|Z_s^{\e,u_\e}-Y_s^{u_\e}|^2 \right) \d s \nonumber\\
 &+2 \int_0^T \kappa \left( \ee \left(  |X_s^\e -X_s^0|^2\right) \right)\d s+ C \rho_{b,\e}^2 T.
\end{align}

Hence,
\begin{align}
E \left( \sup_{t \in [0,T]} |M_2(t)|\right) \le& \eta_5 T  \ee \left( \sup_{t\in [0,T]}|Z_t^{\e,u_\e}-Y_t^{u_\e} |^2 \right) \nonumber \\
&+2\ee \int_0^T \kappa \left( \sup_{r \in[0,s]} |Z_r^{\e,u_\e}-Y_r^{u_\e}|^2 \right) \d s \nonumber\\
&+2T\kappa \left(  \ee \left( \sup_{t \in[0,T]} |X_t^\e -X_t^0|^2\right) \right) + C \rho_{b,\e}^2 T.
\end{align}

For $M_3(t)$ and $M_7(t)$, by Burkholeder-Davis-Gundy's inequality and Young's inequality, for any $\eta_6>0$, we have
\begin{align}
&\ee \left(\sup_{t \in [0,T]} |M_3(t)|+ \sup_{t \in [0,T]} |M_7(t)|\right) \nonumber\\
\le & \eta_6 T \ee \left( \sup_{t\in [0,T]} |Z_t^{\e,u_\e} -Y_t^{u_\e} |^2  \right)+ C \e \ee \int_0^T \| \sigma_\e (Z_s^{\e,u_\e},\LL_{X_s^\e})\|_{\rr^d \otimes \rr^d}^2 \d s  \nonumber\\
\le & \eta_6 T \ee \left( \sup_{t\in [0,T]} |Z_t^{\e,u_\e} -Y_t^{u_\e} |^2  \right)+C \e \ee \int_0^T \| \sigma_\e (Z_s^{\e,u_\e},\LL_{X_s^\e})-\sigma (Z_s^{\e,u_\e},\LL_{X_s^\e}) \|_{\rr^d \otimes \rr^d}^2 \d s\nonumber\\
&+C \e \ee \int_0^T \| \sigma (Z_s^{\e,u_\e},\LL_{X_s^\e})-\sigma (Y_s^{u_\e},\LL_{X_s^0})\|_{\rr^d \otimes \rr^d}^2\d s\nonumber \\
&+C \e \ee \int_0^T \| \sigma (Y_s^{u_\e},\LL_{X_s^0}) \|_{\rr^d \otimes \rr^d}^2 \d s\nonumber\\
\le & \eta_6 T \ee \left( \sup_{t\in [0,T]} |Z_t^{\e,u_\e} -Y_t^{u_\e} |^2  \right)+ C \e \rho_{\sigma, \e }^2 T \nonumber \\
&+C\e \ee \int_0^T \kappa \left(  |Z_s^{\e,u_\e}-Y_s^{u_\e}|^2 +W_2^2(\LL_{X_s^\e} ,\LL_{X_s^0})\right) \d s \nonumber\\
&+C \e  \ee \int_0^T \| \sigma (Y_s^{u_\e},\LL_{X_s^0}) \|_{\rr^d \otimes \rr^d}^2 \d s \nonumber \\
\le & \eta_6 T \ee \left( \sup_{t\in [0,T]} |Z_t^{\e,u_\e} -Y_t^{u_\e} |^2  \right)+ C \e \rho_{\sigma, \e }^2 T \nonumber \\
&+Ca\e T\ee \left( \sup_{ t \in [0,T]}   |Z_t^{\e,u_\e}-Y_t^{u_\e}|^2 \right)   +Ca\e T \ee \left( \sup_{t \in[0,T]} |X_t^\e -X_t^0|^2 \right)+CaT\e.
\end{align}

For $M_4(t)$, by {\bf (H2)}, H\"{o}lder's inequality and Young's inequality, for any $\eta_7>0 $ we have
\begin{align}
&\ee \left(\sup_{t \in [0,T]}| M_4(t)| \right) \nonumber\\
\le & 2 \ee \int_0^T \langle   Z_s^{\e,u_\e}-Y_s^{u_\e},\left( \sigma_\e (Z_s^{\e,u_\e},\LL_{X_s^\e})- \sigma (Z_s^{\e,u_\e},\LL_{X_s^\e}) \right) \phi_\e(s)  \rangle \d s \nonumber\\
+& 2 \ee \int_0^T \langle   Z_s^{\e,u_\e}-Y_s^{u_\e},\left( \sigma (Z_s^{\e,u_\e},\LL_{X_s^\e})- \sigma (Y_s^{u_\e},\LL_{X_s^0})\right) \phi_\e(s)  \rangle \d s \nonumber\\
\le & 2 \rho_{\sigma, \e } \ee \int_0^T |Z_s^{\e,u_\e}-Y_s^{u_\e}| | \phi_\e (s)| \d s \nonumber\\
&+ 2 \ee \int_0^T |Z_s^{\e,u_\e}-Y_s^{u_\e} |\left( \kappa\left( |Z_s^{\e,u_\e}-Y_s^{u_\e}|^2 \right) + \kappa \left(W_2^2(\LL_{X_s^{\e}}, \LL_{X_s^0} ) \right) \right)^{\frac{1}{2}} |\phi_\e (s)| \d s \nonumber\\
\le &  CT\rho_{\sigma,\e} \ee \left( \sup_{t\in [0,T]} |Z_t^{\e,u_\e}-Y_t^{u_\e} |^2 \right)  + C\rho_{\sigma,\e } \int_0^T |\phi_\e (s) |^2 \d s \nonumber\\
&+ 2 \ee \left(  \sup_{t \in [0,T]}| Z_t^{\e,u_\e}-Y_t^{u_\e} |\int_0^T \kappa \left( | Z_s^{\e,u_\e}-Y_s^{u_\e} |^2  \right)^{\frac{1}{2}} | \phi_\e (s)| \d s \right) \nonumber\\
&+2 \ee \left( \sup_{t \in [0,T]}| Z_t^{\e,u_\e}-Y_t^{u_\e} |\int_0^T\left( \kappa \left(\ee \left( |X_s^\e -X_0^\e |^2\right)\right) \right)^{\frac{1}{2}} | \phi_\e (s)| \d s \right)\nonumber\\
\le & CT\rho_{\sigma,\e}  \ee \left( \sup_{t\in [0,T]} |Z_t^{\e,u_\e}-Y_t^{u_\e} |^2 \right) + m\rho_{\sigma,\e } \nonumber\\
&+\eta_7 \ee \left( \sup_{t\in [0,T]} |Z_t^{\e,u_\e}-Y_t^{u_\e} |^2 \right)+Cm \ee \int_0^T \kappa \left( \sup_{r \in[0,s]} | Z_r^{\e,u_\e}-Y_r^{u_\e} |^2 \right) \d s \nonumber \\
&+Cm \kappa \left( \ee \left( \sup_{t \in[0,T]}|X_t^\e -X_t^0 |^2 \right) \right).
\end{align}

For $M_5(t)$ and $M_8(t)$, by H\"{o}lder's inequality and Young's inequality, for any $\eta_8>0$,
\begin{align}
&\ee \left(\sup_{t \in [0,T]} |M_5(t)|+ \sup_{t \in [0,T]} |M_8(t)| \right) \nonumber\\
\le & C \e \ee \left( \int_0^T \int_{Z} | Z_{s-}^{\e,u_\e}-Y_{s-}^{u_\e}|^2 | G_\e (Z_{s-}^{\e,u_\e},\LL_{X_s^\e},z)|^2 N^{\e^{-1} \psi_\e } (\d z,\d s)  \right)^{\frac{1}{2}} \nonumber\\
&+ C\e \ee \left( \int_0^T \int_{Z} | G_\e (Z_{s-}^{\e,u_\e},\LL_{X_s^\e},z)|^2 |\psi_\e (s,z)| \nu(\d z) \d s \right)\nonumber\\
&\le \eta_8 \ee  \left( \sup_{t\in [0,T]} |Z_t^{\e,u_\e}-Y_t^{u_\e} |^2 \right)+C \e \ee \left( \int_0^T \int_{Z} | G_\e (Z_{s}^{\e,u_\e},\LL_{X_s^\e},z)|^2 |\psi_\e (s,z)| \nu(\d z) \d s \right),
\end{align}
where
\begin{align}
&C \e \ee \left( \int_0^T \int_{Z} | G_\e (Z_{s}^{\e,u_\e},\LL_{X_s^\e},z)|^2 |\psi_\e (s,z)| \nu(\d z) \d s \right) \nonumber\\
\le& C\e \ee \left( \int_0^T \int_{Z} | G_\e (Z_{s}^{\e,u_\e},\LL_{X_s^\e},z)-G (Z_{s}^{\e,u_\e},\LL_{X_s^\e},z)|^2 |\psi_\e (s,z)|\nu(\d z) \d s \right)\nonumber\\
&+C\e \ee \left( \int_0^T \int_{Z} | G (Z_{s}^{\e,u_\e},\LL_{X_s^\e},z)-G (Y_{s}^{u_\e},\LL_{X_s^0},z)|^2 |\psi_\e (s,z)| \nu(\d z) \d s \right)\nonumber\\
&+C\e \ee \left( \int_0^T \int_{Z} | G (Y_{s}^{u_\e},\LL_{X_s^0},z)|^2 |\psi_\e (s,z)| \nu(\d z) \d s \right)\nonumber\\
\le & C \e \rho_{G,\e}^2 T \Theta_m \nonumber\\
&+ C \e \ee \int_0^T \int_Z L_1^2 (z) \kappa\left( |Z_s^{\e,u_\e}-Y_s^{u_\e} |^2+W_2^2 (\LL_{X_s^\e},\LL_{X_s^0} )\right) |\psi_\e (s,z)|\nu(\d z) \d s \nonumber\\
&+C  \e \ee \left( \int_0^T \int_{Z} | G (Y_{s}^{u_\e},\LL_{X_s^0},z)|^2 |\psi_\e (s,z)|\nu(\d z) \d s \right) \nonumber\\
\le &  C  \e \rho_{G,\e}^2 T \Theta_m \nonumber\\
&+C  \e \Theta_m  \ee \left( \sup_{t \in[0,T]}|Z_t^{\e,u_\e}-Y_t^{u_\e} |^2 \right) + C \e \Theta_m  \ee \left( \sup_{t \in[0,T]}|X_t^\e -X_t^0 |^2 \right) \nonumber\\
&+C\e \Theta_m \ee \left( \sup_{u_\e \in S_1^m \times S_2^m}\sup_{t \in [0,T]} | Y_t^{u_\e} | + \sup_{t \in [0,T]}
|X_t^0| \right)+ C  \e \Theta_m \nonumber\\
\le &  C  \e \rho_{G,\e}^2 T \Theta_m +C  \e \Theta_m  \ee \left( \sup_{t \in[0,T]}|Z_t^{\e,u_\e}-Y_t^{u_\e} |^2 \right)\nonumber\\
& + C \e \Theta_m  \ee \left( \sup_{t \in[0,T]}|X_t^\e -X_t^0 |^2 \right)  +C  \e \Theta_m.
\end{align}
Here
$$
\Theta_{m}=\sup_{\psi \in S_2^m} \int_0^T \int_Z \left(  L_1^2(z)+L_2^2(z)+L_3^2(z)\right)\left(\psi(s,z)+1 \right) \nu(\d z) \d s <+\infty.
$$

For $M_6(t)$, by {\bf (H6)}, \eqref{L} and Young's inequality, for any $\eta_9>0$ we have
\begin{align}\label{M6}
&\ee \left(\sup_{t \in [0,T]} |M_6(t)| \right) \nonumber\\
\le & 2 \ee  \int_0^T \int_Z | Z_s^{\e,u_\e}-Y_s^{u_\e}| |  G_\e (Z_{s}^{\e,u_\e},\LL_{X_s^\e},z)-G (Z_{s}^{\e,u_\e},\LL_{X_s^\e},z) | |\psi_\e (s,z)-1|  \nu(\d z) \d s \nonumber\\
&+2 \ee  \int_0^T \int_Z | Z_s^{\e,u_\e}-Y_s^{u_\e}| |  G (Z_{s}^{\e,u_\e},\LL_{X_s^\e},z)- G (Y_{s}^{u_\e},\LL_{X_s^0},z)| |\psi_\e (s,z)-1|  \nu(\d z) \d s \nonumber\\
\le & 2 \ee \int_0^T \int_Z  \rho_{G,\e} L_3(z)| Z_s^{\e,u_\e}-Y_s^{u_\e}| |\psi_\e (s,z)-1|  \nu(\d z) \d s \nonumber\\
&+ 2 \ee \int_0^T \int_Z   L_1(z)| Z_s^{\e,u_\e}-Y_s^{u_\e}| \left( \kappa \left( | Z_s^{\e,u_\e}-Y_s^{u_\e}|^2\right)+\kappa \left(W_2^2(X_s^\e, X_s^0)\right) \right)^{\frac{1}{2}} |\psi_\e (s,z)-1|  \nu(\d z) \d s \nonumber\\
\le &  \eta_9\ee \int_0^T \int_Z | Z_s^{\e,u_\e}-Y_s^{u_\e}|^2 L_3(z) |\psi_\e (s,z)-1|  \nu(\d z) \d s \nonumber\\
&+C \int_0^T \int_Z  \rho_{G,\e}^2  L_3(z) |\psi_\e (s,z)-1|  \nu(\d z) \d s \nonumber\\
&+ \eta_9 \ee \int_0^T \int_Z | Z_s^{\e,u_\e}-Y_s^{u_\e}|^2 L_1(z) |\psi_\e (s,z)-1|  \nu(\d z) \d s\nonumber \\
&+C \ee \int_0^T \int_Z \left(\kappa \left( | Z_s^{\e,u_\e}-Y_s^{u_\e}|^2\right)+ \kappa \left(W_2^2(X_s^\e, X_s^0)\right) \right)L_1(z)|\psi_\e (s,z)-1|  \nu(\d z) \d s\nonumber \\
\le & \eta_9 \ee \left( \sup_{t \in [0,T]} | Z_t^{\e,u_\e}-Y_t^{u_\e} |^2 \right) \int_0^T \int_Z  L_3(z) |\psi_\e (s,z)-1|  \nu(\d z) \d s  \nonumber \\
&+C \rho_{G,\e}^2 \int_0^T \int_Z  L_3(z) |\psi_\e (s,z)-1|  \nu(\d z) \d s \nonumber\\
& +\eta_9 \ee \left( \sup_{t \in [0,T]} | Z_t^{\e,u_\e}-Y_t^{u_\e} |^2 \right) \int_0^T \int_Z  L_1(z) |\psi_\e (s,z)-1|  \nu(\d z) \d s  \nonumber \\
&+C \ee \int_0^T \int_Z \kappa \left( | Z_s^{\e,u_\e}-Y_s^{u_\e}|^2\right)L_1(z)|\psi_\e (s,z)-1|  \nu(\d z) \d s\nonumber \\
&+C  \int_0^T \int_Z \kappa \left( \ee \left( |X_s^\e -X_s^0|^2 \right) \right)L_1(z)|\psi_\e (s,z)-1|  \nu(\d z) \d s \nonumber \\
\le& C\eta_9 \ee \left( \sup_{t \in [0,T]} | Z_t^{\e,u_\e}-Y_t^{u_\e} |^2 \right)+C \rho_{G,\e}^2 +C\kappa\left( \ee \left( \sup_{t \in [0,T]}|X_t^\e -X_t^0|^2 \right) \right) \nonumber \\
&+C \int_0^T \int_Z
\ee \left( \kappa \left( \sup_{ r \in [0,s]} | Z_r^{\e,u_\e}-Y_r^{u_\e}|^2\right) \right)L_1(z)|\psi_\e (s,z)-1|  \nu(\d z) \d s.
\end{align}

Combining \eqref{M0}-\eqref{M6},
\begin{align}
&\left(1-\eta_5 T -\eta_6T-\eta_7-\eta_8-C\eta_9-(CaT+C \Theta_m) \e -CT \rho_{ \sigma, \e} \right) \ee  \left(\sup_{t \in [0,T]} | Z_t^{\e,u_\e}-Y_t^{u_\e}|^2 \right)\nonumber\\
\le & \ee \left(\int_0^T\left(C+\int_ZL_1(z) |\psi_\e(s,z)-1| \nu(\d z)\right) \kappa \left(\sup_{r \in [0,s]}| Z_r^{\e,u_\e}-Y_r^{u_\e}|^2\right) \d s\right)\nonumber\\
&+C \kappa\left( \ee \left( \sup_{t \in [0,T]}|X_t^\e -X_t^0|^2 \right) \right)+C\e \ee \left( \sup_{t \in [0,T]}|X_t^\e -X_t^0|^2 \right)\nonumber\\
&+C\left( \e +\rho_{b,\e}^2+ \e \rho_{\sigma,\e}^2+\rho_{\sigma,\e}+\e \rho_{G,\e}^2+\rho_{G,\e}^2 \right)\nonumber\\
=& : \int_0^T\left(C+\int_ZL_1(z) |\psi_\e(s,z)-1| \nu(\d z)\right) \ee \left( \kappa \left(\sup_{r \in [0,s]}| Z_r^{\e,u_\e}-Y_r^{u_\e}|^2\right) \right)\d s+O_3(\e).
\end{align}

Similar to the proof of Lemma \ref{LDPlem}, we can choose $\eta_5$-$\eta_9$ and $\e_0$ small enough such that for some constant $C_0 \ge \frac{1}{5} $ and any $\e \in (0,\e_0]$,
$$
1-\eta_5 T -\eta_6T-\eta_7-\eta_8-C\eta_9-(CaT+C \Theta_m) \e -CT \rho_{ \sigma, \e} \ge C_0 \ge \frac{1}{5}.
$$
Hence, we have
\begin{align}
&\frac{1}{5} \ee  \left(\sup_{t \in [0,T]} | Z_t^{\e,u_\e}-Y_t^{u_\e}|^2 \right) \nonumber\\
\le & \int_0^T\left(C+\int_ZL_1(z) |\psi_\e(s,z)-1| \nu(\d z)\right) \ee \left( \kappa \left(\sup_{r \in [0,s]}| Z_r^{\e,u_\e}-Y_r^{u_\e}|^2\right) \right)\d s+O_3(\e).
\end{align}
By {\bf (H5)}, {\bf (H6)} and Lemma \ref{LDPlem}, we have
\beq\label{O3}
\lim_{\e \to 0} O_3(\e)=0.
\nneq

Setting $f(t)=\int_1^t \frac{1}{\kappa(s)} \d s $, it follows from Lemma \ref{Bihari} that
$$
\ee  \left(\sup_{t \in [0,T]} | Z_t^{\e,u_\e}-Y_t^{u_\e}|^2 \right) \le f^{-1}\left(  f(O_3(\e))+ \int_0^T \left(C+\int_ZL_1(z) |\psi_\e(s,z)-1| \nu(\d z)\right) \d s\right).
$$
By \eqref{L1}, we have
$$
\int_0^T \left(C+\int_ZL_1(z) |\psi_\e(s,z)-1| \nu(\d z)\right) \d s < +\infty.
$$
Recalling the condition $ \int_{0+} \frac{1}{\kappa(s)} \d s=+\infty $, by \eqref{O3} we can conclude that
$$
f(O_3(\e))+ \int_0^T \left(C+\int_ZL_1(z) |\psi_\e(s,z)-1| \nu(\d z)\right) \d s \to -\infty \mbox{ as } \e \to 0.
$$
On the other hand, because $f$ is a strictly increasing function, then we obtain that $f$ has an inverse function which is strictly increasing, and $f^{-1}(-\infty)=0$. Thus,
$$
f^{-1}\left(  f(O_3(\e))+ \int_0^T \left(C+\int_ZL_1(z) |\psi_\e(s,z)-1| \nu(\d z)\right) \d s\right) \to 0 \mbox{ as  } \e \to 0.
$$
Hence, we get the desired result
$$
\lim_{\e \to 0} \ee \left( \sup_{t \in [0,T]} |Z_t^{\e,u_\e}-Y_t^{u_\e} |^2 \right)=0,
$$
which completes the proof.

\nprf

\subsection{Proof of MDP1 and MDP2}

In order to verify MDP1, we need the following Proposition.

\bprop
For any given $ m \in (0, +\infty)$, let $u_n=(\phi_n,\varphi_n)$, $ n \in \nn$,
$u=(\phi,\varphi) \in S_1^m \times B_2(m) $ be such that $u_n \to u  $ in $ S_1^m \times B_2(m)$ as $n \to +\infty$, then
$$
\lim_{n \to +\infty} \sup_{t \in [0,T]} |\Upsilon^0(u_n)(t) - \Upsilon^0(u)(t)|=0.
$$
\nprop

\bprf
Recall that
$$
V_t^u=\Upsilon^0(u)(t),\quad V_t^{u_n}=\Upsilon^0(u_n)(t)
$$
are the correspond solution to  \eqref{mdp1}. We only need to proof the following result
\beq
\lim_{n \to +\infty} \sup_{t \in [0,T]} |V_t^{u_n} - V_t^u|=0.
\nneq

The proof is similar to the proof of Proposition \ref{LDP1p}, by the It\^{o}'s formula we can get the result. So we omit the tedious proofs here.

%we just give a sketch here.
%Firstly, we show that $\{ V^{u_n}\}_{n \ge 1}$ is pre-compact in $C([0,T],\rr^d)$, \eqref{nulim} implies that
%$\{V^{u_n}\}_{n \ge 1}$ is uniformly bounded i.e.
%\beq\label{Cm}
%C_m:= \sup_{n\ge 1} \sup_{t\in [0,T]} |V_t^{u_n} | < +\infty.
%\nneq
%For any $s,t \in [0,T]$ with $s<t$, by \eqref{22},\eqref{23} and \eqref{Cm},
%\begin{align}
%&|V_t^{u_n} -V_s^{u_n} | \nonumber\\
%=&\left| \int_s^t b'(X_r^0, \LL_{X_r^0}) V_r^{u_n} \d r + \int_s^t \sigma (X_r^0, \LL_{X_r^0} ) \phi_n(r) \d r \right. \nonumber\\
%&\left. \int_s^t \int_Z G(X_r^0, \LL_{X_r^0},z) \psi_n(r,z) \nu(\d z) \d r-\d K_t^{u_n}  +\d K_s^{u_n}
%\right|\nonumber\\
%\le & C_m \int_s^t | b'(X_r^0, \LL_{X_r^0})|\d r+(2m)^{\frac{1}{2}} \left( \int_s^t \| \sigma (X_r^0, \LL_{X_r^0} )\|_{\rr^d \otimes \rr^d}^2 \d r \right)^{\frac{1}{2}}\nonumber\\
%&+ m \left( \int_s^t \inf_Z |G(X_r^0, \LL_{X_r^0},z) |^2 \nu (\d z) \d r  \right)^{\frac{1}{2}} +|K^{u_n}|_s^t <+\infty.
%\end{align}
%By \eqref{Gbound} and ({\bf C1}), we have
%$ \{ V^{u_n}\}_{n \ge 1 }$ is equi-continuous in $C([0,T],\rr^d)$. Hence $\{ V^{u_n}\}_{n \ge 1}$ is pre-compact in $C([0,T],\rr^d)$.
%
%Let $\tilde{V}$ be any limit of some subsequence of $\{V^{u_n}\}_{n \ge 1}$ in $C([0,T],\rr^d)$. Using the similar arguments as in the proof of Proposition \ref{LDP1p}, we can show
%$\tilde{V}=V^u$, which completes the proof.

\nprf

In order to verify (MDP2), we need the following three lemmas. The first one is taken
from Lemma 4.2, Lemma 4.3 and Lemma 4.7 in \cite{BDG16}.

\begin{lemma}\label{L123BOUND}
Fix $x\in (0,+\infty)$.

\noindent (a) \quad There exists $\zeta_m \in (0,+\infty)$ such that for all $I\in \BB([0,T])$ and $\epsilon \in (0,+\infty),$
\begin{align}\label{LL21}
\sup_{\psi\in S_{+,\epsilon}^m}\int_{Z\times I}\big(L_1^2(y)+L_2^2(y)+L_3^2(y)\big)\psi(y,s)\nu({\rm d}y){\rm d}s \leq \zeta_m(a^2(\epsilon)+Leb_T(I)).
\end{align}

\noindent (b) \quad There exists ${\rm \Gamma}_m, \rho_m :  (0,+\infty)\rightarrow  (0,+\infty)$ such that ${\rm \Gamma}_m(s)\downarrow 0$ as $s\uparrow +\infty$, and for all $I\in \BB([0,T])$ and $\epsilon,\beta \in (0,+\infty),$
\begin{align}\label{LL22}
&\sup_{\varphi\in S_{\epsilon}^m}\int_{Z\times I}\big(L_1(z)+L_2(z)+L_3(z)\big)|\varphi(y,s)|1_{\{|\varphi|\geq \beta/a(\epsilon)\}}(y,s)\nu({\rm d}y){\rm d}s \nonumber\\
&\leq {\rm \Gamma}_m(\beta)(1+\sqrt{Leb_T(I)}),
\end{align}

and

\begin{align}\label{LL23}
&\sup_{\varphi\in S_{\epsilon}^m}\int_{Z\times I}\big(L_1(z)+L_2(z)+L_3(z)\big)|\varphi(y,s)|\nu({\rm d}y){\rm d}s \nonumber\\
&\leq \rho_m(\beta) \sqrt{Leb_T(I)}+{\rm \Gamma}_m(\beta)a(\epsilon).
\end{align}

\noindent (c) \quad For any $\beta>0$,
\begin{align}\label{LL24}
\lim_{\epsilon\rightarrow 0}\sup_{\varphi\in S_{\epsilon}^m}\int_{Z\times [0,T]}\big(L_1(z)+L_2(z)+L_3(z)\big)|\varphi(y,s)|1_{\{|\varphi|\geq \beta/a(\epsilon)\}}(y,s)\nu({\rm d}y){\rm d}s
=0.
\end{align}
\end{lemma}

\blem\label{lipbound}
Under {\bf (H1)}, {\bf (H2)'},{\bf (H3)'}, {\bf (H5)} and {\bf (H6)'}, there exists some constant $\e_1>0$ and a positive constant $C_T$ independent of $\e$ such that for any $\e \in (0,\e_1]$,
\beq\label{Lipbound}
\ee  \left(\sup_{t \in [0,T]} | X_t^\e-X_t^0|^2 \right)\le C_{T} \left( \e + \rho_{b,\e}^2+\e \rho_{\sigma,\e}^2+\e \rho_{G,\e}^2  \right).
\nneq
\nlem

\bprf
By It\^{o}'s formula, using the similar proof to Lemma \ref{LDPlem},
we have
\begin{align}
&\ee  \left(\sup_{t \in [0,T]} | X_t^\e-X_t^0|^2 \right)\nonumber\\
\le &  C\int_0^T \ee \left(  \sup_{r \in [0,s]}| X_r^\e-X_r^0|^2  \right) \d s +C\left( \rho_{b,\e}^2 T +\e \rho_{\sigma,\e}^2 T+ \e \rho_{G,\e}^2+ \e \right).
\end{align}
By Gronwall's inequality, there exists some constant $\e_1>0$ and $C_T>0$ such that for any $\e \in (0,\e_1]$,
\beq
\ee  \left(\sup_{t \in [0,T]} | X_t^\e-X_t^0|^2 \right)\le C_{T} \left( \e + \rho_{b,\e}^2+\e \rho_{\sigma,\e}^2+\e \rho_{G,\e}^2  \right).
\nneq

\nprf

\blem\label{Mbound}
Let $M^{\e, u_{\e}}$ be the solution to \eqref{mdp2m}. Then there exists some $\kappa_0>0$ such that
\beq
\sup_{\e \in (0,\kappa_0]} \ee \left( \sup_{t \in [0,t]}|M^{\e,u_\e}_t|^2 \right) <+\infty.
\nneq
\nlem

\bprf
By It\^{o} formula, we have for any $t \in [0,T]$,
\begin{align}\label{MI}
|M_t^{\e, u_\e}|^2=&\frac{2}{\lambda (\e)} \int_0^t \langle b_\e(\lambda(\e)M^{\e,u_{\e}}_s+X_s^0,\mathcal{L}_{X_s^\e})-b(X_s^0,\mathcal{L}_{X_s^0}), M_s^{\e, u_\e}\rangle \d s \nonumber\\
&\frac{2 \sqrt{\e}}{\lambda(\e)}\int_0^t \langle  M_s^{\e, u_\e},\sigma_{\e}(\lambda(\e)M^{\e,u_{\e}}_s+X^0_s,\mathcal{L}_{X_s^\e}) \d W_s \rangle \nonumber\\
&+2 \int_0^t \langle \sigma_{\e}(\lambda(\e)M^{\e,u_{\e}}_s+X^0_s,\mathcal{L}_{X_s^\e})\phi_{\e}(s), M_s^{\e, u_\e} \rangle \d s \nonumber\\
& -2 \int_0^t \langle M_s^{\e, u_\e}, \mathrm{d} \hat{K}^{\e,u_\e}_s\rangle\nonumber\\
&+\frac{2\e}{\lambda(\e)} \int_0^t \int_Z \langle G_{\e}(\lambda(\e)M^{\e,u_{\e}}_{s-}+X^0_{s-},\mathcal{L}_{X_s^\e},z), M_s^{\e, u_\e} \rangle \tilde{N}^{\e^{-1}\psi_\e} (\d z, \d s)\nonumber\\
&+\frac{2}{\lambda(\e)} \int_0^t \int_Z \langle  G_{\e}(\lambda(\e)M^{\e,u_{\e}}_{s}+X^0_{s},\mathcal{L}_{X_s^\e},z )(\psi_\e(s,z)-1), M_s^{\e, u_\e} \rangle \nu( \d z) \d s \nonumber\\
&+\frac{\e}{\lambda^2(\e)}\int_0^t \|\sigma_{\e}(\lambda(\e)M^{\e,u_{\e}}_s+X^0_s,\mathcal{L}_{X_s^\e}) \|_{\rr^d \otimes \rr^d}^2 \d s \nonumber\\
&+\frac{\e^2}{\lambda^2(\e)} \int_0^t \int_Z | G_{\e}(\lambda(\e)M^{\e,u_{\e}}_{s-}+X^0_{s-},\mathcal{L}_{X_s^\e},z) |^2 N^{\e^{-1}\psi_\e} (\d z, \d s) \nonumber\\
=&:I_1(t)+I_2(t)+I_3(t)+I_4(t)+I_5(t)+I_6(t)+I_7(t)+I_8(t).
\end{align}

By equation \eqref{Lambda}, {\bf (H5)}, {\bf (H6)'} and {\bf (C2)} imply that there exists some constant $\e_2 >0$ such that
\beq\label{e2}
\frac{\e }{\lambda^2(\e)}\vee \ \lambda(\e) \vee \rho_{b,\e}\vee\rho_{\sigma,\e} \vee \rho_{G,\e} \vee \frac{\rho_{b,\e}}{\lambda(\e)} \in (0,\frac{1}{2}], \quad \forall \e \in (0, \e_2].
\nneq

Now we set
\beq\label{e3}
\e_3 = \e_1 \wedge \e_2 \wedge \frac{1}{2},
\nneq
where $\e_1$ is the same in Lemma \ref{lipbound}.

In the following proof, denote by $C$ a generic constant which may be change from line to line and is independent of $\e$.

By {\bf (H2)'}, {\bf (H5)}, \eqref{Lipbound} and Young's inequality, for any $\e \in (0,\e_3]$,
\begin{align}\label{MI1}
I_1(t)=&\frac{2}{\lambda (\e)} \int_0^t \langle b_\e(\lambda(\e)M^{\e,u_{\e}}_s+X_s^0,\mathcal{L}_{X_s^\e})-b(X_s^0,\mathcal{L}_{X_s^0}), M_s^{\e, u_\e}\rangle \d s \nonumber\\
=&\frac{2}{\lambda (\e)} \int_0^t \langle b_\e(\lambda(\e)M^{\e,u_{\e}}_s+X_s^0,\mathcal{L}_{X_s^\e})-b(\lambda(\e)M^{\e,u_{\e}}_s+X_s^0,\mathcal{L}_{X_s^\e}), M_s^{\e, u_\e}\rangle \d s \nonumber\\
&+\frac{2}{\lambda (\e)} \int_0^t \langle b(\lambda(\e)M^{\e,u_{\e}}_s+X_s^0,\mathcal{L}_{X_s^\e})-b(\lambda(\e)M^{\e,u_{\e}}_s+X_s^0,\mathcal{L}_{X_s^0}), M_s^{\e, u_\e}\rangle \d s \nonumber\\
&+\frac{2}{\lambda (\e)} \int_0^t \langle b(\lambda(\e)M^{\e,u_{\e}}_s+X_s^0,\mathcal{L}_{X_s^0})-b(X_s^0,\mathcal{L}_{X_s^0}), M_s^{\e, u_\e}\rangle \d s \nonumber\\
\le& \frac{2 \rho_{b,\e}}{ \lambda(\e)} \int_0^t |M_s^{\e, u_\e}| \d s + \frac{2}{ \lambda(\e)} L \int_0^t  W_2(\mathcal{L}_{X_s^\e},\mathcal{L}_{X_s^0}) |M_s^{\e, u_\e} | \d s \nonumber\\
&+2L \int_0^t |M_s^{\e, u_\e} |^2 \d s \nonumber\\
\le& \frac{CL( \rho_{b,\e}+\sqrt{\e})}{ \lambda(\e)} \int_0^t |M_s^{\e, u_\e}| \d s +2L \int_0^t |M_s^{\e, u_\e} |^2 \d s\nonumber\\
\le&  C  \int_0^t |M_s^{\e, u_\e} |^2 \d s+C.
\end{align}

By {\bf (H2)'}, {\bf (H5)}, \eqref{Lipbound}, \eqref{Gbound}, H\"{o}lder's inequality and Young's inequality, for any $\e \in (0,\e_3]$,
\begin{align}\label{MI3}
I_3(t)=& 2 \int_0^t \langle \sigma_{\e}(\lambda(\e)M^{\e,u_{\e}}_s+X^0_s,\mathcal{L}_{X_s^\e})\phi_{\e}(s), M_s^{\e, u_\e} \rangle \d s \nonumber\\
=& 2 \int_0^t \left\langle   \left( \sigma_{\e}(\lambda(\e)M^{\e,u_{\e}}_s+X^0_s,\mathcal{L}_{X_s^\e})-
\sigma(\lambda(\e)M^{\e,u_{\e}}_s+X^0_s,\mathcal{L}_{X_s^\e}) \right)\phi_{\e}(s), M_s^{\e, u_\e}   \right\rangle \d s \nonumber\\
&+2 \int_0^t \left\langle   \left( \sigma(\lambda(\e)M^{\e,u_{\e}}_s+X^0_s,\mathcal{L}_{X_s^\e})-
\sigma(X^0_s,\mathcal{L}_{X_s^0}) \right)\phi_{\e}(s), M_s^{\e, u_\e}   \right\rangle \d s \nonumber\\
&+2 \int_0^t \left\langle \sigma(X^0_s,\mathcal{L}_{X_s^0})\phi_{\e}(s), M_s^{\e, u_\e}   \right\rangle \d s \nonumber\\
\le & 2 \rho_{\sigma,\e} \int_0^t |M_s^{\e, u_\e}| |\phi_{\e}(s) | \d s \nonumber \\
&+2 L \int_0^t \left(\lambda(\e) |M_s^{\e, u_\e}| + W_2(\LL_{X_s^\e}, \LL_{X_s^0})  \right) |M_s^{\e, u_\e}||\phi_{\e}(s) | \d s \nonumber \\
&+2 \int_0^t \| \sigma(X_s^0, \LL_{X_s^0})\|_{\rr^d \otimes \rr^d} |M_s^{\e, u_\e}||\phi_{\e}(s) | \d s \nonumber \\
\le & 2L\lambda(\e) \int_0^t |M_s^{\e, u_\e}|^2  |\phi_{\e}(s) | \d s \nonumber \\
&+ \left( 2 \rho_{\sigma,\e} + 2L \ee \left( \sup_{s \in[0,T]} |X_s^\e -X_s^0|^2\right)^{\frac{1}{2}}\right)
\int_0^t |M_s^{\e, u_\e}||\phi_{\e}(s) | \d s \nonumber \\
&+2 \int_0^t \| \sigma(X_s^0, \LL_{X_s^0})\|_{\rr^d \otimes \rr^d} |M_s^{\e, u_\e}||\phi_{\e}(s) | \d s \nonumber \\
\le& C \left[ \int_0^t  |M_s^{\e, u_\e}|^2  \d s  + \int_0^t |M_s^{\e, u_\e}|^2 |\phi_{\e}(s) |^2 \d s \right] \nonumber\\
&+C \left[ \int_0^t |M_s^{\e, u_\e}|^2 \d s+ \int_0^s  |\phi_{\e}(s) |^2 \d s \right] \nonumber \\
&+ C \left[  \int_0^t \|  \sigma(X_s^0, \LL_{X_s^0})  \|_{\rr^d \otimes \rr^d}^2 \d s + \int_0^t |M_s^{\e, u_\e}|^2 |\phi_{\e}(s) |^2 \d s \right] \nonumber \\
\le & C \int_0^t |M_s^{\e, u_\e}|^2 \left( |\phi_{\e}(s) |^2+1 \right) \d s + C\int_0^t \|  \sigma(X_s^0, \LL_{X_s^0})  \|_{\rr^d \otimes \rr^d}^2 \d s + C \int_0^t |\phi_{\e}(s) |^2 \d s \nonumber \\
\le & C \int_0^t |M_s^{\e, u_\e}|^2 \left( |\phi_{\e}(s) |^2+1 \right) \d s +C.
\end{align}
The last inequality holds by \eqref{Gbound} and $\phi_\e \in \SS_1^m$.

For $I_4(t)$, recall the definition of $M_t^{\e,u_\e}$, since $A$ is monotone, by Lemma \ref{lem24} and H\"{o}lder's inequality, we have
\begin{align}\label{MI4}
\sup_{t \in [0,T]} I_4(t) =&\sup_{t \in [0,T]} \left(-2 \int_0^t \langle M_s^{\e, u_\e}, \d \hat{K}_s^{\e,u_\e}   \rangle \right)\nonumber\\
\le&\sup_{t \in [0,T]}\left(-|\hat{K}^{\e,u_\e}|_0^t \right) + C\int_0^T | M_s^{\e, u_\e} | \d s +C  \nonumber \\
\le& C \int_0^T \sup_{r\in[0,s]} |M_r^{\e, u_\e}|^2 \d s +C.
\end{align}

For $I_7(t)$, we have
\begin{align}\label{MI7}
I_7(t)=&\frac{\e}{\lambda^2(\e)}\int_0^t \|\sigma_{\e}(\lambda(\e)M^{\e,u_{\e}}_s+X^0_s,\mathcal{L}_{X_s^\e}) \|_{\rr^d \otimes \rr^d}^2 \d s \nonumber\\
\le& \frac{\e}{\lambda^2(\e)} \int_0^t \| \sigma_{\e}(\lambda(\e)M^{\e,u_{\e}}_s+X^0_s,\mathcal{L}_{X_s^\e})-
\sigma (\lambda(\e)M^{\e,u_{\e}}_s+X^0_s,\mathcal{L}_{X_s^\e})  \|_{\rr^d \otimes \rr^d}^2 \d s \nonumber \\
&+\frac{\e}{\lambda^2(\e)} \int_0^t \| \sigma(\lambda(\e)M^{\e,u_{\e}}_s+X^0_s,\mathcal{L}_{X_s^\e})-
\sigma (X^0_s,\mathcal{L}_{X_0^\e})  \|_{\rr^d \otimes \rr^d}^2 \d s \nonumber \\
&+\frac{\e}{\lambda^2(\e)} \int_0^t \| \sigma (X^0_s,\mathcal{L}_{X_0^\e})  \|_{\rr^d \otimes \rr^d}^2 \d s \nonumber \\
\le& \frac{T \e}{\lambda^2(\e)} \rho_{\sigma,\e }^2 +\frac{L \e}{ \lambda^2(\e)} \int_0^t \left( \lambda^2(\e)|M^{\e,u_{\e}}_s|^2 +W_2^2(\LL_{X_s^\e},\LL_{X_s^0}) \right) \d s \nonumber \\
&+\frac{\e}{ \lambda^2(\e)} \int_0^t \| \sigma (X^0_s,\mathcal{L}_{X_0^\e})  \|_{\rr^d \otimes \rr^d}^2 \d s \nonumber \\
\le & \frac{T \e }{ \lambda^2(\e)} \rho_{\sigma,\e}^2 +C \e \int_0^t |M^{\e,u_{\e}}_s|^2 \d s \nonumber \\
&+\frac{C\e }{ \lambda^2 (\e)} \int_0^t \ee |X_s^\e -X_s^0|^2 \d s + \frac{\e}{\lambda^2(\e)}
\int_0^t \| \sigma (X^0_s,\mathcal{L}_{X_0^\e})  \|_{\rr^d \otimes \rr^d}^2 \d s \nonumber \\
\le& C \int_0^t  |M^{\e,u_{\e}}_s|^2 \d s +C.
\end{align}

Recall $\varphi_\e =(\ \psi_\e -1 )/ \lambda(\e)$, by {\bf (H6)'}, \eqref{Lipbound}, for any $\e \in (0,\e_3)$,
\begin{align}\label{MI6}
I_6(t)=&2 \int_0^t \int_Z \left\langle  G_{\e}(\lambda(\e)M^{\e,u_{\e}}_{s}+X^0_{s},\mathcal{L}_{X_s^\e},z ) \frac{(\psi_\e(s,z)-1)}{\lambda(\e)}, M_s^{\e, u_\e} \right\rangle \nu( \d z) \d s \nonumber\\
=&2 \int_0^t \int_Z \left\langle  \left[ G_{\e}(\lambda(\e)M^{\e,u_{\e}}_{s}+X^0_{s},\mathcal{L}_{X_s^\e},z )
-G(\lambda(\e)M^{\e,u_{\e}}_{s}+X^0_{s},\mathcal{L}_{X_s^\e},z )\right] \right. \nonumber\\
& \left. \times \varphi_\e(s,z), M_s^{\e, u_\e} \right\rangle \nu( \d z) \d s \nonumber\\
&+2 \int_0^t \int_Z \left\langle  \left[ G_{\e}(\lambda(\e)M^{\e,u_{\e}}_{s}+X^0_{s},\mathcal{L}_{X_s^\e},z )
-G(0,\delta_0,z )\right] \varphi_\e(s,z), M_s^{\e, u_\e} \right\rangle \nu( \d z) \d s \nonumber\\
&+2 \int_0^t \int_Z \left\langle  G(0,\delta_0,z ) \varphi_\e(s,z), M_s^{\e, u_\e} \right\rangle \nu( \d z) \d s \nonumber\\
\le & 2 \rho_{G,\e} \int_0^t \int_Z  L_3(z)  |\varphi_\e(s,z)|| M_s^{\e, u_\e}| \nu( \d z) \d s \nonumber\\
&+2 \int_0^t \int_Z L_1(z) \left( |\lambda(\e)   M_s^{\e, u_\e}+X_s^0|+W_2(\LL_{X_s^\e},\delta_0)\right)
|\varphi_\e(s,z)|| M_s^{\e, u_\e}| \nu( \d z) \d s \nonumber \\
&+2 \int_0^t \int_Z L_2(z) |\varphi_\e(s,z)|| M_s^{\e, u_\e}| \nu( \d z) \d s \nonumber \\
\le& C \int_0^t \int_Z \left(  |\lambda(\e)   |M_s^{\e, u_\e}|+|X_s^0|+W_2(\LL_{X_s^\e}, \delta_0 )  \right) |\varphi_\e(s,z)|| M_s^{\e, u_\e}| \nu( \d z) \d s \nonumber \\
&+C \int_0^t \int_Z \left(L_2(z)+L_3(z) \right) |\varphi_\e(s,z)|| M_s^{\e, u_\e}| \nu( \d z) \d s \nonumber \\
\le& C \int_0^t \int_Z \left( L_1(z)+L_2(z)+L_3(z)  \right) |\varphi_\e(s,z)|| M_s^{\e, u_\e}|^2 \nu( \d z) \d s \nonumber \\
&+ C \int_0^t \int_Z \left( L_1(z)+L_2(z)+L_3(z)  \right) |\varphi_\e(s,z)|\nu( \d z) \d s.
\end{align}

To deduce the last inequality, the following facts have been used
\benu
\item
$X^0 \in C([0,T], \rr^d)$;
\item
$W_2(\LL_{X_s^\e}, \delta_0 ) \le W_2(\LL_{X_s^\e}, \LL_{X_s^0})+|X_s^0|$.
\nenu

Set
\begin{align}
D_\e := \int_0^T \left( |\phi_\e(s)|^2+1  \right) \d s + \int_0^T \int_Z (L_1(z)+L_2(z)+ L_3(z)) |\varphi_\e(s,z)| \nu(\d z) \d s.
\end{align}

By substituting \eqref{MI1}-\eqref{MI6} back into \eqref{MI} and applying Gronwall's inequality, we obtain
\beq\label{MQ1}
|M_t^{\e,u_\e}|^2 \le e^{CD_\e} \left\{ C D_\e + \sup_{s\in [0,T]} |I_2(s)+I_5(s)+I_8(s)| \right\}
\nneq
for all $\e \in (0,\e_3],\ t\in[0,T]$.

Since $(\phi_\e,\varphi_\e) \in \SS_1^m \times \SS_\e^m$ $P$-a.s., we have
\beq\label{MQ2}
\frac{1}{2} \int_0^T |\phi_\e (s)|^2 \le m, \ P\mbox{-a.s. } \forall \e \in (0,\e_3].
\nneq

Hence by \eqref{LL23}, \eqref{MQ1} and \eqref{MQ2}, there exists some constant $C_0>0$ such that for each $\e \in (0, \e_3]$,
\beq\label{MMM}
\ee \left( \sup_{t \in [0,T]} |M_t^{\e,u_\e}|^2 \right) \le
C_0\left\{  1+ \ee \left( \sup_{t \in [0,T]} I_2(t) \right) +\ee \left( \sup_{t \in [0,T]} I_5(t) \right) +\ee \left( \sup_{t \in [0,T]} I_8(t) \right)   \right\}.
\nneq

By Burkholder-Davis-Gundy's inequality, {\bf (H2)', (H5)}, Young's inequality, \eqref{Lipbound}, \eqref{Gbound} and \eqref{MI6}, we have
\begin{align}\label{MI2}
\ee \left( \sup_{t \in [0,T]} I_2(t) \right) \le& \frac{C \sqrt{\e}}{ \lambda(\e )} \ee \left[ \int_0^T  |M_s^{\e,u_\e}|^2 \| \sigma_\e ( \lambda (\e )  M_s^{\e,u_\e}+X_s^0, \LL_{X_s^\e} )\|_{\rr^d \otimes \rr^d}^2 \d s
\right]^{\frac{1}{2}}\nonumber \\
\le & \frac{C T \sqrt{\e}}{ \lambda (\e)} \ee \left( \sup_{s \in [0,T]} | M_s^{\e,u_\e}|^2 \right)+
\frac{C \sqrt{\e}}{ \lambda(\e) } \ee \left(\int_0^T \| \sigma_\e ( \lambda (\e )  M_s^{\e,u_\e}+X_s^0, \LL_{X_s^\e} )\|_{\rr^d \otimes \rr^d}^2 \d s  \right) \nonumber \\
\le & \frac{C T \sqrt{\e}}{ \lambda (\e)} \ee \left( \sup_{s \in [0,T]} | M_s^{\e,u_\e}|^2 \right)
+ \frac{C \sqrt{\e} \rho_{\sigma,\e}^2}{ \lambda (\e)} + C \sqrt{\e} \lambda(\e) \ee \int_0^T |M_s^{\e,u_\e} |^2 \d s
\nonumber \\
&+ \frac{C \sqrt{\e}}{ \lambda(\e) } \left( \int_0^T \left[ W_2^2(\LL_{X_s^\e}, \LL_{X_s^0})+\| \sigma(X_s^0, \LL_{X_s^0})\|_{\rr^d \otimes \rr^d}^2 \right]  \right) \nonumber \\
\le & C \left( \frac{\sqrt{\e }}{ \lambda(\e)}+\sqrt{\e} \lambda(\e) \right) \ee \left(  \sup_{s \in [0,T]} |M_s^{\e,u_\e}|^2  \right) +C.
\end{align}

Similarly, by {\bf (H6)'}, for any $\e \in (0,\e_3]$, we have
\begin{align}\label{MI58}
&C_0 \left( \ee \left( \sup_{t \in [0,T]} I_5(t) \right) +\ee \left( \sup_{t \in [0,T]} I_8(t) \right) \right) \nonumber\\
\le& \ee \left( \frac{ 2 \e }{\lambda(\e)} \int_0^T \int_Z \langle G_\e(\lambda(\e)M_{s-}^{\e,u_\e}+X_{s-}^0, \LL_{X_s^\e},z), M_s^{\e,u_\e} \rangle \tilde{N}^{\e^{-1} \psi_\e} (\d z,\d s)\right) \nonumber \\
&+\ee \left( \frac{\e^2}{ \lambda^2(\e)}  \int_0^T \int_Z | G_\e(\lambda(\e)M_{s-}^{\e,u_\e}+X_{s-}^0, \LL_{X_s^\e},z)  |^2 N^{\e^{-1} \psi_\e} (\d z,\d s) \right) \nonumber \\
\le & \frac{C \e}{ \lambda (\e) } \ee \left(  \int_0^T \int_Z | G_\e(\lambda(\e)M_{s-}^{\e,u_\e}+X_{s-}^0, \LL_{X_s^\e},z)  |^2 |M_s^{\e, u_\e}|^2 N^{\e^{-1} \psi_\e} (\d z, \d s)  \right)^{\frac{1}{2}} \nonumber \\
&+ \frac{\e}{ \lambda^2(\e)} \ee \left( \int_0^T \int_Z   | G_\e(\lambda(\e)M_{s-}^{\e,u_\e}+X_{s-}^0, \LL_{X_s^\e},z)  |^2   \psi_\e(s,z) \nu(\d z) \d s \right) \nonumber \\
\le & \frac{1}{10} \ee \left( \sup_{s \in [0,T]} |M_s^{\e, u_\e} |^2 \right)\nonumber \\
&+\frac{C \e }{ \lambda^2(\e)} \ee \left( \int_0^T \int_Z | G_\e(\lambda(\e)M_{s-}^{\e,u_\e}+X_{s-}^0, \LL_{X_s^\e},z)  -G(\lambda(\e)M_{s-}^{\e,u_\e}+X_{s-}^0, \LL_{X_s^\e},z)|^2 \psi_\e(s,z) \nu(\d z) \d s \right)\nonumber \\
&+\frac{C \e }{ \lambda^2(\e)} \ee \left( \int_0^T \int_Z |G(\lambda(\e)M_{s-}^{\e,u_\e}+X_{s-}^0, \LL_{X_s^\e},z)-
G(0, \delta_0, z)|^2 \psi_\e(s,z) \nu(\d z) \d s \right)\nonumber \\
&+\frac{C \e }{ \lambda^2(\e)} \ee \left( \int_0^T \int_Z |G(0, \delta_0, z)|^2 \psi_\e(s,z) \nu(\d z) \d s \right)\nonumber \\
\le & \frac{1}{10} \ee \left( \sup_{s \in [0,T]} |M_s^{\e, u_\e} |^2 \right)\nonumber \\
&+\frac{C \e }{ \lambda^2(\e)} \ee \left(  \int_0^T \int_Z \left( \rho_{G,\e}^2 L_3^2(z) + L_1^2(z) \left( \lambda(\e) |M_s^{\e,u_\e}| +|X_s^0| +W_2(\LL_{X_s^\e,\delta_0}) \right)^2\right)\psi_\e(s,z) \nu(\d z) \d s  \right)\nonumber \\
&+\frac{C \e }{ \lambda^2(\e)} \ee\int_0^T \int_Z L_2^2(z) \psi_\e(s,z) \nu(\d z) \d s \nonumber\\
\le& \frac{1}{10} \ee \left( \sup_{s \in [0,T]} |M_s^{\e, u_\e} |^2 \right)\nonumber \\
&+C \e \sup_{\psi \in S_{+,\e}^K} \int_0^T \int_Z L_1^2(z) \psi(s,z) \nu(\d z ) \d s \ee \left(\sup_{s \in [0,T] }|M_s^{\e,u_\e}|^2 \right)\nonumber \\
&+C \sup_{\psi \in S_{+,\e}^K} \int_0^T \int_Z (L_1^2(z)+L_2^2(z)+L_3^2(z)) \psi(s,z) \nu(\d z ) \d s \nonumber \\
& \times \left(  1+ \sup_{s \in [0,T] }|X_s^0|^2 + \ee \left(  \sup_{s\in [0,T]} |X_s^\e - X_s^0|^2\right)\right) \nonumber \\
\le & \frac{1}{10} \ee \left( \sup_{s \in [0,T]} |M_s^{\e, u_\e} |^2 \right)\nonumber \\
&+C \e \sup_{\psi \in S_{+,\e}^K} \int_0^T \int_Z L_1^2(z) \psi(s,z) \nu(\d z ) \d s \ee \left(\sup_{s \in [0,T] }|M_s^{\e,u_\e}|^2 \right)\nonumber \\
&+C \sup_{\psi \in S_{+,\e}^K} \int_0^T \int_Z (L_1^2(z)+L_2^2(z)+L_3^2(z)) \psi(s,z) \nu(\d z ) \d s.
\end{align}

Combining \eqref{MMM}-\eqref{MI58}, we have
\begin{align}
&\left(\frac{9}{10} -C\frac{\sqrt{\e}}{ \lambda{\e}} + C \sqrt{\e} \lambda{\e} -C \e \sup_{\psi \in S_{+,\e}^K} \int_0^T \int_Z L_1^2(z) \psi(s,z) \nu(\d z ) \d s \right) \nonumber \\
\times &  \ee \left(\sup_{s \in [0,T] }|M_s^{\e,u_\e}|^2 \right) \nonumber \\
\le & C \left(  1+ \sup_{\psi \in S_{+,\e}^K} \int_0^T \int_Z (L_1^2(z)+L_2^2(z)+L_3^2(z)) \psi(s,z) \nu(\d z ) \d s\right).
\end{align}

Hence, by \eqref{LL21} and \eqref{Lambda}, there exists some constant $\kappa_0>0$ such that for any $\e \in (0,\kappa_0]$,
$$
\left(\frac{9}{10} -C\frac{\sqrt{\e}}{ \lambda{\e}} + C \sqrt{\e} \lambda{\e} -C \e \sup_{\psi \in S_{+,\e}^K} \int_0^T \int_Z L_1^2(z) \psi(s,z) \nu(\d z ) \d s \right)\ge \frac{1}{5}>0.
$$
Hence, we have
$$
\sup_{\e \in (0,\kappa_0]} \ee \left(  \sup_{s\in[0,T]} |M_s^{\e,u_{\e}}|^2\right) < + \infty.
$$

\nprf

Finally, the verification of ({\bf MDP2}) is given in the next proposition.
Recall $\tilde{u_\e}$ in \eqref{md2}.

\bprop
For any $\varpi>0$,
\beq\label{MDP22}
\lim_{\e\to 0} P \left( \sup_{t \in [0,T]} |M_t^{\e, u_\e} -V_t^{\tilde{u_\e}} |> \varpi \right)=0.
\nneq
\nprop

\bprf
For each fixed $\e >0$ and $j \in \nn$, define a stopping time
$$
\tau_\e^j=\inf \{ t \ge 0: |M_t^{\e, u_\e} | \ge j\} \wedge T.
$$

By Lemma \ref{Mbound}, we have
$$
P(\tau_\e^j <T)\le \frac{\ee \left( \sup_{t \in [0,T]} |M_t^{\e, u_\e}|^2\right)}{j^2} \le \frac{C}{j^2}, \ \forall \e \in (0,\kappa_0],
$$
where $\kappa_0$ is the same as in Lemma \ref{Mbound}.

Let $Q_s^\e=M_s^{\e, u_\e}-V_s^{\tilde{u}_\e}$ for each $s \in [0,T]$. Notice that the corresponding equations $M_s^{\e, u_\e}$ and $ V_s^{\tilde{u}_\e}$ satisfied are distribution independent SDEs. By It\^{o}'s formula, we have
\begin{align}\label{MQ}
| Q_{t \wedge \tau_\e^{j}}^\e|^2=&-2\int_0^{t \wedge \tau_\e^{j}} \left\langle  Q_s^\e, \mathrm{d} \hat{K}^{\e,u_\e}_s- \mathrm{d} \hat{K}^{\tilde{u}_\e}_s  \right\rangle \nonumber\\
&+2 \int_0^{t \wedge \tau_\e^{j}} \left\langle \frac{1}{\lambda(\e)}\left(b_\e(\lambda(\e)M^{\e,u_{\e}}_s+X_s^0,\mathcal{L}_{X_s^\e})-b(X_s^0,\mathcal{L}_{X_s^0})\right)- \nabla b(X^0_s,\mathcal{L}_{X^0_s})V^{\tilde{u}_\e}_s, Q_s^\e \right\rangle \d s \nonumber\\
&+2 \frac{\sqrt{\e}}{ \lambda(\e)} \int_0^{t \wedge \tau_\e^{j}} \left\langle Q_s^\e,  \sigma_{\e}(\lambda(\e)M^{\e,u_{\e}}_s+X^0_s,\mathcal{L}_{X_s^\e})   \d W_s \right\rangle \nonumber\\
&+2 \int_0^{t \wedge \tau_\e^{j}} \left\langle \left( \sigma_{\e}(\lambda(\e)M^{\e,u_{\e}}_s+X^0_s,\mathcal{L}_{X_s^\e}) - \sigma(X_s^0,\mathcal{L}_{X_s^0})\right) \phi_\e(s) ,Q_s^\e\right\rangle \d s \nonumber\\
&+\frac{2 \e}{ \lambda(\e)} \int_0^{t \wedge \tau_\e^{j}} \int_Z \left\langle  G_{\e}(\lambda(\e)M^{\e,u_{\e}}_{s-}+X^0_{s-},\mathcal{L}_{X_s^\e},z )  ,Q_s^\e   \right\rangle \tilde{N}^{\e^{-1}\psi_\e} (\d z, \d s) \nonumber\\
&+2 \int_0^{t \wedge \tau_\e^{j}} \int_Z \left\langle  G_{\e}(\lambda(\e)M^{\e,u_{\e}}_{s}+X^0_{s},\mathcal{L}_{X_s^\e},z )\varphi_\e(s,z)\right.\nonumber \\
&\quad\left.-G(X_s^0,\LL_{X_s^0},z) \varphi_\e(s,z) 1_{\{|\varphi_\e| \le \frac{\beta}{\lambda(\e)}\}}(s,z) ,Q_s^\e   \right\rangle \nu(\d z) \d s  \nonumber\\
&+\frac{\e}{ \lambda^2(\e)} \int_0^{t \wedge \tau_\e^{j}} \| \sigma_{\e}(\lambda(\e)M^{\e,u_{\e}}_s+X^0_s,\mathcal{L}_{X_s^\e}) \|_{\rr^d \otimes \rr^d}^2 \d s \nonumber \\
&+ \frac{\e^2}{ \lambda^2(\e)} \int_0^{t \wedge \tau_\e^{j}} \int_Z | G_{\e}(\lambda(\e)M^{\e,u_{\e}}_{s}+X^0_{s},\mathcal{L}_{X_s^\e},z ) |^2  N^{\e^{-1}\psi_\e} (\d z, \d s) \nonumber \\
=&:J_1(t)+J_2(t)+J_3(t)+J_4(t)+J_5(t)+J_6(t)+J_7(t)+J_8(t).
\end{align}

Due to \eqref{nulim} and the fact that
$
\tilde{u}_\e \in S_1^m \times B_2(\sqrt{m \kappa_2 (1)}),
$
there exists some $\Omega^0 \in \FF$ with $P (\Omega^0)=1$ such that
\beq\label{Vbound}
\kappa:= \sup_{\e \in (0,\kappa_0]} \sup_{ \omega \in \Omega^0, t \in [0,T]} | V_t^{\tilde{u}_\e} (\omega)| < +\infty.
\nneq

Recall the constant $\e_3 $ appearing in \eqref{e3}. Set
$
\e_4 = \e_3 \wedge \kappa_0.
$

By Definition \ref{solution}, we have
\beq\label{MJ1}
J_1(t)=-2\int_0^{t \wedge \tau_\e^{j}} \left\langle  Q_s^\e, \mathrm{d} \hat{K}^{\e,u_\e}_s- \mathrm{d} \hat{K}^{\tilde{u}_\e}_s  \right\rangle \le 0.
\nneq

For $J_2(t)$, we have
\begin{align}\label{MJ2}
J_2(t)=&2 \int_0^{t \wedge \tau_\e^{j}} \left\langle \frac{1}{\lambda(\e)}\left(b_\e(\lambda(\e)M^{\e,u_{\e}}_s+X_s^0,\mathcal{L}_{X_s^\e})-b(X_s^0,\mathcal{L}_{X_s^0})\right)- \nabla b(X^0_s,\mathcal{L}_{X^0_s})V^{\tilde{u}_\e}_s, Q_s^\e \right\rangle \d s \nonumber\\
=&2 \int_0^{t \wedge \tau_\e^{j}} \left\langle \frac{1}{\lambda(\e)}\left(b_\e(\lambda(\e)M^{\e,u_{\e}}_s+X_s^0,\mathcal{L}_{X_s^\e})-b(\lambda(\e)M^{\e,u_{\e}}_s+X_s^0,\mathcal{L}_{X_s^\e})\right), Q_s^\e \right\rangle \d s \nonumber\\
&+2 \int_0^{t \wedge \tau_\e^{j}} \left\langle \frac{1}{\lambda(\e)}\left(b(\lambda(\e)M^{\e,u_{\e}}_s+X_s^0,\mathcal{L}_{X_s^\e})-b(\lambda(\e)M^{\e,u_{\e}}_s+X_s^0,\mathcal{L}_{X_s^0})\right), Q_s^\e \right\rangle \d s \nonumber\\
&+2 \int_0^{t \wedge \tau_\e^{j}} \left\langle \frac{1}{\lambda(\e)}\left(b(\lambda(\e)M^{\e,u_{\e}}_s+X_s^0,\mathcal{L}_{X_s^0})-b(X_s^0,\mathcal{L}_{X_s^0})\right)-
\nabla b(X^0_s,\mathcal{L}_{X^0_s})M_s^{\e, u_\e}, Q_s^\e \right\rangle \d s \nonumber\\
&+2 \int_0^{t \wedge \tau_\e^{j}} \left\langle \nabla b(X^0_s,\mathcal{L}_{X^0_s})\left(M_s^{\e, u_\e}-V^{\tilde{u}_\e}_s\right), Q_s^\e \right\rangle \d s \nonumber\\
=&: J_{2,1}(t)+J_{2,2}(t)+J_{2,3}(t)+J_{2,4}(t).
\end{align}

For $J_{2,3}(t)$, by the mean value theorem and {\bf (C0)} and {\bf (C1)}, for any $\e \in (0,\e_4]$, there exists
$
\theta_s^\e \in [0,1]
$
such that
\begin{align}\label{MJ23}
J_{2,3}(t) =& 2 \int_0^{t \wedge \tau_\e^{j}} \left\langle \frac{b(\lambda(\e)M^{\e,u_{\e}}_s+X_s^0,\mathcal{L}_{X_s^0})-b(X_s^0,\mathcal{L}_{X_s^0})}{\lambda(\e)}-
\nabla b(X^0_s,\mathcal{L}_{X^0_s})M_s^{\e, u_\e}, Q_s^\e \right\rangle \d s \nonumber \\
\le& 2 L' \int_0^{t \wedge \tau_\e^{j}} \| \nabla b( \lambda(\e)M^{\e,u_{\e}}_s \theta_s^\e + X_s^0,\mathcal{L}_{X_s^0} ) - \nabla b(X^0_s,\mathcal{L}_{X^0_s})\|_{\rr^d \otimes \rr^d } |M^{\e,u_{\e}}_s | Q_s^\e | \d s \nonumber \\
\le & 2 L'\int_0^{t \wedge \tau_\e^{j}} \left( 1+| \lambda(\e)M^{\e,u_{\e}}_s \theta_s^\e + X_s^0 |^{q'}+|X_s^0|^{q'}  \right) | \lambda(\e)M^{\e,u_{\e}}_s \theta_s^\e | |M^{\e,u_{\e}}_s | |Q_s^\e | \d s \nonumber \\
\le & C_j \lambda(\e),
\end{align}
where
$$
C_j=2L' \left(1 + |j+ \sup_{s\in [0,T]}|X_s^0| |^{q'} +\sup_{s \in [0,T]} |X_s^0|^{q'} \right) j^2 (j+\kappa)T,
$$
which is independent of $\e $.

In the following proof, $C_j$ will denote generic constants which are independent of $\e$, may be different from line to line.

\begin{align}\label{MJ2124}
J_{2,1}(t)+J_{2,2}(t)+J_{2,4}(t)\le& 2 \frac{\rho_{b,\e}}{ \lambda(\e )} \int_0^{t \wedge \tau_\e^{j}} |Q_s^\e| \d s \nonumber \\
&+2L \int_0^{t \wedge \tau_\e^{j}} \frac{\left(\ee |X_s^\e -X_s^0|^2 \right)^\frac{1}{2}}{\lambda(\e)} |Q_s^\e | \d s \nonumber\\
&+2 \int_0^{t }\| \nabla b(X^0_{s \wedge \tau_\e^{j}},\mathcal{L}_{X^0_{s \wedge \tau_\e^{j}}})\|_{\rr^d \otimes \rr^d}|Q_{s \wedge \tau_\e^{j}}^\e |^2 \d s.
\end{align}

Hence
\begin{align}\label{MJ2Z}
J_2(t) \le& C_j\left( \frac{\rho_{b,\e}}{ \lambda(\e)} + \frac{ \left(\e + \rho_{b,\e}^2+ \e \rho_{\sigma,\e}^2+\e \rho_{G,\e}^2\right)^{\frac{1}{2}}}{\lambda(\e)} + \lambda(\e)\right)\nonumber \\
&+2 \int_0^{t }\| \nabla b(X^0_{s \wedge \tau_\e^{j}},\mathcal{L}_{X^0_{s \wedge \tau_\e^{j}}})\|_{\rr^d \otimes \rr^d}|Q_{s \wedge \tau_\e^{j}}^\e |^2 \d s.
\end{align}

Inserting the inequalities \eqref{MJ2Z} into \eqref{MQ}, and using Gronwall's inequality, we deduce that for any $\e \in (0,\e_4]$
\begin{align}
&\sup_{e\in [0,T]} | Q_{s \wedge \tau_\e^{j}}^\e |^2\nonumber \\
\le& \exp \{ 2 \int_0^T \|\nabla b(X_s^0, \LL_{X_s^0})\|_{\rr^d \otimes \rr^d} \d s \} \nonumber \\
&\times \left\{ C_j \left( \frac{\rho_{b,\e}}{ \lambda(\e)} + \frac{ \left(\e + \rho_{b,\e}^2+ \e \rho_{\sigma,\e}^2+\e \rho_{G,\e}^2\right)^{\frac{1}{2}}}{\lambda(\e)} + \lambda(\e)\right)+\sum_{i=3}^{8} \sup_{s \in [0,T] } |J_i(s)| \right\}.
\end{align}

Set $C_0=\exp \{ 2 \int_0^T \|\nabla b(X_s^0, \LL_{X_s^0})\|_{\rr^d \otimes \rr^d} \d s \}$.

Since
\begin{align}
|J_3(t)| \le& \frac{2 \sqrt{\e}}{ \lambda(\e)} \left( \int_0^{t \wedge \tau_\e^{j}} \rho_{\sigma, \e} |Q_s^\e| +L W_2(\LL_{X_s^\e}, \LL_{X_s^0}) |Q_s^\e| \right. \nonumber\\
&\left.
+L \lambda(\e) |Q_s^\e|^2 + L \lambda(\e) |V_\e^{\tilde{u}_\e}| |Q_s^\e|+ \int_0^{t \wedge \tau_\e^{j}} \| \sigma(X_s^0,\LL_{X_s^0})\|_{\rr^d \otimes \rr^d} | Q_s^\e | \d s \right),
\end{align}
we have
\begin{align}\label{MJ37}
&C_0 \left( \ee \left( \sup_{t \in [0,T]} |J_3(t)|\right) +\ee \left( \sup_{t \in [0,T]} |J_7(t)|\right)\right) \nonumber \\
\le& \frac{1}{10} \ee \left( \sup_{t \in [0,T]} |Q_{t \wedge \tau_\e^j}^\e|^2\right) \nonumber \\
&+ \frac{C \e}{ \lambda^2(\e)} \ee \left( \int_0^{t \wedge \tau_\e^{j}} \| \sigma_\e \left(\lambda(\e)M_s^{\e,u_\e}+X_s^\e \right)   \|_{\rr^d \otimes \rr^d}^2 \right) \nonumber \\
\le &  \frac{1}{10} \ee \left( \sup_{t \in [0,T]} |Q_{t \wedge \tau_\e^j}^\e|^2\right)+ \frac{C \e \rho_{\sigma,\e}^2}{ \lambda(\e)} \nonumber \\
&+ \frac{C \e }{\lambda^2(\e)} \ee \left(  \int_0^{t \wedge \tau_\e^{j}} \lambda^2(\e) |M_s^{\e,u_\e} |^2 +\ee \left( |X_s^\e- X_s^0|^2\right) \d s   \right) + \frac{C \e }{\lambda^2(\e)} \int_0^T \| \sigma(X_s^0,\LL_{X_s^0}) \|_{\rr^d \otimes \rr^d}^2 \d s \nonumber \\
\le & \frac{1}{10} \ee \left( \sup_{t \in [0,T]} |Q_{t \wedge \tau_\e^j}^\e|^2\right)+ C_j \frac{\e}{\lambda^2(\e)} \left( 1+ \rho_{\sigma,\e}^2+\e+\rho_{b,\e}^2+\e \rho_{\sigma,\e}^2+\e \rho_{G,\e}^2 \right).
\end{align}

Since
\begin{align}
J_4(t)=&2 \int_0^{t \wedge \tau_\e^{j}} \left\langle \left( \sigma_{\e}(\lambda(\e)M^{\e,u_{\e}}_s+X^0_s,\mathcal{L}_{X_s^\e}) - \sigma(X_s^0,\mathcal{L}_{X_s^0})\right) \phi_\e(s) ,Q_s^\e\right\rangle \d s \nonumber\\
=& 2 \int_0^{t \wedge \tau_\e^{j}} \left\langle \left( \sigma_{\e}(\lambda(\e)M^{\e,u_{\e}}_s+X^0_s,\mathcal{L}_{X_s^\e}) - \sigma(\lambda(\e)M^{\e,u_{\e}}_s+X^0_s,\mathcal{L}_{X_s^\e}) \right) \phi_\e(s) ,Q_s^\e\right\rangle \d s \nonumber\\
&+2 \int_0^{t \wedge \tau_\e^{j}} \left\langle \left( \sigma(\lambda(\e)M^{\e,u_{\e}}_s+X^0_s,\mathcal{L}_{X_s^\e}) - \sigma(X_s^0,\mathcal{L}_{X_s^0})\right) \phi_\e(s) ,Q_s^\e\right\rangle \d s,
\end{align}
we have
\begin{align}\label{MJ4}
C_0 \ee \left( \sup_{t \in [0,T]}  J_4(t)\right) \le& C \rho_{\sigma,\e} \ee \int_0^{t \wedge \tau_\e^{j}} |\phi_\e(s)| | Q_s^\e| \d s \nonumber\\
&+C \ee \int_0^{t \wedge \tau_\e^{j}} \left( \lambda(\e) | M_s^{\e,u_\e}|+ W_2(\LL_{X_s^\e},\LL_{X_s^0}) \right)|\phi_\e(s)| |Q_s^\e| \d s \nonumber \\
\le & \left( C \rho_{\sigma,\e} j^2+ \lambda(\e) + \left( \e +\rho_{b,\e}^2+\e \rho_{\sigma,\e}^2 +\e \rho_{G,\e}^2\right)^{\frac{1}{2}} \right) \ee \left( \int_0^T |\phi_\e(s)|^2 \d s\right)^{\frac{1}{2}} \nonumber \\
\le & C_j\left( \rho_{\sigma,\e}+ \lambda(\e) + \left( \e +\rho_{b,\e}^2+\e \rho_{\sigma,\e}^2 +\e \rho_{G,\e}^2\right)^{\frac{1}{2}}\right).
\end{align}

By Burkholder-Davis-Gundy's inequality, \eqref{LL21} and \eqref{e2}, using the similar proof of \eqref{MI58}, we have
 for any $ \e \in (0, \e_4]$
\begin{align} \label{MJ58}
&C_0 \left( \ee \left( \sup_{t\in [0,T] } |J_5(t)|\right)+\ee \left( \sup_{t\in [0,T] } |J_8(t)|\right)\right) \nonumber \\
\le & \frac{1}{10} \ee \left( \sup_{t \in [0,T]} | Q^\e_{t \wedge \tau_\e^j}|^2\right)+ \frac{C \e }{ \lambda^2 (\e)} \ee \left( \int_0^{T \wedge \tau_\e^j} \int_Z |G_\e \left( \lambda(\e)M_s^{\e, u_\e} + X_s^0, \LL_{X_s^\e},z \right)  |^2 \psi_\e(s,z) \nu(\d z) \d s \right) \nonumber \\
\le & \frac{1}{10} \ee \left( \sup_{t \in [0,T]} | Q^\e_{t \wedge \tau_\e^j}|^2\right) \nonumber \\
&+ \frac{C \e \rho_{G,\e}^2 }{ \lambda^2(\e) } \ee \left( \int_0^{T \wedge \tau_\e^j
} \int_Z  L_3^2(z) \psi_\e(s,z) \nu(\d z) \d s \right) \nonumber \\
&+\frac{C \e }{ \lambda^2 (\e)} \ee \left( \int_0^{T \wedge \tau_\e^j } \int_Z |G(0,\delta_0,z) |^2 \psi_\e (s,z) \nu(\d z) \d s  \right) \nonumber \\
&+\frac{C \e}{ \lambda^2(\e)} \ee \left( \int_0^{T \wedge \tau_\e^j} \int_Z \left( | \lambda(\e) M_s^{\e,u_\e} + X_s^0  |^2 + \ee (|X_s^\e|^2 )  \right) L_1^2(z) \psi_\e ( s, z) \nu(\d z) \d s \right) \nonumber \\
\le & \frac{1}{10} \ee \left( \sup_{t \in [0,T]} | Q^\e_{t \wedge \tau_\e^j}|^2\right) \nonumber \\
& + \frac{C_j \e }{ \lambda^2(\e)} \sup_{\psi \in S_{+,\e}^m} \left( \int_0^T \int_Z \left( L_1^2(z)+L_2^2(z)+L_3^2(z) \right)\psi(s,z) \nu(\d z) \d s \right) \nonumber \\
\le & \frac{1}{10} \ee \left( \sup_{t \in [0,T]} | Q^\e_{t \wedge \tau_\e^j}|^2\right)+ \frac{C_j \e}{ \lambda^2(\e)}.
\end{align}

Note that
\begin{align}
J_6(t) =& 2 \int_0^{t \wedge \tau_\e^{j}} \int_Z \left\langle  G_{\e}(\lambda(\e)M^{\e,u_{\e}}_{s}+X^0_{s},\mathcal{L}_{X_s^\e},z )\varphi_\e(s,z)\right.\nonumber \\
&\quad\left.-G(X_s^0,\LL_{X_s^0},z) \varphi_\e(s,z) 1_{\{|\varphi_\e| \le \frac{\beta}{\lambda(\e)}\}}(s,z) ,Q_s^\e   \right\rangle \nu(\d z) \d s  \nonumber\\
=& 2 \int_0^{t \wedge \tau_\e^{j}} \int_Z \left\langle  G_{\e}(\lambda(\e)M^{\e,u_{\e}}_{s}+X^0_{s},\mathcal{L}_{X_s^\e},z )\varphi_\e(s,z)\right.\nonumber \\
&\quad\left.-G(\lambda(\e)M^{\e,u_{\e}}_{s}+X^0_{s},\mathcal{L}_{X_s^\e},z ) \varphi_\e(s,z) ,Q_s^\e   \right\rangle \nu(\d z) \d s  \nonumber\\
&+2 \int_0^{t \wedge \tau_\e^{j}} \int_Z \left\langle  G (\lambda(\e)M^{\e,u_{\e}}_{s}+X^0_{s},\mathcal{L}_{X_s^\e},z )\varphi_\e(s,z)\right.\nonumber \\
&\quad\left. - G(X_s^0,\LL_{X_s^0},z) \varphi_\e(s,z), Q_s^\e   \right\rangle \nu(\d z) \d s  \nonumber\\
&+ 2 \int_0^{t \wedge \tau_\e^{j}} \int_Z \left\langle  \left( G(X_s^0,\LL_{X_s^0},z)-G(0,\delta_0,z)  \right) \varphi_\e(s,z) 1_{\{|\varphi_\e| > \frac{\beta}{\lambda(\e)}\}}(s,z) ,Q_s^\e   \right\rangle \nu(\d z) \d s \nonumber \\
&+ 2 \int_0^{t \wedge \tau_\e^{j}} \int_Z \left\langle  G(0,\delta_0,z) \varphi_\e(s,z) 1_{\{|\varphi_\e| > \frac{\beta}{\lambda(\e)}\}}(s,z) ,Q_s^\e   \right\rangle \nu(\d z) \d s.
\end{align}

Hence by {\bf (H6)'} and \eqref{Vbound}, we have
\begin{align}\label{MJ6}
&C_0 \ee \left( \sup_{t \in [0,T]} | J_6(t)| \right) \nonumber \\
\le & C \rho_{G, \e} \ee \left( \int_0^{t \wedge \tau_\e^{j}} \int_Z L_3(z) |  \varphi_\e(s,z) | |Q_s^\e | \nu(\d z) \d s  \right) \nonumber \\
& + C \ee \left( \int_0^{t \wedge \tau_\e^{j}} \int_Z L_1(z) \left( \lambda(\e)|M_s^{\e,u_\e}| +W_2(\LL_{X_s^\e}, \LL_{X_s^0}) \right) | \varphi_\e (s,z)| |Q_s^\e | \nu(\d z)  \d s   \right) \nonumber \\
&+ C \ee \left(  \int_0^{t \wedge \tau_\e^{j}} \int_Z L_1(z) |X_s^0| | Q_s^\e | | \varphi_\e(s,z)| 1_{\{ |\varphi_\e|>\frac{\beta}{\lambda(\e)} \}} (s,z) \nu(\d z) \d s   \right) \nonumber \\
&+C \ee \left(  \int_0^{t \wedge \tau_\e^{j}} \int_Z  L_2(z) |Q_s^\e| | \varphi_\e(s,z)| 1_{\{ |\varphi_\e|>\frac{\beta}{\lambda(\e)} \} } \nu (\d z) \d s \right) \nonumber \\
\le & C_j \left( \rho_{G,\e} +\lambda(\e) + \left(  \e + \rho_{b,\e}^2+ \e \rho_{\sigma,\e}^2+ \e \rho_{G,\e}^2   \right)^{\frac{1}{2}} \right) \sup_{\varphi \in S_\e^m} \int_0^T \int_Z (L_1(z)+L_3(z)) |\varphi(s,z)| \nu(\d z) \d s \nonumber \\
&+C_j \sup_{\varphi \in S_\e^m} \int_0^T \int_Z (L_1(z)+L_2(z)) |\varphi_\e (s,z) | 1_{\{ |\varphi_\e|>\frac{\beta}{\lambda(\e)} \} } (s,z) \nu(\d z) \d s \nonumber \\
\le & C_j \left( \rho_{G,\e} +\lambda(\e) + \left(  \e + \rho_{b,\e}^2+ \e \rho_{\sigma,\e}^2+ \e \rho_{G,\e}^2   \right)^{\frac{1}{2}} \right) \nonumber\\
&+C_j \sup_{\varphi \in S_\e^m} \int_0^T \int_Z (L_1(z)+L_2(z)) |\varphi_\e (s,z) | 1_{\{ |\varphi_\e|>\frac{\beta}{\lambda(\e)} \} } (s,z) \nu(\d z) \d s.
\end{align}

Combining \eqref{e2} and \eqref{MQ}-\eqref{MJ6} together, we obtain that for any $\e \in (0,\e_4]$,
\begin{align}
&\frac{8}{10} \ee \left( \sup_{t \in [0,T]} |Q_{t \wedge \tau_\e^j}^\e |^2 \right) \nonumber \\
\le & C_j \left\{ \rho_{G,\e} + \rho_{\sigma,\e} + \frac{\rho_{b,\e}}{ \lambda(\e)}+ \frac{\left( \e +\rho_{b,\e}^2 +\e \rho_{\sigma,\e}^2+ \e \rho_{G,\e}^2 \right)^\frac{1}{2}}{\lambda(\e)} \right. \nonumber \\
& \left. +\frac{\e}{ \lambda^2(\e) } +\lambda(\e)+  \left( \e +\rho_{b,\e}^2 +\e \rho_{\sigma,\e}^2+ \e \rho_{G,\e}^2 \right)^\frac{1}{2} \right. \nonumber \\
& \left. +\sup_{\varphi \in S_\e^m} \int_0^T \int_Z  (L_1(z)+L_2(z)) |\varphi_\e (s,z) | 1_{\{ |\varphi_\e|>\frac{\beta}{\lambda(\e)} \} } (s,z) \nu(\d z) \d s   \right\}.
\end{align}

By {\bf (C2)}, \eqref{Lambda} and \eqref{LL24}, it follows that
\beq
\lim_{\e \to 0} \ee \left(  \sup_{t \in [0,T]}  | M_{t \wedge \tau_\e^j}^{\e,u_\e}-V_{t \wedge \tau_\e^j}^{\tilde{u}_\e} |^2 \right)=\lim_{\e \to 0} \ee \left( \sup_{t \in [0,T]} |Q_{t \wedge \tau_\e^j}^{\e} |^2 \right)=0.
\nneq

Now for any $ \varpi >0,\ \e \in (0,\e_4]$ and $j \in \nn$, we have
\begin{align}
&P\left( \sup_{t \in [0,T]} |M_t^{\e, u_\e} -V_t^{\tilde{u_\e}} |> \varpi \right) \nonumber \\
\le & P\left( \left( \sup_{t \in [0,T]} | M_{t \wedge \tau_\e^j}^{\e,u_\e}-V_{t \wedge \tau_\e^j}^{\tilde{u}_\e}|> \varpi \right) \bigcap \left( \tau_\e^j \ge T \right) \right) +P \left(\tau_\e^j < T  \right) \nonumber\\
\le & \frac{1}{\varpi^2} \ee \left( \sup_{t \in [0,T]} | M_{t \wedge \tau_\e^j}^{\e,u_\e}-V_{t \wedge \tau_\e^j}^{\tilde{u}_\e}|^2 \right)+\frac{C}{j^2}.
\end{align}

Letting $\e \to 0$ first and then $j \to +\infty$, we get
\beq
\lim_{\e \to 0} P\left( \sup_{t \in [0,T]} |M_t^{\e, u_\e} -V_t^{\tilde{u_\e}} |> \varpi \right)=0,
\nneq
which is the desired result.

%
%
%
%		\begin{equation}\label{mdp2m}
%			\left\{\begin{aligned}
%				\mathrm{d}M^{\e,u_\e}_t=&\frac{1}{\lambda(\e)}\left(b_\e(\lambda(\e)M^{\e,u_{\e}}_t+X_t^0,\mathcal{L}_{X_t^\e})-b(X_t^0,\mathcal{L}_{X_t^0})\right)\mathrm{d}t  \\
%				&+\frac{\sqrt{\e}}{\lambda(\e)}\sigma_{\e}(\lambda(\e)M^{\e,u_{\e}}_t+X^0_t,\mathcal{L}_{X_t^\e})\mathrm{d}W_t  \\
%				&+\sigma_{\e}(\lambda(\e)M^{\e,u_{\e}}_t+X^0_t,\mathcal{L}_{X_t^\e})\phi_{\e}(t)\mathrm{d}t-\mathrm{d} \hat{K}^{\e,u_\e}_t,  \\
%&+\frac{\e }{ \lambda(\e)} \int_Z G_{\e}(\lambda(\e)M^{\e,u_{\e}}_{t-}+X^0_{t-},\mathcal{L}_{X_t^\e},z ) \tilde{N}^{\e^{-1}\psi_\e} (\d z, \d t) \\
%&+ \frac{1}{\lambda(\e)} \int_Z G_{\e}(\lambda(\e)M^{\e,u_{\e}}_{t}+X^0_{t},\mathcal{L}_{X_t^\e},z )(\psi_\e(t,z)-1) \nu( \d z) \d t,  \\				
%\,\,\,\, M^{\e,u_{\e}}_0=&\,\,0.
%			\end{aligned}\right.
%		\end{equation}
%		
%
%
%
%
%\begin{equation}\label{mdp1}
%\left\{ \begin{aligned}
%				&\mathrm{d}V^{u}_t=b'(X^0_t,\mathcal{L}_{X^0_t})V^{u}_t\mathrm{d}t+\sigma(X_t^0,\mathcal{L}_{X_t^0})\phi(t)\mathrm{d}t\\
%&\quad \quad\quad+\int_Z G(X_t^0, \LL_{X^0_t},z) \psi(t,z)\nu(\d z) \d t-\d \hat{K}^u_t     \\
%				&V^{u}_0=0.
%\end{aligned}  \right.
%\end{equation}

\nprf

\vskip0.3cm

\noindent{\bf Acknowledgements:} Cheng L. is supported by Natural Science Foundation of China (12001272 and 12171173). Liu W. is supported by Natural Science Foundation of China (12471143 and 12131019).

\end{document}